\documentclass[]{amsart}       
  
\usepackage[margin=1in]{geometry}
\usepackage{
array, 
amsmath, 
amssymb,
} 
\usepackage{%
enumerate, 
color}
\usepackage{colortbl}
\usepackage[]{subfig}
\usepackage[parfill]{parskip}   

\usepackage[]{lineno}

\usepackage{setspace}

\usepackage[]{floatrow} 
\usepackage{tikz}
\usetikzlibrary{calc}
\usetikzlibrary{backgrounds}
\usetikzlibrary{shapes}
\usetikzlibrary{positioning}
\usetikzlibrary{decorations, decorations.pathmorphing, decorations.text}
\usetikzlibrary{snakes}

\DeclareGraphicsExtensions{.pdf}

\usepackage{scrextend}

\usepackage[colorlinks=true,citecolor=red,linkcolor=blue,urlcolor=magenta]{hyperref}
\usepackage{cleveref}

\usepackage{colortbl}
\definecolor{hl}{gray}{.85}
\newcolumntype{a}{>{\columncolor{hl}}c}
\usepackage{slashbox}
\newcommand{\Flow}[1]{\ensuremath{\text{Flow}(#1)}}

\usepackage[normalem]{ulem}
\usepackage{cancel}

\theoremstyle{theorem}
\newtheorem{theorem}[]{Theorem}[section]
\newtheorem{prop}[theorem]{Proposition}
\newtheorem{lemma}[theorem]{Lemma}
\newtheorem{corollary}[theorem]{Corollary}

\newtheorem{observation}[theorem]{Observation}

\theoremstyle{definition}
\newtheorem{definition}[]{Definition}
\newtheorem{algorithm}[theorem]{Algorithm}

\newtheorem{question}{Question}
\newtheorem{conjecture}{Conjecture}

\newtheorem{remark}{Remark}

\newcommand{\be}{\begin{enumerate}}
\newcommand{\ee}{\end{enumerate}}

\newcommand{\bi}{\begin{itemize}}
\newcommand{\ei}{\end{itemize}}

\newcommand{\mc}{\mathcal}

\tikzstyle{vtx}=[inner sep=1pt,draw, shape=circle, minimum size=5pt]
\tikzstyle{line}=[inner sep=3pt,draw, shape=circle, minimum size=6pt, fill=black!20]
\tikzset{>=stealth}

\tikzset{arclabel/.style={midway, fill=white, inner sep = 2 pt}}

\definecolor{mycyan}{RGB}{0, 255, 255}
\definecolor{mymagenta}{RGB}{255,0, 255}

\newcommand{\drawPentaBlobs}[4][1]{
\pgfmathtruncatemacro{\result}{#2-1}
\pgfmathsetmacro{\angle}{360/#2}
\foreach \i in {0, ..., \result}{
	\node[rotate = \i*\angle,  circle, gray!70!white, minimum width = #1 cm, inner sep = 1pt] (G\i) at ($(\angle*\i:2)$){};
	\node[vtx, fill = orange] (C\i) at (G\i.180) {};
	\node[vtx, fill = mymagenta] (A\i) at (G\i.180-72) {};
	\node[vtx, fill = gray!50!white] (v\i) at (G\i.180-2*72) {};
	\node[vtx, fill = yellow] (A'\i) at (G\i.180+72) {};
	\node[vtx, fill = gray!50!white] (w\i) at (G\i.180+2*72) {};
	\node[vtx, fill = cyan] (B\i) at ($(v\i) + (\angle*\i:#4)$) {};
	\node[vtx, fill = green!50!black] (B'\i) at ($(w\i) + (\angle*\i:#4)$) {};

	\node[vtx, fill=black] (z\i) at (\angle*\i:#3){};
	\draw[gray!70!white] (C\i) -- (v\i) -- (A'\i) -- (A\i) -- (w\i)--(C\i);
	\draw[gray!70!white] (v\i) -- (B\i) -- (B'\i) -- (w\i);
}
}

\newcommand{\drawPentaBlobOne}[6][1]{
\pgfmathsetmacro{\angle}{360/#2}
	\node[rotate = #5*\angle,  circle, gray!70!white,  minimum width = #1 cm, inner sep = 1pt] (G#5) at ($(\angle*#5:2)$){};
	\node[vtx, fill = orange] (C#5) at (G#5.180) {};
	\node[vtx, fill = mymagenta] (A#5) at (G#5.180-72) {};
	\node[vtx, fill = gray!50!white] (v#5) at (G#5.180-2*72) {};
	\node[vtx, fill = yellow] (A'#5) at (G#5.180+72) {};
	\node[vtx, fill = gray!50!white] (w#5) at (G#5.180+2*72) {};
	\node[vtx, fill = cyan] (B#5) at ($(v#5) + (\angle*#5:#4)$) {};
	\node[vtx, fill = green!50!black] (B'#5) at ($(w#5) + (\angle*#5:#4)$) {};

	\node[vtx, fill=black] (z#5) at (\angle*#5:#3){};
	\draw[gray!70!white, #6] (C#5) -- (v#5) -- (A'#5) -- (A#5) -- (w#5)--(C#5);
	\draw[gray!70!white, #6] (v#5) -- (B#5) -- (B'#5) -- (w#5);
}

\newcommand{\drawPentaBlobsNoBlack}[4][1]{
\pgfmathtruncatemacro{\result}{#2-1}
\pgfmathsetmacro{\angle}{360/#2}
\foreach \i in {0, ..., \result}{
	\node[rotate = \i*\angle,  circle, gray!70!white, minimum width = #1 cm, inner sep = 1pt] (G\i) at ($(\angle*\i:2)$){};
	\node[vtx, fill = orange] (C\i) at (G\i.180) {};
	\node[vtx, fill = mymagenta] (A\i) at (G\i.180-72) {};
	\node[vtx, fill = gray!50!white] (v\i) at (G\i.180-2*72) {};
	\node[vtx, fill = yellow] (A'\i) at (G\i.180+72) {};
	\node[vtx, fill = gray!50!white] (w\i) at (G\i.180+2*72) {};
	\node[vtx, fill = cyan] (B\i) at ($(v\i) + (\angle*\i:#4)$) {};
	\node[vtx, fill = green!50!black] (B'\i) at ($(w\i) + (\angle*\i:#4)$) {};

	\draw[gray!70!white] (C\i) -- (v\i) -- (A'\i) -- (A\i) -- (w\i)--(C\i);
	\draw[gray!70!white] (v\i) -- (B\i) -- (B'\i) -- (w\i);
}
}

\newcommand{\drawBlobs}[4][1]{
\pgfmathtruncatemacro{\result}{#2-1}
\pgfmathsetmacro{\angle}{360/#2}
\foreach \i in {0, ..., \result}{
	\node[rotate = \i*\angle, draw,  cloud, gray!70!white, minimum width = #1 cm, minimum height=#4 cm, inner sep = 1pt] (G\i) at ($(\angle*\i:2)$){};
	\node[vtx, fill = orange] (C\i) at (G\i.180) {};
	\node[vtx, fill = mymagenta] (A\i) at (G\i.180-45) {};
	\node[vtx, fill = yellow] (A'\i) at (G\i.180+45) {};
	\node[vtx, fill = cyan] (B\i) at (G\i.90-45) {};
	\node[vtx, fill =green!70!black] (B'\i) at (G\i.-90+45) {};

	\node[vtx, fill=black] (z\i) at (\angle*\i:#3){};
}
}

\newcommand{\drawAlphaA}[6][1]{ 
\pgfmathtruncatemacro{\result}{int(mod(#3+#4, #2))}
\pgfmathsetmacro{\angle}{360/#2}
\draw[#5] let \n1 = \result in (A#3) --
(A'\n1);}

\newcommand{\drawAlphaB}[6][1]{ 
\pgfmathtruncatemacro{\result}{int(mod(#3+#4, #2))}
\pgfmathsetmacro{\angle}{360/#2}
\draw[#5] let \n1 = \result in (B#3) --
(B'\n1);}

\newcommand{\drawBetaA}[6][1]{ 
\pgfmathtruncatemacro{\result}{int(mod(#3+#4, #2))}
\pgfmathsetmacro{\angle}{360/#2}
\draw[#5] let \n1 = \result in (A#3) -- 
(B'\n1);}

\newcommand{\drawBetaB}[6][1]{ 
\pgfmathtruncatemacro{\result}{int(mod(#3+#4, #2))}
\pgfmathsetmacro{\angle}{360/#2}
\draw[#5] let \n1 = \result in (B#3)  --
(A'\n1);}

\newcommand{\drawSpoke}[2]{ 
\draw[#2] (C#1) -- (z#1);}

\newcommand{\drawLoop}[4]{
\pgfmathtruncatemacro{\result}{int(mod(#2+#4, #1))}
\draw[#3] let \n1 = \result in (z#2) -- (z\n1);}

\newcommand{\drawGTwentyEightBetaBlobs}[4][1]{
\pgfmathtruncatemacro{\result}{#2-1}
\pgfmathsetmacro{\angle}{360/#2}
\foreach \i in {0, ..., \result}{
	\node[rotate = \i*\angle,  circle, gray!70!white, minimum width = #1 cm, inner sep = 1pt] (G\i) at ($(\angle*\i:2)$){};
	\node[vtx, fill = orange] (C\i) at (G\i.180) {};
	\node[vtx, fill = gray!50!white,] (x4\i) at (G\i.180-72) {};
	\node[vtx, fill = gray!50!white,] (x1\i) at (G\i.180-2*72) {};
	\node[vtx, fill = yellow] (A'\i) at (G\i.180+72) {};
	\node[vtx, fill = gray!50!white,] (x3\i) at (G\i.180+2*72) {};
	\node[vtx, fill = gray!50!white] (x2\i) at ($(x1\i) + (\angle*\i:#4)$) {};
	\node[vtx, fill = green!70!black] (B'\i) at ($(x3\i) + (\angle*\i:#4)$) {};
	
	\node[vtx, fill = mymagenta] (A\i) at ($(\angle/2+\angle*\i:2+#4)$){};
	\node[vtx, fill = cyan] (B\i) at ($(\angle/2+\angle*\i:2)$){};

	\node[vtx, fill=black] (z\i) at (\angle*\i:#3){};
	\draw[gray!70!white] (C\i) -- (x3\i) -- (x4\i) -- (A'\i) -- (x1\i)--(C\i);
	\draw[gray!70!white] (x4\i) -- (B\i) -- (A\i) -- (x2\i)--(B'\i)--(x3\i);
	\draw[gray!70!white] (x1\i) -- (x2\i);
}
}

\newcommand{\drawGTwentyEightAlphaBlobs}[4][1]{
\pgfmathtruncatemacro{\result}{#2-1}
\pgfmathsetmacro{\angle}{360/#2}
\foreach \i in {0, ..., \result}{
	\node[rotate = \i*\angle,  circle, gray!70!white, minimum width = #1 cm, inner sep = 1pt] (G\i) at ($(\angle*\i:2)$){};
	\node[vtx, fill = orange] (C\i) at (G\i.180) {};
	\node[vtx, fill = gray!50!white] (x4\i) at (G\i.180-72) {};
	\node[vtx, fill = gray!50!white] (x1\i) at (G\i.180-2*72) {};
	\node[vtx, fill = yellow] (A'\i) at (G\i.180+72) {};
	\node[vtx, fill = gray!50!white] (x3\i) at (G\i.180+2*72) {};
	\node[vtx, fill = gray!50!white] (x2\i) at ($(x1\i) + (\angle*\i:#4)$) {};
	\node[vtx, fill = green!70!black] (B'\i) at ($(x3\i) + (\angle*\i:#4)$) {};
	
	\node[vtx, fill = mymagenta] (A\i) at ($(\angle/3+\angle*\i:2+#4)$){};
	\node[vtx, fill = cyan] (B\i) at ($(2*\angle/3+\angle*\i:2+#4)$){};

	\node[vtx, fill=black] (z\i) at (\angle*\i:#3){};
	\draw[gray!70!white] (C\i) -- (x3\i) -- (x4\i) -- (A'\i) -- (x1\i)--(C\i);
	\draw[gray!70!white] (x4\i) -- (B\i) -- (A\i) -- (x2\i)--(B'\i)--(x3\i);
	\draw[gray!70!white] (x1\i) -- (x2\i);
}
}

\newcommand{\drawGTwentyEightNEWBlobs}[4][1]{
\pgfmathtruncatemacro{\result}{#2-1}
\pgfmathsetmacro{\angle}{360/#2}
\foreach \i in {0, ..., \result}{
	\node[rotate = \i*\angle,  circle,  minimum width = #1 cm, inner sep = 1pt] (G\i) at ($(\angle*\i:2)$){};
	\node[vtx, fill = gray!50!white] (x1\i) at (G\i.0) {};
	\node[vtx, fill = gray!50!white] (x4\i) at (G\i.0-72) {};
	\node[vtx, fill = gray!50!white ] (x3\i) at (G\i.0-2*72) {};
	\node[vtx, fill = gray!50!white] (x2\i) at (G\i.72) {};
	\node[vtx, fill = mymagenta ] (A\i) at (G\i.2*72) {};
	\node[vtx, fill = cyan] (B\i) at ($(0,0)!1.3!(x2\i)$) {};
	\node[vtx, fill = gray!50!white] (x5\i) at($(0,0)!1.3!(x1\i)$) {};
	\node[vtx, fill = green!70!black] (B'\i) at ($(0,0)!1.3!(x4\i)$)  {};
	
	\node[vtx, fill = yellow] (A'\i) at ($(-\angle/2+\angle*\i:1-#4/2)$){};
	\node[inner sep = 1pt] (C\i) at (A'\i){};

	\node[vtx, fill=black] (z\i) at (-\angle/2+\angle*\i:#3){};
	\draw[gray!70!white] (A\i) -- (x4\i) -- (x2\i) -- (x3\i) -- (x1\i)--(A\i);
	\draw[gray!70!white] (x2\i) -- (B\i) -- (x5\i) -- (B'\i) -- (x4\i);
	\draw[gray!70!white] (x3\i) -- (A'\i) (x1\i)--(x5\i);
}
}

\newcommand{\drawGThirtyFourNoNineBlobs}[5][1]{
\pgfmathtruncatemacro{\result}{#2-1}
\pgfmathsetmacro{\angle}{360/#2}
\foreach \i in {0, ..., \result}{
	\node[rotate = \i*\angle,  circle, gray!70!white, minimum width = #1 cm, inner sep = 1pt] (G\i) at ($(\angle*\i:2)$){};
	\node[vtx,fill = gray!50!white](x1\i) at (G\i.0) {};
	\node[vtx,fill = gray!50!white] (x4\i) at (G\i.0-72) {};
	\node[vtx,fill = cyan] (B\i) at (G\i.0-2*72) {};
	\node[vtx, fill = gray!50!white] (x2\i) at (G\i.72) {};
	\node[vtx,fill = green!70!black] (B'\i) at (G\i.2*72) {};
	\node[vtx, fill = gray!50!white] (x6\i) at ($(0,0)!1.3!(x2\i)$) {};
	\node[vtx, fill = gray!50!white] (x5\i) at($(0,0)!1.2!(x1\i)$) {};
	\node[vtx, fill = gray!50!white] (x3\i) at ($(0,0)!1.3!(x4\i)$)  {};
	\node[vtx, fill = yellow] (A'\i) at ($(-\angle/3+\angle*\i:2*1.3)$){};
	\node[vtx, fill = mymagenta] (A\i) at ($(\angle/3+\angle*\i:2*1.3)$){};
	\node[vtx, fill = orange] (C\i) at (\angle*\i:#5){};

	\node[vtx, fill=black] (z\i) at (\angle*\i:#3){};
	
	\draw[gray!70!white] (B'\i) -- (x4\i) -- (x2\i) -- (B\i) -- (x1\i)--(B'\i);
	\draw[gray!70!white] (x2\i) -- (x6\i) -- (x5\i) -- (x3\i) -- (x4\i);
	\draw[gray!70!white] (x6\i) -- (A\i) -- (C\i) -- (A'\i)--(x3\i);
	\draw[gray!70!white]  (x1\i)--(x5\i);
}
}

\newcommand{\drawOneBlob}[7][]{
	\node[ draw,  cloud, gray!70!white, minimum width = 1.3 cm, minimum height=2 cm, inner sep = 1pt] (G#2) at (#2,0) {#1};
	
	\node[vtx, fill = orange] (C#2) at (G#2.270) {};
	\node[vtx, fill = mymagenta] (A#2) at (G#2.270-45) {};
	\node[vtx, fill = yellow] (A'#2) at (G#2.270+45) {};
	\node[vtx, fill = cyan] (B#2) at (G#2.90+45) {};
	\node[vtx, fill = green!70!black] (B'#2) at (G#2.90-45) {};
	
	\node[left=of A#2] (a#2) {};
	\node[left = of B#2] (b#2) {};
	\node[right =of A'#2] (a'#2) {};
	\node[right =of B'#2] (b'#2) {};

	\node[vtx,fill=black, 
	below=of C#2] (z#2){};
	
	\draw[ultra thick, #3] (B#2) -- (b#2);
	\draw[ultra thick, #4] (A#2) -- (a#2);
	\draw[ultra thick, #5] (C#2) -- (z#2);
	\draw[ultra thick, #6] (A'#2) -- (a'#2);
	\draw[ultra thick, #7] (B'#2) -- (b'#2);
	}

\newcommand{\drawOneBlobWithLabels}[7][]{
	\node[ draw,  cloud, gray!70!white, minimum width = 1.5 cm, minimum height=2.5 cm, inner sep = 1pt] (G#2) at (#2,0) {#1};
	
	\node[vtx, white,fill = orange] (C#2) at (G#2.270) {$v$};
	\node[vtx, white,fill = cyan] (B#2) at (G#2.90+45) {$B$};
	\node[vtx, white,fill = green!70!black] (B'#2) at (G#2.90-45) {$B'$};
	\node[vtx, white,fill = mymagenta] (A#2) at (G#2.270-45) {$A$};
	\node[vtx, fill = yellow] (A'#2) at (G#2.270+45) {$A'$};
	
	\node[vtx,white,fill=black, 
	below=of C#2] (z#2){$w$};
	
	}

\newcommand{\drawOneGTwentyEight}{
\pgfmathsetmacro{\off}{1}
\pgfmathsetmacro{\angle}{90}
\foreach \i in {1}{
	\node[rotate = \angle,  circle, gray!70!white, minimum width = 2 cm, inner sep = 1pt] (G\i) at ($(\angle*\i:2)$){};
	\node[vtx, fill = orange] (C\i) at (G\i.180) {$v$};
	\node[vtx, fill = gray!50!white] (x4\i) at (G\i.180-72) {};
	\node[vtx, fill = gray!50!white] (x1\i) at (G\i.180-2*72) {};
	\node[vtx, fill = yellow] (A'\i) at (G\i.180+72) {$A'$};
	\node[vtx, fill = gray!50!white] (x3\i) at (G\i.180+2*72) {};
	\node[vtx, fill = gray!50!white] (x2\i) at ($(x1\i) + (\angle*\i:\off)$) {};
	\node[vtx, fill = green!70!black] (B'\i) at ($(x3\i) + (\angle*\i:\off)$) {$B'$};
	
	\node[vtx, fill = mymagenta,  below left=of x2\i] (A\i){$A$}; 
	\node[vtx, fill = cyan, left =  .5 of x4\i] (B\i){$B$};

	\node[vtx, white,fill=black, below = of C\i] (z\i) {$w$};
	
	\draw[gray!70!white] (C\i) -- (z\i);
	\draw[gray!70!white] (C\i) -- (x3\i) -- (x4\i) -- (A'\i) -- (x1\i)--(C\i);
	\draw[gray!70!white] (x4\i) -- (B\i) -- (A\i) -- (x2\i)--(B'\i)--(x3\i);
	\draw[gray!70!white] (x1\i) -- (x2\i);
	}
}

\newcommand{\drawOneGTwentyEightNEW}{
\pgfmathsetmacro{\angle}{90}
\pgfmathsetmacro{\off}{.6}
\pgfmathsetmacro{\scale}{1.4}
\foreach \i in {1}{
	\node[rotate = \i*\angle,  circle, blue, minimum width = 2 cm, inner sep = 1pt] (G\i) at ($(\angle*\i:2)$){};
	\node[vtx, fill = gray!50!white] (x1\i) at (G\i.0) {};
	\node[vtx, fill = gray!50!white] (x4\i) at (G\i.0-72) {};
	\node[vtx, fill = gray!50!white ] (x3\i) at (G\i.0-2*72) {};
	\node[vtx, fill = gray!50!white] (x2\i) at (G\i.72) {};
	\node[vtx, fill = mymagenta ] (A\i) at (G\i.2*72) {$A$};
	\node[vtx, fill = cyan] (B\i) at ($(0,0)!\scale!(x2\i)$) {$B$};
	\node[vtx, fill = gray!50!white] (x5\i) at($(0,0)!\scale!(x1\i)$) {};
	\node[vtx, fill = green!70!black] (B'\i) at ($(0,0)!\scale!(x4\i)$)  {$B'$};
	
	\node[vtx, fill = yellow, below right = .5 of x3\i] (A'\i) {$A'$};
	\node[inner sep = 1pt] (C\i) at (A'\i){};

	\node[vtx, white,fill=black, below = 3.5 of x1\i] (z\i) {$w$};
	
		\draw[gray!70!white] (C\i) -- (z\i);
	\draw[gray!70!white] (A\i) -- (x4\i) -- (x2\i) -- (x3\i) -- (x1\i)--(A\i);
	\draw[gray!70!white] (x2\i) -- (B\i) -- (x5\i) -- (B'\i) -- (x4\i);
	\draw[gray!70!white] (x3\i) -- (A'\i) (x1\i)--(x5\i);
}
}

\newcommand{\drawGThirtyFourNoFourBlobs}[4][1]{
\pgfmathtruncatemacro{\result}{#2-1}
\pgfmathsetmacro{\angle}{360/#2}
\pgfmathsetmacro{\scale}{1.3}
\foreach \i in {0, ..., \result}{
	\node[rotate = \i*\angle,  circle,  minimum width = #1 cm, inner sep = 1pt] (G\i) at ($(\angle*\i:2)$){};
	\node[vtx, fill = gray!50!white] (x1\i) at (G\i.0) {};
	\node[vtx, fill = gray!50!white] (x4\i) at (G\i.0-72) {};
	\node[vtx, fill = yellow ] (A'\i) at (G\i.0-2*72) {};
	\node[vtx, fill = gray!50!white] (x2\i) at (G\i.72) {};
	\node[vtx, fill = gray!50!white ] (x3\i) at (G\i.2*72) {};
	\node[vtx, fill = gray!50!white] (x5\i) at ($(0,0)!\scale!(x2\i)$) {};
	\node[vtx, fill = gray!50!white] (x6\i) at($(0,0)!\scale!(x1\i)$) {};
	\node[vtx, fill = green!70!black] (B'\i) at ($(0,0)!\scale!(x4\i)$)  {};
	
	\node[vtx, fill = mymagenta] (A\i) at ($(\angle/2+\angle*\i:1+#4/2)$){};
	\node[vtx, fill = cyan] (B\i) at ($(\angle/2+\angle*\i:2+#4)$){};

	\node[vtx, fill=orange] (C\i) at (\angle*\i:1+#3){};

	\node[vtx, fill=black] (z\i) at (\angle*\i:#3){};
	
	\draw[gray!70!white] (x1\i) -- (x3\i) -- (x4\i) -- (x2\i) -- (A'\i)--(x1\i);
	\draw[gray!70!white] (x4\i) -- (B'\i) -- (x6\i) -- (x5\i) -- (x2\i);
	\draw[gray!70!white](B\i) -- (C\i) -- (A\i)-- (x3\i) (x5\i) -- (B\i);
	\draw[gray!70!white] (x2\i) -- (x5\i) (x1\i) -- (x6\i);

}
}

\newcommand{\drawOneGThirtyFourNoFour}{
\pgfmathsetmacro{\angle}{90}
\pgfmathsetmacro{\scale}{1.3}
\pgfmathsetmacro{\off}{1.3}

\foreach \i in {1}{
	\node[rotate = \i*\angle,  circle,  minimum width = 2 cm, inner sep = 1pt] (G\i) at ($(\angle*\i:2)$){};
	\node[vtx, fill = gray!50!white] (x1\i) at (G\i.0) {};
	\node[vtx, fill = gray!50!white] (x4\i) at (G\i.0-72) {};
	\node[vtx, fill = yellow ] (A'\i) at (G\i.0-2*72) {$A'$};
	\node[vtx, fill = gray!50!white] (x2\i) at (G\i.72) {};
	\node[vtx, fill = gray!50!white ] (x3\i) at (G\i.2*72) {};
	\node[vtx, fill = gray!50!white] (x5\i) at ($(0,0)!\scale!(x2\i)$) {};
	\node[vtx, fill = gray!50!white] (x6\i) at($(0,0)!\scale!(x1\i)$) {};
	\node[vtx, fill = green!70!black] (B'\i) at ($(0,0)!\scale!(x4\i)$)  {$B'$};
	
	\node[vtx, fill = mymagenta, below left = of x3\i] (A\i) {$A$};
	\node[vtx, fill = cyan, below left = of x5\i] (B\i) {$B$};

	\node[vtx, fill=orange] (C\i) at (\angle*\i:.5){$v$};

	\node[vtx, white,fill=black] (z\i) at (\angle*\i:-1){$w$};
	
	\draw[gray!70!white] (x1\i) -- (x3\i) -- (x4\i) -- (x2\i) -- (A'\i)--(x1\i);
	\draw[gray!70!white] (x4\i) -- (B'\i) -- (x6\i) -- (x5\i) -- (x2\i);
	\draw[gray!70!white](B\i) -- (C\i) -- (A\i)-- (x3\i) (x5\i) -- (B\i);
	\draw[gray!70!white] (x2\i) -- (x5\i) (x1\i) -- (x6\i);
	\draw[gray!70!white] (z\i) -- (C\i);

}
}

\newcommand{\colorNineAlt}[4]{
\pgfmathsetmacro{\i}{#1}
\draw[thick, #3] (C\i) -- (A\i) (A'\i) -- (x3\i) (x2\i) -- (x4\i) (x6\i) -- (x5\i) (B'\i) -- (x1\i) ;
\draw[thick,#4] (C\i) -- (A'\i) (A\i) -- (x6\i) (x2\i) --(B\i) (x4\i) -- (x3\i) (x1\i) -- (x5\i);
\draw[thick,#2] (x2\i) -- (x6\i) (x5\i) -- (x3\i) (B'\i) -- (x4\i) (B\i) -- (x1\i);}

\newcommand{\colorNineAltBot}[4]{
\pgfmathsetmacro{\i}{#1}
\draw[thick, #3] (C\i) -- (A'\i) (x2\i) -- (x6\i) (x5\i) -- (x3\i) (B'\i) -- (x4\i) (B\i)--(x1\i);
\draw[thick,#4] (x6\i) -- (x5\i) (x1\i) -- (B'\i) (B\i)--(x2\i) (x4\i) -- (x3\i) (A\i) -- (C\i);
\draw[thick,#2] (A\i) -- (x6\i) (x2\i) --(x4\i) (x3\i) -- (A'\i) (x1\i) -- (x5\i);}

\newcommand{\drawBetaAwiggle}[6][1]{ 
\pgfmathtruncatemacro{\result}{int(mod(#3+#4, #2))}
\pgfmathsetmacro{\angle}{360/#2}
\draw[#5] let \n1 = \result in (A#3) to[bend left=#1]  
(B'\n1);}

\newcommand{\drawBetaBwiggle}[6][1]{ 
\pgfmathtruncatemacro{\result}{int(mod(#3+#4, #2))}
\pgfmathsetmacro{\angle}{360/#2}
\draw[#5] let \n1 = \result in (B#3)  to[bend left=#1]   
(A'\n1);}

\newcommand{\betaPair}[5][0]{
\drawBetaAwiggle[#1]{#2}{#3}{1}{#4, ultra thick}{0}
\drawBetaBwiggle[#1]{#2}{#3}{1}{#5, ultra thick}{0}
}

\newcommand{\drawAlphaAwiggle}[6][1]{ 
\pgfmathtruncatemacro{\result}{int(mod(#3+#4, #2))}
\pgfmathsetmacro{\angle}{360/#2}
\draw[#5] let \n1 = \result in (A#3) to[bend left=#1, looseness = #6]  
(A'\n1);}

\newcommand{\drawAlphaBwiggle}[6][1]{ 
\pgfmathtruncatemacro{\result}{int(mod(#3+#4, #2))}
\pgfmathsetmacro{\angle}{360/#2}
\draw[#5] let \n1 = \result in (B#3)  to[bend left=#1, looseness = #6]   
(B'\n1);}

\newcommand{\GpointsStuff}[8][1]{

\pgfmathtruncatemacro{\result}{#2-1}
\pgfmathsetmacro{\angle}{360/#2}
\pgfmathtruncatemacro{\i}{mod(#5,#2)}
\pgfmathtruncatemacro{\ii}{mod(#5+1, #2)}
	\node[rotate = \i*\angle,  circle,  minimum width = #1 cm, inner sep = 1pt] (G\i) at ($(\angle*\i:2)$){};
	\node[vtx, fill = gray!50!white] (x1\i) at (G\i.0) {};
	\node[vtx, fill = gray!50!white] (x4\i) at (G\i.0-72) {};
	\node[vtx, fill = gray!50!white ] (x3\i) at (G\i.0-2*72) {};
	\node[vtx, fill = gray!50!white] (x2\i) at (G\i.72) {};
	\node[vtx, fill = mymagenta ] (A\i) at (G\i.2*72) {};
	\node[vtx, fill = cyan] (B\i) at ($(0,0)!1.3!(x2\i)$) {};
	\node[vtx, fill = gray!50!white, label={[inner sep = 1 pt, outer sep = 3pt, font=\tiny]
	\angle*\i:#8}] (x5\i) at($(0,0)!1.3!(x1\i)$) {};
	\node[vtx, fill = green!70!black] (B'\i) at ($(0,0)!1.3!(x4\i)$)  {};
	
	\node[vtx, fill = yellow] (A'\i) at ($(-\angle/2+\angle*\i:#4)$){};
	
		\node[rotate = \ii*\angle,circle,  minimum width = #1 cm, inner sep = 1pt] (G\ii) at ($(\angle*\ii:2)$){};
		\coordinate (a\i) at ($(-\angle/2+\angle*\ii:#4)$);
		\node[] (x4\ii) at (G\ii.0-72) {};
		\coordinate[] (b\i) at ($(0,0)!1.3!(x4\ii)$);

	\node[vtx, fill=black] (z\i) at (-\angle/2+\angle*\i:#3){};
	\coordinate (wa\i) at (-\angle/2+\angle*\ii:#3){};}

\newcommand{\ThreeA}[8][1]{

\GpointsStuff{#2}{#3}{#4}{#5}{#6}{#7}{#8}

	\begin{scope}[on background layer]
	\draw[ cyan, thick] (A\i) -- (x4\i);
	\draw[red,ultra thick] (x4\i) -- (x2\i);
	\draw[ #6](x2\i) -- (x3\i);
	\draw[red,ultra thick ] (x3\i) -- (x1\i);
	\draw[  cyan, thick] (x1\i)--(A\i);
	\draw[  #6] (x2\i) -- (B\i);
	\draw[ red ,ultra thick] (B\i) -- (x5\i);
	\draw[  cyan, thick ] (x5\i)  -- (B'\i);
	\draw[  cyan, thick ] (B'\i) -- (x4\i);
	\draw[#6 ] (x3\i) -- (A'\i);
	\draw[  cyan, thick ] (x1\i)--(x5\i);
	\draw [#6 ] (A'\i) -- (z\i);
	\draw[  red,ultra thick ] (A\i) -- (a\i);
	\draw[ #6 ] (B\i) -- (b\i);
	\draw[ red,ultra thick ] (z\i) --(wa\i);
	\end{scope}

}

\newcommand{\FourB}[8][1]{
\GpointsStuff{#2}{#3}{#4}{#5}{#6}{#7}{#8}
	\begin{scope}[on background layer]
	\draw[ #6] (A\i) -- (x4\i);
	\draw[red,ultra thick] (x4\i) -- (x2\i);
	\draw[#6 ](x2\i) -- (x3\i);
	\draw[ #6] (x3\i) -- (x1\i);
	\draw[red,ultra thick ] (x1\i)--(A\i);
	\draw[ #6 ] (x2\i) -- (B\i);
	\draw[ red,ultra thick ] (B\i) -- (x5\i);
	\draw[ #6 ] (x5\i)  -- (B'\i);
	\draw[ #6 ] (B'\i) -- (x4\i);
	\draw[red,ultra thick ] (x3\i) -- (A'\i);
	\draw[ #6 ] (x1\i)--(x5\i);
	\draw [#7 ] (A'\i) -- (z\i);
	\draw[  #6 ] (A\i) -- (a\i);
	\draw[ #6 ] (B\i) -- (b\i);
	\draw[ #7 ] (z\i) -- (wa\i);
\end{scope}
}

\newcommand{\FiveB}[8][1]{
\GpointsStuff{#2}{#3}{#4}{#5}{#6}{#7}{#8}
	\begin{scope}[on background layer]
	\draw[#6 ] (A\i) -- (x4\i);
	\draw[red,ultra thick] (x4\i) -- (x2\i);
	\draw[ #6](x2\i) -- (x3\i);
	\draw[ red,ultra thick] (x3\i) -- (x1\i);
	\draw[#6 ] (x1\i)--(A\i);
	\draw[ #6 ] (x2\i) -- (B\i);
	\draw[ #6 ] (B\i) -- (x5\i);
	\draw[ red,ultra thick ] (x5\i)  -- (B'\i);
	\draw[ #6 ] (B'\i) -- (x4\i);
	\draw[#6 ] (x3\i) -- (A'\i);
	\draw[ #6 ] (x1\i)--(x5\i);
	\draw [#6 ] (A'\i) -- (z\i);
	\draw[  red,ultra thick ] (A\i) -- (a\i);
	\draw[ red,ultra thick ] (B\i) -- (b\i);
	\draw[ #6 ] (z\i) -- (wa\i);
\end{scope}

}

\newcommand{\SixA}[8][1]{
\GpointsStuff{#2}{#3}{#4}{#5}{#6}{#7}{#8}

	\begin{scope}[on background layer]
	\draw[ #6] (A\i) -- (x4\i);
	\draw[red,ultra thick] (x4\i) -- (x2\i);
	\draw[ cyan, thick ](x2\i) -- (x3\i);
	\draw[  cyan, thick] (x3\i) -- (x1\i);
	\draw[red,ultra thick ] (x1\i)--(A\i);
	\draw[  cyan, thick ] (x2\i) -- (B\i);
	\draw[  cyan, thick ] (B\i) -- (x5\i);
	\draw[ red,ultra thick ] (x5\i)  -- (B'\i);
	\draw[ #6 ] (B'\i) -- (x4\i);
	\draw[ red,ultra thick] (x3\i) -- (A'\i);
	\draw[  cyan, thick ] (x1\i)--(x5\i);
	\draw [#7 ] (A'\i) -- (z\i);
	\draw[ #6  ] (A\i) -- (a\i);
	\draw[ red,ultra thick] (B\i) -- (b\i);
	\draw[ red,ultra thick ] (z\i) -- (wa\i);
	\end{scope}

}
\newcommand{\Nine}[9][1]{

\GpointsStuff{#2}{#3}{#4}{#5}{#6}{#7}{#9}

	\begin{scope}[on background layer]
	\draw[red,ultra thick ] (A\i) -- (x4\i);
	\draw[#7] (x4\i) -- (x2\i);
	\draw[#7 ](x2\i) -- (x3\i);
	\draw[ red,ultra thick] (x3\i) -- (x1\i);
	\draw[#6 ] (x1\i)--(A\i);
	\draw[ red,ultra thick ] (x2\i) -- (B\i);
	\draw[ #6 ] (B\i) -- (x5\i);
	\draw[ red,ultra thick ] (x5\i)  -- (B'\i);
	\draw[ #7 ] (B'\i) -- (x4\i);
	\draw[ #7] (x3\i) -- (A'\i);
	\draw[ #6 ] (x1\i)--(x5\i);
	\draw [ red, ultra thick] (A'\i) -- (z\i);
	\draw[  #6 ] (A\i) -- (a\i);
	\draw[ #6 ] (B\i) -- (b\i);
	\draw[ #8 ] (z\i) -- (wa\i);
	\end{scope}
}
\newcommand{\Twelve}[8][1]{
\GpointsStuff{#2}{#3}{#4}{#5}{#6}{#7}{#8}
	\begin{scope}[on background layer]
	\draw[#6 ] (A\i) -- (x4\i);
	\draw[red,ultra thick] (x4\i) -- (x2\i);
	\draw[#6 ](x2\i) -- (x3\i);
	\draw[ red,ultra thick] (x3\i) -- (x1\i);
	\draw[ #6] (x1\i)--(A\i);
	\draw[ #6 ] (x2\i) -- (B\i);
	\draw[ #6 ] (B\i) -- (x5\i);
	\draw[ red,ultra thick ] (x5\i)  -- (B'\i);
	\draw[ #6 ] (B'\i) -- (x4\i);
	\draw[#6 ] (x3\i) -- (A'\i);
	\draw[ #6 ] (x1\i)--(x5\i);
	\draw [red,ultra thick ] (A'\i) -- (z\i);
	\draw[ red,ultra thick  ] (A\i) -- (a\i);
	\draw[ red,ultra thick ] (B\i) -- (b\i);
	\draw[ #7 ] (z\i) -- (wa\i);
	\end{scope}

}
\newcommand{\ThirteenB}[8][1]{
\GpointsStuff{#2}{#3}{#4}{#5}{#6}{#7}{#8}

	\begin{scope}[on background layer]
	\draw[ #6, thick] (A\i) -- (x4\i);
	\draw[red,ultra thick] (x4\i) -- (x2\i);
	\draw[ #6, thick](x2\i) -- (x3\i);
	\draw[#6 , thick] (x3\i) -- (x1\i);
	\draw[#6 , thick] (x1\i)--(A\i);
	\draw[ #6, thick ] (x2\i) -- (B\i);
	\draw[ #6 , thick] (B\i) -- (x5\i);
	\draw[ #6, thick ] (x5\i)  -- (B'\i);
	\draw[ #6 , thick] (B'\i) -- (x4\i);
	\draw[red,ultra thick ] (x3\i) -- (A'\i);
	\draw[ red,ultra thick ] (x1\i)--(x5\i);
	\draw [#7 ] (A'\i) -- (z\i);
	\draw[  red,ultra thick ] (A\i) -- (a\i);
	\draw[ red,ultra thick ] (B\i) -- (b\i);
	\draw[ #7 ] (z\i) -- (wa\i);
\end{scope}

}
\newcommand{\FourteenA}[8][1]{
\GpointsStuff{#2}{#3}{#4}{#5}{#6}{#7}{#8}

	\begin{scope}[on background layer]
	\draw[ red,ultra thick] (A\i) -- (x4\i);
	\draw[ cyan, thick] (x4\i) -- (x2\i);
	\draw[ red,ultra thick](x2\i) -- (x3\i);
	\draw[ #6] (x3\i) -- (x1\i);
	\draw[ #6] (x1\i)--(A\i);
	\draw[  cyan, thick ] (x2\i) -- (B\i);
	\draw[  cyan, thick ] (B\i) -- (x5\i);
	\draw[  cyan, thick ] (x5\i)  -- (B'\i);
	\draw[  cyan, thick ] (B'\i) -- (x4\i);
	\draw[#6 ] (x3\i) -- (A'\i);
	\draw[ red,ultra thick ] (x1\i)--(x5\i);
	\draw [ #6] (A'\i) -- (z\i);
	\draw[ #6  ] (A\i) -- (a\i);
	\draw[ red ,ultra thick ] (B\i) -- (b\i);
	\draw[ red,ultra thick ] (z\i) -- (wa\i);
	\end{scope}

}
\newcommand{\ExtraAB}[7][1]{
\pgfmathsetmacro{\angle}{360/#2}
\pgfmathtruncatemacro{\i}{mod(#5,#2)}
\pgfmathtruncatemacro{\ii}{mod(#5+1, #2)}
\pgfmathtruncatemacro{\iii}{mod(#5-1, #2)}
	\node[rotate = \iii*\angle,circle,  minimum width = #1 cm, inner sep = 1pt] (G\iii) at ($(\angle*\iii:2)$){};
		\coordinate (a'\i) at  (G\iii.2*72) ;
		\node[] (x4\iii) at (G\iii.0-72) {};
		\node[] (x2\iii) at (G\iii.72) {};
		\coordinate[] (b'\i) at 
		($(0,0)!1.3!(x2\iii)$);
		\coordinate (wb\i) at (-\angle/2+\angle*\iii:#3);
}

\newcommand{\SixteenB}[8][1]{
\GpointsStuff{#2}{#3}{#4}{#5}{#6}{#7}{#8}
	\begin{scope}[on background layer]
	\draw[ #6] (A\i) -- (x4\i);
	\draw[#6] (x4\i) -- (x2\i);
	\draw[ red ,ultra thick](x2\i) -- (x3\i);
	\draw[#7 ] (x3\i) -- (x1\i);
	\draw[red,ultra thick ] (x1\i)--(A\i);
	\draw[ #6 ] (x2\i) -- (B\i);
	\draw[ red,ultra thick ] (B\i) -- (x5\i);
	\draw[ #7 ] (x5\i)  -- (B'\i);
	\draw[ red,ultra thick ] (B'\i) -- (x4\i);
	\draw[ #7] (x3\i) -- (A'\i);
	\draw[ #7 ] (x1\i)--(x5\i);
	\draw [ #7] (A'\i) -- (z\i);
	\draw[  #6 ] (A\i) -- (a\i);
	\draw[ #6 ] (B\i) -- (b\i);
	\draw[ #7 ] (z\i) -- (wa\i);
	\end{scope}

}

\begin{document}

\title{Cyclic pseudo-{L}oupekine snarks}



\author{Leah Wrenn Berman}
\address{Department of Mathematics \& Statistics, University of Alaska Fairbanks\\ Fairbanks, Alaska, USA}
\email{lwberman@alaska.edu}
\author{D\'eborah Oliveros}
\address{Instituto de Mat\'ematicas, Universidad Nacional Auton\'oma de M\'exico---Campus Juriquilla\\ Juriquilla, Qu\'eretaro, M\'exico}
\email{debolivero@matem.unam.mx}
\author{Gordon I. Williams}
\address{Department of Mathematics \& Statistics, University of Alaska Fairbanks\\ Fairbanks, Alaska, USA}
\email{giwilliams@alaska.edu}

\date{Received: date / Accepted: date}

\maketitle

\begin{abstract}
In 1976, Loupekine introduced (via Isaacs) a very general way of constructing new snarks from old snarks by cyclically connecting multipoles constructed from smaller snarks. In this paper, we generalize Loupekine's construction to produce a variety of snarks which can be drawn with $m$-fold rotational symmetry for $m\geq 3$ (and often, $m$ odd), constructed as $\mathbb{Z}_{m}$ lifts of \emph{voltage graphs} with certain properties; we call these snarks \emph{cyclic pseudo-Loupekine snarks}. In particular, we discuss three infinite families of snarks which can be drawn with $\mathbb{Z}_{m}$ rotational symmetry whose smallest element is constructed from 3 snarks with 3-fold rotational symmetry on 28 vertices; one family has the property that the oddness of the family increases with $m$.
We also develop a new infinite family of snarks, of order $12m$ for each odd $m\geq 3$, which can be drawn with $m$-fold rotational symmetry and which are constructed beginning with a 3-edge-colorable graph, instead of a snark.  


\end{abstract}




We define a \emph{snark} to be a cubic graph whose chromatic index is 4. We define a \emph{proper snark} to be a snark which has cyclic edge-connectivity $\geq 4$ (that is, it is impossible to find a set of three edges whose removal disconnects the graph into components which each contain a cycle) and whose girth is at least 5. 

We are interested in developing specific families of snarks with the property that they have (relatively) large automorphism groups (which we define arbitrarily as having order strictly greater than 2), by forming graphs which have drawings with $\mathbb{Z}_{m}$ rotational symmetry. In this paper, we first determine all proper snarks of order $n$ (with girth $\geq 4$ for $10 \leq n \leq 32$ and girth $\geq 5$ for $n = 34, 36$) on at most 36 vertices whose automorphism group is divisible by 3, working from a database of all proper snarks of order at most 36 from the House of Graphs \cite[Snarks]{HoG}. Next, we modify Loupekine's classical construction of snarks using voltage graphs, to produce snarks with guaranteed rotational symmetry. In particular, we describe three snarks on 28 vertices with 3-fold rotational symmetry which can arise from Loupekine's construction and from them produce infinite families of snarks with $10m$ vertices and $\mathbb{Z}_{m}$ symmetry  for any odd $m \geq 3$, and we discuss the oddness and cyclic edge-connectivity of members of these families. Finally, we construct  a new infinite family of snarks with $12m$ vertices and $m$-fold rotational symmetry for any odd $m$, which are constructed by a construction similar to Loupekine's construction, by removing a path of length 2 from a certain starting graph and systematically connecting the resulting dangling edges, where the starting graph is \emph{not} a snark! These resulting snarks (possibly excepting the $m = 3$ case, which is one of the 3833587 snarks on 34 vertices with girth at least 5 listed at \cite{HoG}, although we don't think it has been analyzed previously) are new.

\section{Rotationally symmetric snarks}

The first snarks that were discovered tended to possess a lot of symmetry; the smallest and first snark, the Petersen graph, is traditionally drawn with 5-fold dihedral symmetry. The Flower Snarks $J_{2k+1}$ found by Rufus Isaacs \cite{Isa1975} were initially drawn with a twist which broke pure rotational symmetry, but can be drawn with $(2k+1)$-fold dihedral symmetry (see \cite{ClaEnt1983}), and in fact, they have the interesting property that $|V(J_{n})|=|Aut(J_{n})| = 4n$ \cite{FioRui2008}.

Recent investigations on snarks have tended to focus on constructions of infinite families of snarks with certain properties such as having high  
oddness ~\cite{LukMacMaz2015,Hag2016}, having arbitrarily large girth \cite{Koc1996}, being irreducible \cite{MacSko2006b,ChlSko2006,Ste1998}, or being edge and vertex critical \cite{Gol1981}. Snarks are also possible counterexamples to certain important conjectures such as the Fulkerson conjecture \cite{KarCam2014} and other questions relating cubic graphs and perfect matchings~\cite{EspMaz2014} and to graph-coloring conjectures \cite{Koc2009}, etc. Other investigations have determined the automorphism 
groups of particular snarks \cite{FioRui2008}, and characterized irreducible snarks with different conditions \cite{ChlSko2010,Ste1998} as well as considering various factorization properties \cite{CavMesRui1998}. 
New concepts and variations of classical definitions have also been explored in the context of snarks, such as determining the total chromatic number \cite{BriPreSas2015} and almost-Hamiltonicity \cite{CavMurRui2003}.

However, in this paper---like the initial investigation of snarks---we are interested in snarks that can be drawn with nontrivial geometric symmetry. We define a \emph{cyclic snark} to be a snark that can be drawn with $m$-fold rotational symmetry for some $m>2$. 

Thus, every cyclic snark has the property that the automorphism group of that snark must possess an element which is cyclic of order $m$. Unsurprisingly, it turns out that among known snarks of fixed order,  those whose automorphism group has order greater than 2 are rare.

The House of Graphs \cite[Snarks]{HoG} maintains a list, in Graph6 format \cite{graph6}, of all proper snarks on up to 36 vertices. For $n <34$, we computed the automorphism group of all snarks listed in the database (using a combination of Mathematica \cite{mathematica} and Sage \cite{sage}, as well as some standard programs and methods for converting Graph6 objects to graphs and computing graph automorphism groups and orders), and for $n = 34, 36$ we computed the automorphism groups of all snarks with girth at least 5 (due to computational limitations). Our findings are listed in Table \ref{Table:autOrderSnarks}.

\begin{table}[htbp]
\caption{The number of snarks on up to 36 vertices, whose automorphism groups have order more than 2, collected into counts by divisibility (categories are not disjoint). The lists and numbers of snarks were taken from the House of Graphs \cite[Snarks]{HoG}. 
}
\begin{center}
\begin{tabular}{c||c|c|c||c| c
}
$n$ & \# snarks $S$ & \#  ($|Aut(S)|>2)$ & \# $(3 \mid |Aut(S)|)$ & \# $(4 \mid |Aut(S)|)$ & \# $(5 \mid |Aut(S)|)$ 
\\ \hline
10 & 1 & 1 & 1 & 1 & 1 
\\ \hline
18 & 2&2 & 0 & 2 & 0 
\\ \hline
20 & 6&  3 & 0 & 3 & 1 
\\ \hline
22 &31& 7 & 2 & 7 & 0 
\\ \hline 
24 & 155 & 21& 0 & 21 & 0 
\\ 
26 & 1297 & 100 & 0 & 100 & 0
\\ \hline
28 & 12517 & 391 & 3 & 390 & 0 
\\ \hline
30 & 139854 & 2073 & 0 & 2073 & 1 
\\ \hline
32&1764950	 &12630& 0 & 12360 & 0 
\\ \hline
34 & 3833587\footnote{\label{footGirth1}Only girth $\geq 5$} & 5167 & 19& 5153 &0
 \\ \hline
36 & 60167732\footref{footGirth1}&	25936 & 10 & 25931 & 0 
\end{tabular}
\end{center}
\label{Table:autOrderSnarks}
\end{table}%

Further analysis of counts is given in Table \ref{Table:MoreAutOrders}: here, we count all snarks whose automorphism group has order larger than 2, by the order of the group.  Note that this verifies the numbers shown in Table 6 of \cite{BriGoeHag2012}. All the group orders that occurred in the computations are listed. Two things are readily apparent: first, there are a lot of snarks whose automorphism group is a power of 2, and second, there are a few snarks for certain orders whose automorphism group has order divisible by 3. For $n = 22, 28, 34$, observation of the elements of the automorphism group shows that these snarks all have a central fixed point. In thinking about these snarks as potentially being the first member $(m = 3)$ of some infinite family, we expand the central fixed point into a central triangle, which then can be generalized into some sort of $m$-gon (possibly star-$m$-gon, or collection of $m$-gons). (Note that these expanded snarks do not appear in the list of snarks at the House of Graphs, since those counts required that the cyclic edge-connectivity is at least 4, and the expansion of the central vertex into a triangle causes the cyclic edge-connectivity to be 3.)

All snarks on $n = 28, 34,$ and $36$ vertices whose automorphism groups are divisible by 3 are explicitly listed in the appendices as Graph6 strings.

\begin{table}[htbp]
\caption{The number of proper snarks on up to 36 vertices with specified automorphism group order; these orders include all the non-trivial group orders found. Highlighted graphs admit drawings with 3-fold rotational symmetry. (Categories are disjoint.) }
\begin{center}
\begin{tabular}{c|ccccccccccccccc} \hline
order & \multicolumn{15}{c}{Number of snarks $S$ with $|Aut(S)| = x$}\\
$n$ &  3 & 4 & 6 & 8 & 12 & 16 & 20 & 24 & 28 & 32 & 36 & 48 & 64 & 80 & 120
\\ \hline \hline
10 & 0 & 0 & 0 & 0 & 0 & 0 & 0 & 0 & 0 & 
   0 & 0 & 0 & 0 & 0 & 1\footnote{Petersen Graph}\\
18\footnote{There are no proper snarks on 12, 14 or 16 vertices} & 0 & 1 & 0 & 1 & 0 & 0 & 0 & 0 & 0 &
   0 & 0 & 0 & 0 & 0 & 0 \\
20 & 0 & 2 & 0 & 0 & 0 & 0 & 1\footnote{\label{flower}Flower Snark} & 0 & 0 &
   0 & 0 & 0 & 0 & 0 & 0 \\
22 & 0 & 2 & 0 & 2 & \cellcolor{yellow}2\footnote{Loupekine's 22-vertex Snarks} & 1 & 0 & 0 & 0 &
   0 & 0 & 0 & 0 & 0 & 0 \\
24 & 0 & 18 & 0 & 3 & 0 & 0 & 0 & 0 & 0
   & 0 & 0 & 0 & 0 & 0 & 0 \\
26 & 0 & 78 & 0 & 19 & 0 & 2 & 0 & 0 & 0
   & 1 & 0 & 0 & 0 & 0 & 0 \\
28 & 0 & 329 & \cellcolor{yellow}1 & 48 & \cellcolor{yellow}2 & 10 & 0 & 0 &
   1\footref{flower} & 0 & 0 & 0 & 0 & 0 & 0 \\
30 &  0 & 1763 & 0 & 266 & 0 & 35 & 0 & 0
   & 0 & 8 & 0 & 0 & 0 & 1\footnote{Double Star Snark} & 0 \\
32 & 0 & 11236 & 0 & 1241 & 0 & 142 & 0
   & 0 & 0 & 11 & 0 & 0 & 0 & 0 & 0\\
34\footnote{\label{footGirth2}Only girth $\geq 5$} & \cellcolor{yellow}7 & 4798 & \cellcolor{yellow}7 & 329 & \cellcolor{yellow}1 & 20 & 0 & \cellcolor{yellow}3
   & 0 & 1 & 0 & \cellcolor{yellow}1 & 0 & 0 & 0  \\
36\footref{footGirth2} & \cellcolor{yellow}2 & 24855 & \cellcolor{yellow}3 & 1044 & \cellcolor{yellow}3 & 24 & 0 &0 & 0 & 2 & 1\footref{flower} & \cellcolor{yellow}1 & 1 & 0 & 0 \\ \hline
\end{tabular}
\end{center}
\label{Table:MoreAutOrders}
\end{table}%

\section{Loupekine's snark construction}\label{Sec:LoupConstr}

\begin{figure}[htbp]
\begin{center}
\ffigbox{
\begin{subfloatrow}
\ffigbox{\caption{Loupekine's first snark}\label{}}{
\includegraphics[width=.6\linewidth]{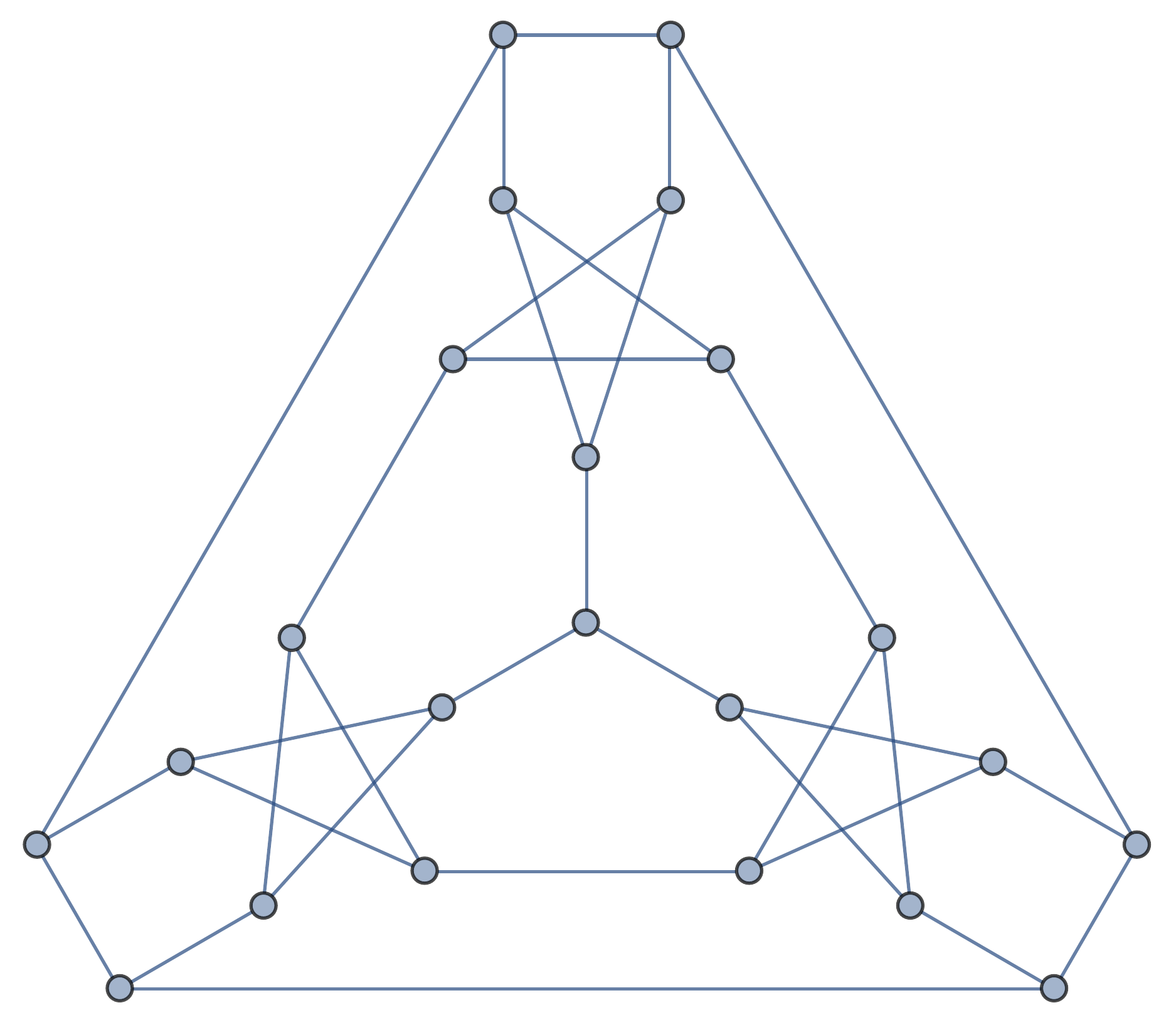} }
\ffigbox{\caption{Loupekine's second snark}\label{}}{
 \includegraphics[width=.6\linewidth]{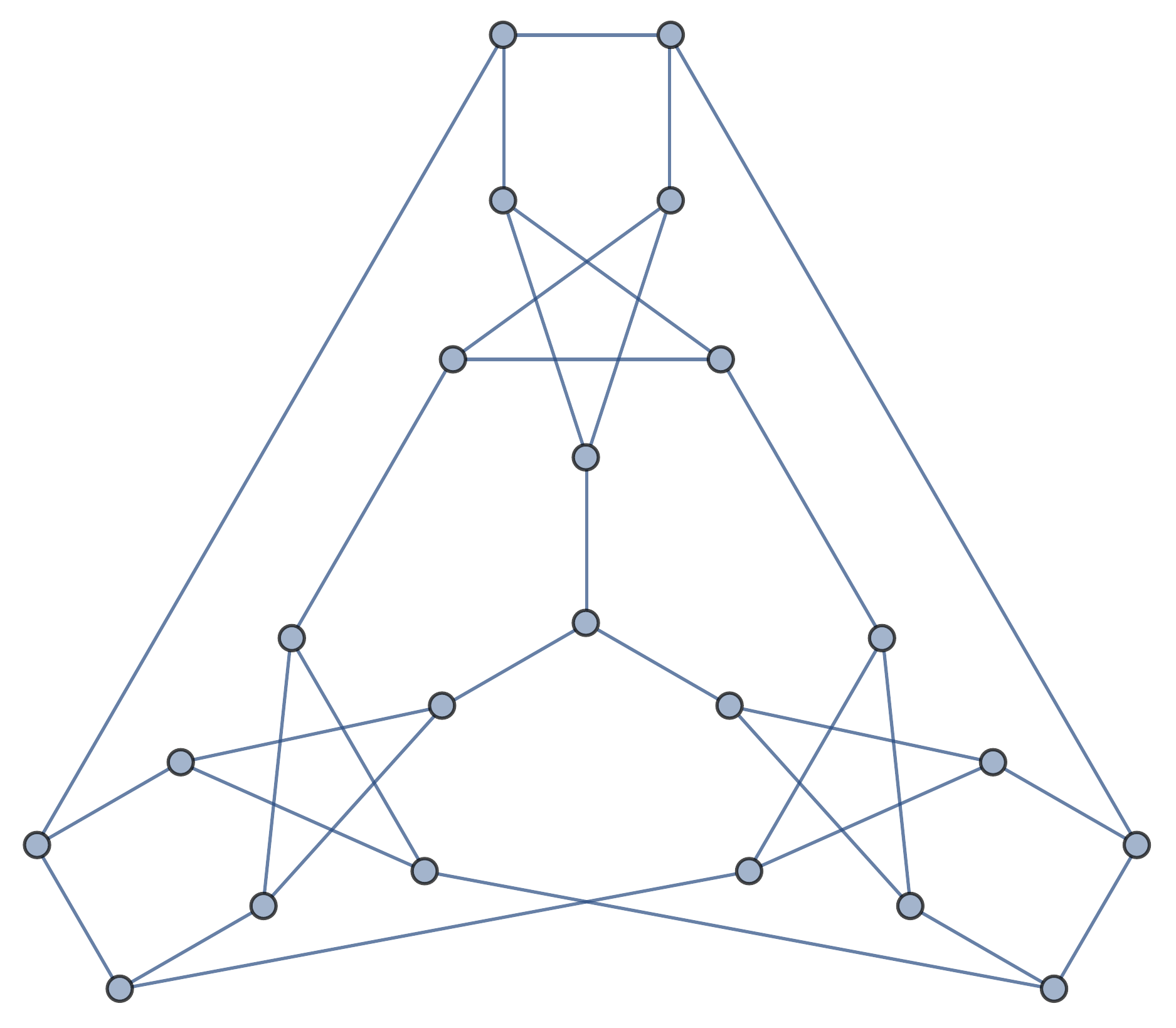}}
\end{subfloatrow}
}{
\caption{Loupekine's snarks on 22 vertices, in their default drawings in  Mathematica \cite{mathematica}. }
\label{LoupekineSnarks}
}
\end{center}
\end{figure}

In \cite{Wat1989,WatWil1991}, J.J. Watkins describes a method, due to   F\'eodor Loupekine, for constructing infinite families of snarks. Loupekine's method, which he developed after reading an article by Martin Gardner \cite{Gar1976}, was originally published by R. Isaacs in a hard-to-access Johns Hopkins technical report \cite{Isa1976}. The method requires the well-known Parity Lemma, which we state as follows (although it has been phrased in several ways, via zones \cite{Des1948,Isa1975}, semiedges/dangling edges/pendant edges \cite{Gol1981,BriPreSas2015,Isa1976,Wat1989,Fio1991}, cutsets \cite{Ste1998}, or multipoles \cite{ChlSko2006,MacSko2006b}). 

\begin{lemma}[Parity Lemma] \label{parityLemma} \cite{Des1948,Wat1989} 
If $G$ is a 3-edge-colored graph, whose edges have been colored 1, 2, 3, and $N$ is a cutset of size $n$ that contains $n_{i}$ edges of color $i$, then
\[ n_{1} \equiv n_{2}\equiv n_{3} \equiv n \bmod 2.\]
\end{lemma}

Although the term multipole is used in a more general setting (see for instance \cite{Hag2016,ChlSko2010,MacSko2006b}), throughout this paper a \emph{$k$-pole} $H$ refers to a graph with $k$ dangling edges (often called \emph{semiedges}); the term was introduced (somewhat more generally) in 1991 by M.A. Fiol \cite{Fio1991}. This type of object has also been called a semigraph \cite{BriPreSas2015} and a $k$-pendant graph \cite{Isa1975}. We are primarily interested in a specific type of 5-pole. In the remainder of the paper, all the 5-poles are formed by removing a path of three vertices and two edges from a cubic graph  (see Figure \ref{5multipole}).

Loupekine's snark construction method is as follows:

\begin{algorithm}\label{alg:Loupekine} \hfill
\begin{enumerate}
\item Choose your favorite snark $\bar{S}$.
\item Delete a path $P_{3}$ of two edges and three vertices from $\bar{S}$ to form a 5-pole $S$ which has 5 dangling edges.
\item By the Parity Lemma, 
 in any 3-edge-coloring of $S$, those 5 dangling edges, labelled (following Isaacs) as $A, B, C,B', A'$ in the order they were incident with the removed $P_{3}$ must be colored with three edges of one color, and the other two edges of the two remaining different colors. 
\item It is straightforward to show that if $S$ is properly 3-edge-colored and $A, B$ are colored different colors, then $A'$ and $B'$ must be colored the same color, so if some copies of $S$ are linked cyclically by joining edges labelled $A, B$ in one copy to edges labelled $A'$, $B'$ in the next copy, and the remaining $C$ edges are connected arbitrarily (in a 3-valent way), then if the result is a cycle of an odd number of linked copies, it's impossible to 3-edge-color that linked cycle. 
\end{enumerate}
\end{algorithm}

\begin{definition}A \emph{Loupekine 5-pole} is a 5-pole formed by deleting a path $P_{3}$ from a snark.\end{definition}

Assume you have $k$ copies $S_i$ from an original 5-pole $S$ with corresponding dangling edges $A_i, B_i, C_i,B'_i, A'_i$ as before, where $i=\{1,\dots , k\}$.
Then, depending on the way that the $A_{i}, B_{i}$ are connected to the successive $A_{i+1}, B_{i+1}$ and how the $C_{i}$ are connected, it may or may not be possible to draw the resulting graph with geometric symmetry. In the next section we specialize this construction to guarantee it produces snarks which can be drawn with nontrivial geometric symmetry, which we call \emph{cyclic Loupekine snarks}; a precise definition will be given below.

\section{Loupekine snarks and voltage graphs}
A \emph{voltage graph} \cite[Chapter 2]{GroTuc1987},\cite[Section 3.5]{PisSer2013} is a directed graph $\tilde{H}$ whose directed edges are labelled by invertible elements of some group $\Gamma$, the \emph{voltage group}. Voltage graphs are typically used to construct a \emph{derived graph}, or \emph{lift graph};  conversely, given a graph with certain symmetry properties, voltage graphs can be used to describe that graph (or class of graphs) concisely. In this paper we are exclusively interested in voltage graphs whose voltage group is some cyclic group $\mathbb{Z}_{m}$ (with group operation addition modulo $m$). 

Given any voltage graph $\tilde{H}$, a $\mathbb{Z}_{m}$ lift centered at $(0,0)$ can be constructed as follows.  For each vertex $v$ in  a voltage graph $\tilde{H}$, construct $v_{0}$ arbitrarily in the lift, and then for $i = 1, \ldots, m-1$, construct $v_{i}$ to be the \emph{clockwise} rotation of $v_{0}$ by $\frac{2 \pi i}{m}$ around $(0,0)$ (counterclockwise rotation is typically used, but in this paper we use clockwise rotation to correspond to the left-to-right order of English strings used later in the paper). If vertices $v$ and $w$ are connected by an arrow from $v$ to $w$ labelled $a$ in $\tilde{H}$, then the lift contains the set of edges $v_{i}w_{i+a}$ for $i = 0, \ldots, m-1$, with all index arithmetic computed modulo $m$. If $\tilde{H}$ contains a loop from $w$ to $w$ labelled $a$, then the lift contains edges $w_{i}w_{i+a}$.  In all voltage graphs presented in this paper, unlabelled, undirected edges correspond to edges with label $0$ (that is, we suppress 0-labelled edges for clarity, and the direction of the edge is not relevant when the label is 0). We call the non-zero-labelled directed edges between two distinct vertices in $\tilde{H}$ the \emph{arrows} in the voltage graph, and a non-zero-labelled edge from a vertex to itself is called a \emph{loop}.

A number of classical snark constructions can be interpreted as arising from $\mathbb{Z}_{m}$ lifts of voltage graphs which are formed from constructing directed edges  from the dangling edges of various $k$-poles, although the particular language of voltage graphs does not appear to have been specifically used. The flower snarks \cite{Isa1975},  for instance,  can be constructed from a voltage graph with 4 vertices forming a claw of unlabelled edges, with a pair of alternately oriented arrows labelled 1 connecting two of the leaves of the claw, and a loop labelled 1 at the other leaf, following the description of flower snarks given in \cite{ClaEnt1983}.
The  Goldberg snarks \cite{Gol1981,Wat1989} are in fact the $\mathbb{Z}_{m}$ generalization of the Loupekine snarks on 22 vertices that we will describe below as $P_{\alpha}(m; 1,1,1)$ (see the description of these snarks given in \cite{FioRui2008}), and the Watkins and Szekeres snarks \cite{Wat1989} can be interpreted as coming from voltage graphs with two arrows (i.e., a 4-pole) constructed by modifying the Petersen graph.  The recent paper \cite{Hag2016} also presents a construction for snarks with high oddness which could be interpreted as corresponding to a voltage-graph-type construction, although he does not use that language.

Given a graph with $\mathbb{Z}_{m}$ rotational symmetry where the orbits (symmetry classes) of the vertices and edges under that rotational symmetry all have the same number of elements (a ``polycyclic'' graph; see e.g., \cite{BobPis2003}),  it is straightforward to construct a corresponding voltage graph, by selecting (arbitrarily) one vertex from each symmetry class to be the $0^{\text{th}}$ element and then constructing the voltage graph by assigning one node for each symmetry class of vertices and  recording the connections between the vertex-class-nodes using appropriately labelled arrows and unlabelled edges. For example, if  $v_{i}w_{i+a}$ is an edge of the graph for all $i = 0, \ldots, m-1$, then in the corresponding voltage graph we have vertices labelled $v$ and $w$ connected with a directed edge from $v$ to $w$  labelled $a$.

For later reference, we collect here some basic facts about voltage graphs and their lifts; see, e.g., \cite[Chapter 2]{GroTuc1987} and \cite[Section 3.5]{PisSer2013}.

\begin{prop}\label{thm:voltageFacts}
Suppose that $\tilde{H}$ is a voltage graph with $\mathbb{Z}_{m}$ lift $H$. 
\begin{enumerate}
\item If $s$ is an integer such that $\gcd(s, m) = 1$, then if $\tilde{H'}$ is formed from $\tilde{H}$ by multiplying all labels by $s$, and $H'$ is the corresponding $\mathbb{Z}_{m}$ lift, then $H \cong H'$.
\item If $\{\lambda_{1}, \ldots, \lambda_{k}\}$ is the collection of non-zero labels in $\tilde{H}$, and if for each $i$, $\gcd(m, \lambda_{i})>1$, then the lift $H$ is disconnected.
\item\label{liftLoop} If $\tilde{H}$ contains a loop of label $c$ and $\gcd(c,m) = t$, then in the lift $H$, the edges induced by the loop form $t$  cycles of length $\frac{m}{t}$. 
\item If $\tilde H$ has an arrow from $v$ to $w$ labelled $a$, then we can replace it with an arrow from $w$ to $v$ labelled $-a$ (that is, by the inverse of $a$ in the voltage group) and the resulting lift will be the same.

\end{enumerate}
\end{prop}

Let  $H$ be a 5-pole constructed by removing a $P_{3}$ from some cubic graph $\bar{H}$ (not necessarily a snark), in which the 5 semiedges of $H$ have labels  $A, B, C, B', A'$, as before, where edges $A, B$ and $A', B'$ were incident with the two endpoints (respectively) of the removed $P_{3}$ and $C$ was incident with the center vertex of the $P_{3}$. (See Figure \ref{fig:GenericVoltageGraphEx}.) Then we will label the corresponding endpoints of the five semiedges in $H$ as $A, A'$, $B, B'$ as well, and label the endpoint of $C$ in $H$ as $v$.

\begin{figure}[htbp]
\begin{center}
\ffigbox{
\begin{subfloatrow}[4]
\ffigbox{\caption{A cubic graph $\bar{H}$, with a path $P_{3}=swt$ shown dashed.}\label{}}{
\begin{tikzpicture}[scale=.8]
\drawOneBlobWithLabels{0}{white}{white}{gray, thin}{white}{white}

\node[below left=of A0, fill = white, vtx] (s) {$s$};
\node[below right=of A'0, fill = white, vtx] (t) {$t$};

\draw[ultra thick] (A0) -- node[draw, fill=white, inner sep=2, thin]{\tiny{A}} (s);
\draw[ultra thick]  (B0) to[bend right= 30] node[draw, fill=white, inner sep=2, thin]{\tiny{$B$}} (s);
\draw[ultra thick]  (B'0) to[bend left = 30]  node[draw, fill=white, inner sep=2, thin]{\tiny{$B'$}} (t);
\draw[ultra thick] (A'0) -- node[draw, fill=white, inner sep=2, thin]{\tiny{$A'$}} (t);
\draw[ultra thick, dashed, gray] (s) -- (z0)--(t);
\draw[ultra thick] (C0) --node[draw, fill=white, inner sep=2, thin]{\tiny{$C$}} (z0);

\end{tikzpicture}}

\ffigbox{\caption{A 5-pole $H$ constructed by removing a path $P_{3}$ from a cubic graph $\bar{H}$}\label{5multipole}}{
\begin{tikzpicture}
\drawOneBlobWithLabels{0}{white}{white}{gray, thin}{white}{white}
\node[left=of A0] (a0) {};
\node[left=of B0] (b0) {};
\node[right=of A'0] (a'0) {};
\node[right=of B'0] (b'0) {};

\draw[ultra thick] (A0)--(a0);
\draw[ultra thick] (B0)--(b0);
\draw[ultra thick] (A'0)--(a'0);
\draw[ultra thick] (B'0)--(b'0);

\draw[->, ultra thick] (z0.west) arc(90:360+50:.4) node[midway, arclabel] {$c$};
\draw[] (C0) -- (z0);


\end{tikzpicture}
}
\ffigbox{\caption{A schematic of the voltage graph corresponding to the $\alpha$-connection $\tilde{H}_{\alpha}(a, b, c)$}\label{Galpha}}{
\begin{tikzpicture}
\drawOneBlobWithLabels{0}{white}{white}{gray, thin}{white}{white}

%
%

\draw[<-, ultra thick] (A0) to[->, out=180+45, in = -45, looseness=3] node[arclabel, near start] {$a$} (A'0); 

\draw[<-, ultra thick] (B0)  to[->, out=180-45, in = 45, looseness=3] node[arclabel, midway] {$b$} (B'0) ; 

\draw[->, ultra thick] (z0.west) arc(90:360+50:.4) node[midway, arclabel] {$c$};
\draw[] (C0) -- (z0);

\end{tikzpicture}

}

\ffigbox{\caption{A schematic of the voltage graph corresponding to the $\beta$-connection $\tilde{H}_{\beta}(a, b, c)$}\label{Gbeta}}{
\begin{tikzpicture}
\drawOneBlobWithLabels{0}{white}{white}{gray, thin}{white}{white}
\draw[<-, ultra thick] (A0) --node[circle, arclabel, near end, inner sep = 1 pt] {$a$} (B'0); 

\draw[<-, ultra thick] (B0)  -- node[arclabel, near end,  inner sep = 1 pt] {$b$} (A'0) ; 

\draw[->, ultra thick] (z0.west) arc(90:360+50:.4) node[midway, arclabel] {$c$};
%

\draw[] (C0) -- (z0);

\end{tikzpicture}
}

\end{subfloatrow}
}{
\caption{A generic 5-pole $H$ constructed by removing a path $P_{3}$ from a cubic graph $\bar{H}$, and the $\alpha$-connection and $\beta$-connection voltage graphs which we construct from the 5-pole. An addtional vertex and loop is added at the end of semiedge $C$. }
\label{fig:GenericVoltageGraphEx}
}
\end{center}
\end{figure}
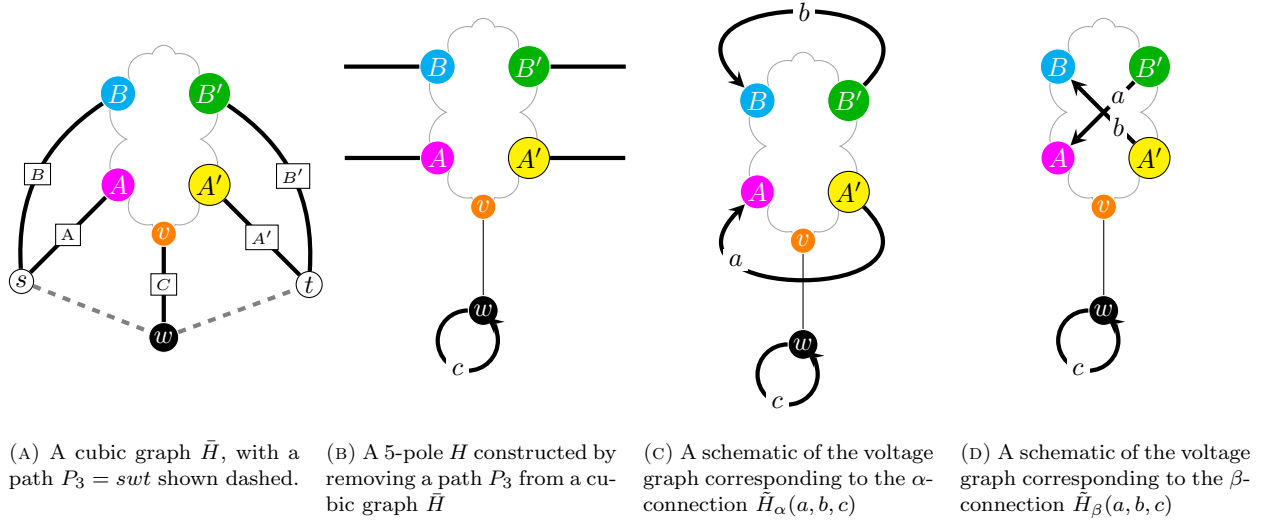

From $H$, we can construct two voltage graphs, called $\tilde{H}_{\alpha}(a,b,c)$ and $\tilde{H}_{\beta}(a,b,c)$.  To construct $\tilde{H}_{\alpha}$, called the  \emph{$\alpha$-connection}, we introduce an arrow labelled $a$ from $A'$ to $A$ and an arrow labelled $b$ from $B'$ to $B$. Add an endpoint  $w$ at the other end of semiedge $C$, and add a loop at $w$ labelled $c$. In symbols, $A \xrightarrow{a} A'$ and $B \xrightarrow{b} B'$.    To construct  $\tilde{H}_{\beta}(a,b,c)$, we follow the same instructions, but the two arrows connect $B' \xrightarrow{a} A$ and $A' \xrightarrow{b} B$. Figure \ref{Galpha} and \ref{Gbeta} show schematics of such voltage graphs. Note that the same initial graph $\bar{H}$ can generate non-isomorphic 5-poles $H$ depending on which path is deleted, and for a fixed 5-pole $H$, the corresponding $\alpha$- and $\beta$- voltage graphs may or may not be isomorphic. (Additionally, there is a choice as to which semiedges are labelled what when the $P_{3}$ is removed from the original graph, but this basically corresponds to constructing the $\alpha$- or $\beta$-connections.) 
To ensure that the resulting lift is actually cubic, we require that $1 \leq a, b < m$ and $1 \leq c <m/2$. 

Note that Loupekine's construction as described above in Algorithm \ref{alg:Loupekine} is more general than this, since he allows the edges labelled $C_{i}$ to be connected up arbitrarily, while we force them to be ``spoke edges'' all incident with a central collection of concentric congruent polygons (a single $m$-gon when $\gcd(c,m) = 1$). On the other hand, Loupekine's construction assumed that each cluster was connected by pairs to the next one, and the voltage graph construction provides for other choices of skips.

\begin{definition}A  \emph{pseudo-Loupekine graph} is any graph constructed as a $\mathbb{Z}_{m}$ lift of $\tilde{H}_{\alpha}(a,b,c)$ or $\tilde{H}_{\beta}(a,b,c)$, where $H$ is a 5-pole constructed by removing a $P_{3}$ from any cubic graph $\bar{H}$ (not necessarily a snark).
\end{definition}

For convenience, we define the following parts of the graph $H_{x}(m; a,b,c)$, $x\in \{\alpha, \beta\}$. For a fixed integer $j$, the \emph{cluster} $H_{j}$ is the subgraph of the lift induced by the vertices in the lift that all have index $j$ (represented by the cloud shape in the schematic drawings, e.g., Figure \ref{fig:GenericVoltageGraphEx}). The \emph{spoke edges} are the edges $v_{i}w_{i}$. The \emph{loop edges} are the edges induced by the loop from $w$ to itself; that is, each loop edge is of the form $w_{i}w_{i+c}$. The \emph{connecting edges} are induced by the arrows in the graph; they are of the form $A'_{i}A_{i+a}$ and $B'_{i}B_{i+b}$ in an $\alpha$-connection, and of the form $A'_{i}B_{i+b}$ and $B'_{i}A_{i+a}$ in a $\beta$-connection. (Note that vertex classes $A, B, A', B'$ are labelled in accordance with the labels of the corresponding dangling edges.) 
In all pictures in this paper, the voltage graph vertices $A, B, A', B', v, w$ and the corresponding symmetry classes of vertices $A_{i}, B_{i}, A'_{i}, B'_{i}, v_{i}, w_{i}$ are colored magenta, dark cyan, yellow, dark green, orange, and black. Other vertices in the voltage graphs or in the lifts, which are not connected to arrows or spokes in the voltage graph, are colored gray.

Of particular interest to us are graphs $H_{x}(m; a,a,c)$.
\begin{observation}\label{clusterCycle} In $H_{x}(m; a,a,c)$, the pairs of connecting edges  and the clusters together form \emph{cluster-cycles},  where all the elements of the cluster $H_i$ are considered as  ``vertices'' and the ``edges'' of the cycle are the pairs of edges between $H_{i}$ and $H_{i+a}$. If $\gcd(a,m) = t$, then the connecting edges form $t$ cluster-cycles of length $m/t$.
\end{observation}

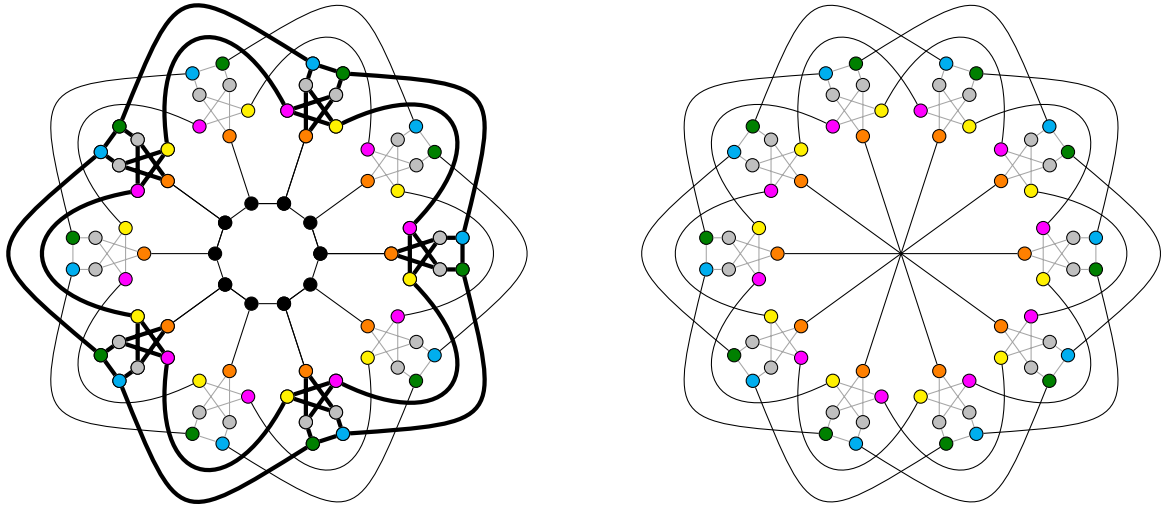
\begin{figure}[htbp]
\begin{center}
\ffigbox{
\begin{subfloatrow}
\ffigbox{\caption{The graph $P_{\alpha}(10; 2,2,1)$. Since $\gcd(10,2) = 2$, the graph has two cluster-cycles, both of length 5; one is shown with thick lines.}
\label{fig:newLSex}}{
\begin{tikzpicture}[scale=1]
\drawPentaBlobs[.7]{10}{.7}{.3}
\foreach \i in {0, 1, ..., 9}{
	\drawAlphaAwiggle[-80]{10}{\i}{2}{black}{2.5}
	\drawAlphaBwiggle[-50]{10}{\i}{2}{black}{2}
	\drawSpoke{\i}{black}
	\drawLoop{10}{\i}{black}{1}
}
\foreach \i in {0,  2, 4, 6, 8}{
	\drawPentaBlobOne[.7]{10}{.7}{.3}{\i}{ultra thick, black}
	\drawAlphaAwiggle[-80]{10}{\i}{2}{ultra thick,black}{2.5}
	\drawAlphaBwiggle[-50]{10}{\i}{2}{ultra thick,black}{2}
	\drawSpoke{\i}{ }
	\drawLoop{10}{\i}{ }{1}
}
\end{tikzpicture}
}
\ffigbox{\caption{A second graph, where the spoke and loop edges are replaced by ``diameters'', corresponding to a semiedge labelled $m/2$ (in this case, 5) in the voltage graph}\label{fig:newLSexDiams}}{
\begin{tikzpicture}[scale=1]
\drawPentaBlobsNoBlack[.7]{10}{.7}{.3}
\foreach \i in {0, 1, ..., 9}{
	\drawAlphaAwiggle[-80]{10}{\i}{2}{black}{2.5}
	\drawAlphaBwiggle[-50]{10}{\i}{2}{black}{2}
}
\foreach \i in {0,1,2,3,4}{ \draw let \n1 = {int(mod(\i + 5,10))} in (C\i) -- (C\n1);}
\end{tikzpicture}

}
\end{subfloatrow}
}{
\caption{Two graphs with $10$-fold rotational symmetry, constructed from the 5-pole $P$ formed by removing a path  of length 2 from the Petersen graph. We form the $\alpha$-connection with a skip of 2, and connect up the inner spoke edges in various systematic ways. We will later show these are snarks, in Proposition \ref{thm:LoupFam}.}\label{}
}
\end{center}
\end{figure}

Given any voltage graph with two arrows plus a spoke terminating in a loop, it is straightforward to reconstruct the original graph which had had a $P_{3}$ removed from it. First, label, respectively, the head and tail vertex of the first arrow as $A'$ and $A$ and the head and tail of the second arrow as $B'$ and $B$ (this corresponds to clipping the arrows in half to form four of the semiedges of the 5-pole $H$). 
Remove the two arrows and the loop (but not their incident vertices), and add two new vertices to the graph, called $s$ and $t$. Finally add edges $As, Bs, A't, B't, vw,$ and $sw$ and $wt$ to finish reconstructing the original graph. We will use this process in Sections \ref{Sec:LoupConstr} and \ref{sec:newSnarks}.

\subsection*{\bf Example}
 Let $P$ be the 5-pole formed by removing a path $P_{3}$ from the Petersen graph (Figure \ref{fig:Ex-Petersen}). The two Loupekine Snarks on 22 vertices can be interpreted as $P_{\alpha}(3; 1,1,1)$ and $P_{\beta}(3;1,1,1)$, shown in Figure \ref{Fig:LoupSnarkAsLift}, if the central triangle is contracted to a point.

 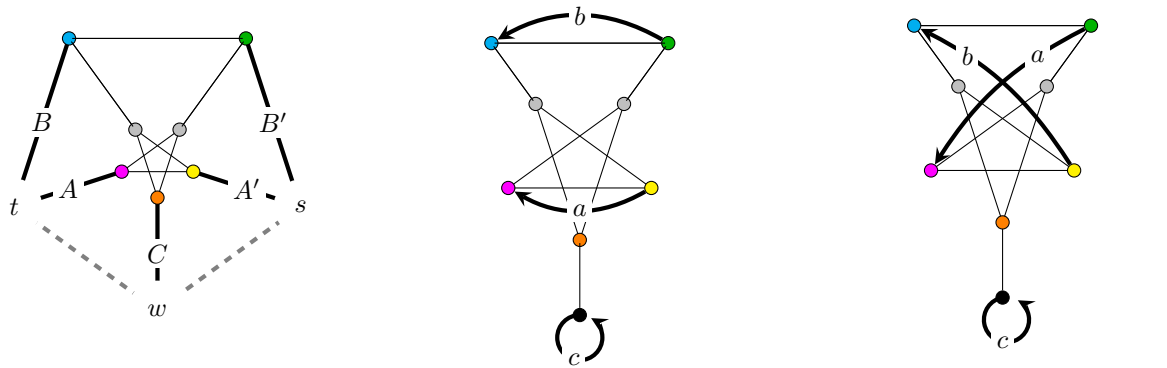
\begin{figure}[htbp]
\begin{center}
\ffigbox{
\begin{subfloatrow}[3]
\ffigbox{\caption{A 5-pole $P$ formed by removing a path $P_{3}$ (dashed) from the Petersen graph $\bar{P} $} \label{fig:Ex-Petersen}}{
\begin{tikzpicture}[arclabel/.style={midway, fill=white, inner sep = 2 pt}]
\foreach \i in {0, ...,4}{
	\node[inner sep=1pt] (v\i) at (  72*\i-72/4:.5){};
	\node[inner sep=1pt] (w\i) at ( 72*\i-72/4:2){};%
}
\foreach \i in {0,..., 4}{
	\draw[] let \n1={int(mod(\i+2, 5))} in  (v\i) -- (v\n1);}
\foreach \i in {1}{
	\draw let \n1={int(mod(\i+1, 5))} in  (w\i) -- (w\n1);}
\foreach \i in {1,2}{\draw (v\i) -- (w\i);}

\node[vtx, fill = mymagenta]  (A) at (v3){};
\node[vtx, fill = cyan] (B) at (w2){};
\node[vtx, fill = yellow] (A') at (v0){};
\node[vtx, fill = green!70!black] (B') at (w1){};
\node[vtx, fill=black](C) at (w4){};

\draw  (v1) node[vtx, fill=gray!50!white]{};
\draw (v2) node[vtx, fill=gray!50!white]{};
\draw  (v4) node[vtx, fill=orange]{};

\begin{scope}[on background layer]
\draw[ultra thick] (w2) -- node[arclabel]{$B$} (w3);
\draw[ultra thick] (v3) --  node[arclabel]{$A$} (w3);
\draw[ultra thick] (w1) -- node[arclabel]{$B'$} (w0);
\draw[ultra thick] (v0) --  node[arclabel]{$A'$} (w0);
\draw[ultra thick] (v4) --  node[arclabel]{$C$} (w4);
\end{scope}

\draw[ultra thick, dashed, gray]  (w3) -- (w4);
\draw[ultra thick, dashed, gray]  (w4) -- (w0);


\draw (w0) node[fill =white, circle, inner sep = 5 pt]{$s$};
\draw (w3) node[fill =white, circle, inner sep = 5 pt]{$t$};
\draw (w4) node[fill =white, circle, inner sep = 5 pt]{$w$};

\draw[] (v1) -- (w1) -- (w2) -- (v2);
\end{tikzpicture}
}

\ffigbox{\caption{The voltage graph $\tilde{P}_{\alpha}(a,b,c)$}\label{}}{

\begin{tikzpicture}[arclabel/.style={midway, fill=white, inner sep = 2 pt}]
\foreach \i in {0, ...,4}{
	\node[inner sep=1pt] (v\i) at (  72*\i-72/4:1){};
	\node[inner sep=1pt] (w\i) at ( 72*\i-72/4:2){};%
}
\foreach \i in {0,..., 4}{
	\draw[] let \n1={int(mod(\i+2, 5))} in  (v\i) -- (v\n1);}
\foreach \i in {1}{
	\draw let \n1={int(mod(\i+1, 5))} in  (w\i) -- (w\n1);}
\foreach \i in {1,2}{\draw (v\i) -- (w\i);}

\node[vtx, fill = mymagenta]  (A) at (v3){};
\node[vtx, fill = cyan] (B) at (w2){};
\node[vtx, fill = yellow] (A') at (v0){};
\node[vtx, fill = green!70!black] (B') at (w1){};
\node[vtx, fill=black](C) at (w4){};

\draw  (v1) node[vtx, fill=gray!50!white]{};
\draw (v2) node[vtx, fill=gray!50!white]{};
\draw  (v4) node[vtx, fill=orange]{};


\draw[] (v1) -- (w1) -- (w2) -- (v2);
\draw (v4) -- (C);

\draw[ultra thick,<-] (A) to[out=-30, in = 180+30] node[near start, arclabel]{$a$} (A');
\draw[ultra thick, <-] (B) to[out=+30, in = 180-30] node[midway, arclabel]{$b$} (B');
\draw[ultra thick, ->] (C) arc(90:360+60:.3) node[midway, arclabel]{$c$} (C);
\end{tikzpicture}

}

\ffigbox{\caption{The voltage graph $\tilde{P}_{\beta}(a,b,c)$}\label{}}{
\begin{tikzpicture}[arclabel/.style={midway, fill=white, inner sep = 2 pt}]
\foreach \i in {0, ...,4}{
	\node[inner sep=1pt] (v\i) at (  72*\i-72/4:1){};
	\node[inner sep=1pt] (w\i) at ( 72*\i-72/4:2){};%
}
\foreach \i in {0,..., 4}{
	\draw[] let \n1={int(mod(\i+2, 5))} in  (v\i) -- (v\n1);}
\foreach \i in {1}{
	\draw let \n1={int(mod(\i+1, 5))} in  (w\i) -- (w\n1);}
\foreach \i in {1,2}{\draw (v\i) -- (w\i);}

\node[vtx, fill = mymagenta]  (A) at (v3){};
\node[vtx, fill = cyan] (B) at (w2){};
\node[vtx, fill = yellow] (A') at (v0){};
\node[vtx, fill = green!70!black] (B') at (w1){};
\node[vtx, fill=black](C) at (w4){};

\draw  (v1) node[vtx, fill=gray!50!white]{};
\draw (v2) node[vtx, fill=gray!50!white]{};
\draw  (v4) node[vtx, fill=orange]{};

\begin{scope}[on background layer]
\draw[] (v1) -- (w1) -- (w2) -- (v2);
\draw (v4) -- (C);

\draw[ultra thick,<-] (A) to[bend left=15] 
node[midway, arclabel, near end]{$a$} (B');
\draw[ultra thick, <-] (B) to[bend left=15] 
node[midway, arclabel, near start]{$b$} (A');
\draw[ultra thick, ->] (C) arc(100:360+60:.3) node[midway, arclabel]{$c$} (C);
\end{scope}
\end{tikzpicture}

}
\end{subfloatrow}
}{
\caption{Removing a path $P_{3}$ from the Petersen graph $\bar{P}$ produces a 5-pole $P$ which generates the $\tilde{P}_{\alpha}(a,b,c)$ and $\tilde{P}_{\beta}(a,b,c)$ voltage graphs. }
\label{default}}
\end{center}
\end{figure}

\begin{figure}[htbp]
\begin{center}
\ffigbox{
\begin{subfloatrow}
\ffigbox{\caption{$P_{\alpha}(3; 1,1,1)$, which is isomorphic to the first Loupekine Snark on 22 vertices when the central triangle is contracted to a point.}\label{fig:Loup22Snark1}}{
\begin{tikzpicture}
\drawPentaBlobs{3}{.2}{.5}
\foreach \i in {0,1,2}{
\drawAlphaA[.5]{3}{\i}{1}{black}{0}
\drawAlphaB[.5]{3}{\i}{1}{black}{0}
\drawSpoke{\i}{black}
\drawLoop{3}{\i}{black}{1}
}
\end{tikzpicture}
}

\ffigbox{\caption{$P_{\beta}(3; 1,1,1)$, which is isomorphic to the second Loupekine Snark on 22 vertices when the central triangle is contracted to a point.}\label{fig:Loup22Snark2}}{
\begin{tikzpicture}
\drawPentaBlobs{3}{.2}{.5}
\foreach \i in {0,1,2}{
\drawBetaA[.5]{3}{\i}{1}{black}{0}
\drawBetaB[.5]{3}{\i}{1}{black}{0}
\drawSpoke{\i}{black}
\drawLoop{3}{\i}{black}{1}
}
\end{tikzpicture}
}

\end{subfloatrow}

}{
\caption{The Loupekine Snarks on 22 vertices are $\mathbb{Z}_{3}$ lifts of the voltage graphs $\tilde{P}_{\alpha}(1,1,1)$ and $\tilde{P}_{\beta}(1,1,1)$, after the central triangle is contracted to a point. }
\label{Fig:LoupSnarkAsLift}
}
\end{center}
\end{figure}
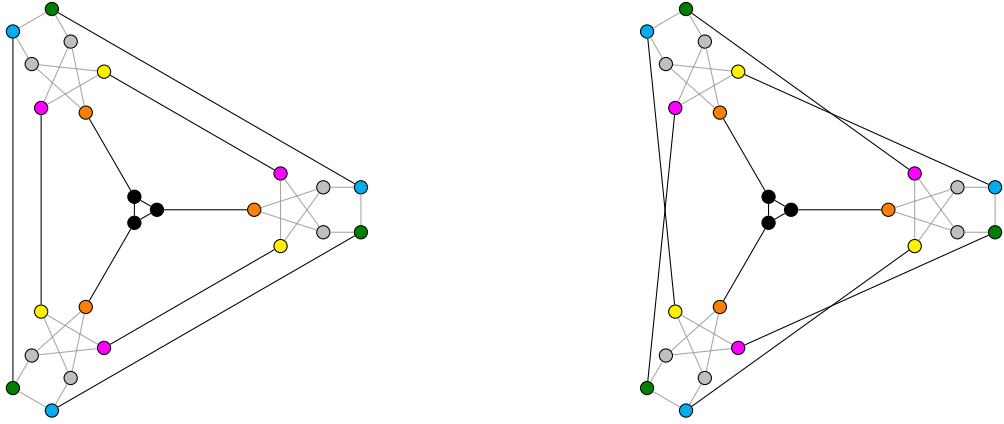

Figure \ref{fig:newLSex} shows the snark $P_{\alpha}(10; 2,2,1)$.
The Goldberg snarks (see \cite{FioRui2008}) are $P_{\alpha}(m; 1,1,1)$ for odd $m$. 

\subsection{\bf More Loupekine snarks}\hfill

Loupekine's argument essentially reduces to the following observation:

\begin{observation}\label{LoupekineSnarks} Any odd cluster-cycle of Loupekine 5-poles $S$ (i.e., formed from a snark 
minus a path of two edges, $P_{3}$) is a snark.
\end{observation}

We can construct more Loupekine snarks if we allow the clusters to connect to each other in a more complicated way than simply one cluster to the next one. That is:

\begin{prop}\label{thm:LoupFam} Suppose that $\gcd(m, a,  c) = 1$ (for connectivity),  $r = m/\gcd(a,m)$ is odd, and  $S$ 
is constructed by removing a $P_{3}$ from a snark  $\bar{S}$. Then $\mc{S} =S_{x}(m; a,a,c)$  for $x = \alpha,\beta$ is a snark. If the girth of $S$ is at least 4,  $3 \neq m/\gcd(m,c)$, and $r \neq 3$, then the snark $\mc{S}$ has girth at least 4 as well.
\end{prop}

\begin{proof} 
By Observation \ref{clusterCycle}, the connecting edges form $\gcd(s,m)$ cluster-cycles each of odd-length $r$.
Observation \ref{LoupekineSnarks} shows that this cycle cannot be coherently 3-edge-colored. Hence $\mc{S}$ is a snark.

By construction, each cluster $S_{i}$ has girth at least 4, so the graph $\mc{S}'$ constructed by deleting the spoke and loop edges from $\mc{S}$ has girth at least 4. Thus, the girth of $\mc{S}$ is determined by the size of the cycles induced by the loop edges. However, by Proposition \ref{thm:voltageFacts}(\ref{liftLoop}), the loop edges form $\gcd(m, c)$ cycles of length  $m/\gcd(m,c)$. Thus, whenever $3 < m/\gcd(m,c)$, $\mc{S}$ has girth greater than 3. If $m = 3$, the loop edges form a triangle, so $\mc{S}$ has girth 3. (However, in this case, the triangle can be contracted to a single point, and the result is still a snark, now with girth at least 4.)
\end{proof}

\begin{definition} A \emph{cyclic Loupekine snark} is a pseudo-Loupekine snark of the form $S_{x}(m;a,a,c)$ for $x = \alpha, \beta$, where $S$ is constructed from a Loupekine 5-pole, that is, by removing a path of length 2 from a snark $\bar S$. \end{definition}

It appears that this easy generalization of Loupekine's construction has not been discussed in this form in the literature.  Figure \ref{fig:newLSex} shows $P_{\alpha}(10; 2,2,1)$ (where the initial snark $P$ is the Petersen graph), which has $m$-fold rotational symmetry for an even $m$. (Of course, the original descriptions of Loupekine's construction \cite{Isa1976,Wat1989} allowed for constructing snarks using an even number of copies of a given 5-pole, but complete details on how to do this were not provided.)

Observe that simply by choosing some even number $m$ and parameter $a$ such that $m/gcd(m,a)$ is odd, then we obtain a snark with rotational symmetry for even $m$. For example, $P_{\alpha}(5(2k); 2,2,1)$ yields  a snark for every $k$. That is:

\begin{corollary}There are infinitely many snarks which can be drawn with $\mathbb{Z}_{m}$ rotational symmetry for even $m$.
\end{corollary}

Martin \v{S}koviera suggested\footnote{personal communication, December 2, 2016} an additional infinite class of snarks which can be formed from the snarks $S_{\alpha}(2k; a, a, 1)$, by replacing the inner ring of points $w_{i}$ and the inner cycle formed by the loop edges with ``diameters'' connecting $v_{i}$ and $v_{i+k}$; this corresponds to replacing the bottom spoke and loop in the voltage graph with a single semiedge labelled $k$, where the voltage group is $\mathbb{Z}_{2k}$. Figure \ref{fig:newLSexDiams} shows an example of this type of snark, based on $P_{\alpha}(10; 2,2,1)$. Preliminary analysis suggests that these snarks may have very interesting properties (e.g., the snark in Figure \ref{fig:newLSexDiams} is cyclically 5-connected and irreducible).

\begin{conjecture}
All cyclic pseudo-Loupekine graphs $S_{x}(m; a,a,c)$, where $S$ is a Loupekine 5-pole and where $m/\gcd(m,a)$ is even and $m/\gcd(m,c)$ is odd, are 3-edge-colorable.
\end{conjecture}

\begin{conjecture}
All cyclic pseudo-Loupekine graphs $S_{x}(m; a,b,c)$ where $S$ is a Loupekine 5-pole and $a \neq b$ are 3-edge-colorable.
\end{conjecture}

\section{The three snarks on 28 vertices with 3-fold rotation}
We analyzed all the snarks listed in the House of Graphs \cite{HoG} on at most 32 vertices whose automorphism group is divisible by 3. In addition to the Petersen graph and the two Loupekine snarks on 22 vertices, there are three other  snarks, all on 28 vertices, which admit realizations with 3-fold rotational symmetry; they are listed in 
\ref{appendix:28}. These snarks are shown in Figure \ref{fig:28snark}. Note that each of them has a fixed point in the center of the graph; if we expand this fixed point to a central triangle, then we can construct a voltage graph over $\mathbb{Z}_{3}$ from the embedding. In all three cases, the voltage graph has two arrows, plus a spoke terminating in a loop.

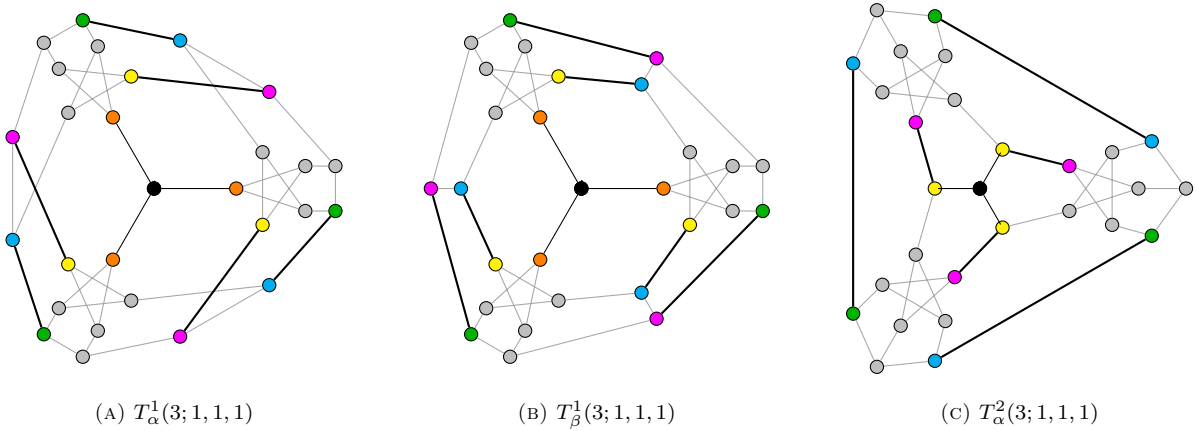
\begin{figure}[htbp]
\begin{center}
\ffigbox{
\begin{subfloatrow}[3]

\ffigbox{\caption{$T^{1}_{\alpha}(3;1,1,1)$}\label{}}{
\begin{tikzpicture}[scale=.8]
\drawGTwentyEightAlphaBlobs{3}{0}{.5}
\foreach \i in {0,1,2}{
\drawAlphaA[0]{3}{\i}{1}{black, thick}{0}
\drawAlphaB[0]{3}{\i}{1}{black, thick}{0}
\drawSpoke{\i}{black}
\drawLoop{3}{\i}{black}{1}
}
\end{tikzpicture}
}

\ffigbox{\caption{$T^{1}_{\beta}(3;1,1,1)$}\label{}}{
\begin{tikzpicture}[scale=.8]
\drawGTwentyEightBetaBlobs{3}{0}{.5}
\foreach \i in {0,1,2}{
\drawBetaA[0]{3}{\i}{1}{black,  thick}{0}
\drawBetaB[0]{3}{\i}{1}{black, thick}{0}
\drawSpoke{\i}{black}
\drawLoop{3}{\i}{black, thick}{1}
}
\end{tikzpicture}

}

\ffigbox{\caption{$T^{2}_{\alpha}(3;1,1,1)$}\label{G28NEW}}{
\begin{tikzpicture}[scale=.8]
\drawGTwentyEightNEWBlobs{3}{0}{.5}
\foreach \i in {0,1,2}{
\drawAlphaA[0]{3}{\i}{1}{black, thick}{0}
\drawAlphaB[0]{3}{\i}{1}{black, thick}{0}
\drawSpoke{\i}{black}
\drawLoop{3}{\i}{black}{1}
}
\end{tikzpicture}
}

\end{subfloatrow}
}{
\caption{The three snarks on 28 vertices which can be drawn with three-fold rotational symmetry. Their labels correspond to their construction using the original graph and voltage graphs shown in Figure \ref{fig:G28-orig-volt}.}
\label{fig:28snark}}
\end{center}
\end{figure}

\begin{figure}[htbp]
\begin{center}
\ffigbox{
\begin{subfloatrow}[3]

\ffigbox{\caption{The voltage graph $\tilde{T}^1_{\alpha}(m; a,b,c)$}\label{fig:G28AlphaVolt}}{
\begin{tikzpicture}
\drawOneGTwentyEight	
	\draw[<-, ultra thick] (A1) to[bend left = 40, looseness=1.5] node[arclabel, midway]{$a$} (A'1);
	\draw[<-, ultra thick] (B1) to[bend left = 80, looseness=1.5] node[arclabel, midway]{$b$} (B'1);
	\begin{scope}[on background layer]
	\draw[->, ultra thick] (z1) arc(100:360+50:.3)node[midway, arclabel]{$c$} (z1);
	\end{scope}
	\path[left of = w] node{$\mathbb{Z}_{m}$};
\end{tikzpicture}}

\ffigbox{\caption{The voltage graph $\tilde{T}^1_{\beta}(a,b,c)$}\label{fig:G28BetaVolt}}{

\begin{tikzpicture}
\drawOneGTwentyEight
		
	\draw[<-, ultra thick] (A1) to[bend left = 60, looseness=1.5] node[arclabel, midway]{$a$} (B'1);
	\draw[<-, ultra thick] (B1) to[bend right = 80, looseness=1.5] node[arclabel, near start]{$b$} (A'1);
	\begin{scope}[on background layer]
	\draw[->, ultra thick] (z1) arc(100:360+50:.3)node[midway, arclabel]{$c$} (z1);
	\end{scope}
	
\end{tikzpicture}

}

\ffigbox{\caption{The original graph $\bar{T}$ for both voltage graphs is the Petersen graph with one vertex replaced with a triangle, minus a $P_{3}$; the removed path is shown dashed, yielding the 5-pole $T^1$.}\label{fig:Ptri1}}{
\begin{tikzpicture}
\drawOneGTwentyEight
	\node[vtx, fill=white,   left = .5  of B1] (s) {$s$};
	\node[vtx, fill = white, right = of A'1] (t) {$t$};
	\draw[ultra thick, dashed, gray] (s) -- (z1) -- (t);
	\draw[ultra thick] (A1)--(s) (B1)--(s);
	\draw[ultra thick] (A'1)--(t) (B'1)--(t);
\end{tikzpicture}

}

\end{subfloatrow}

\begin{subfloatrow}[2]

\ffigbox{\caption{The voltage graph for $\tilde{T}^2_{\alpha}(a,b,c) \cong \tilde{T}^2_{\beta}(a,b,c) $}\label{fig:G28newVolt}}{
\begin{tikzpicture}
\drawOneGTwentyEightNEW
	\draw[<-, ultra thick] (A1) to[bend right = 30] node[ arclabel]{$a$} (A'1);
	\draw[<-, ultra thick] (B1) to[bend left = 20] node[arclabel, near start, ]{$b$}(B'1);
	\begin{scope}[on background layer]
	\draw[->, ultra thick] (z1) arc(100:360+50:.3)node[midway, arclabel]{$c$} (z1);
	\end{scope}
	\end{tikzpicture}

}

\ffigbox{\caption{The original graph $\bar{T}$ is again the Petersen graph with one vertex replaced with a triangle, minus a $P_{3}$, but this time the $P_{3}$ removes one side of the triangle, yielding the 5-pole $T^2$. The removed path is shown dashed.}\label{fig:Ptri2}}{

\begin{tikzpicture}
\drawOneGTwentyEightNEW
	\node[vtx, fill=white,   below left = .7  of A1] (s) {$s$};
	\node[vtx, fill = white, right = of A'1] (t) {$t$};
	\draw[ultra thick, dashed, gray] (s) -- (z1) -- (t);
	\draw[ultra thick] (A1)--(s) (B1)--(s);
	\draw[ultra thick] (A'1)--(t) (B'1)--(t);
	
\end{tikzpicture}
}

\end{subfloatrow}

}{
\caption{Voltage graphs and original graphs for $\tilde{T}^1_{\alpha}(a,b,c)$, $\tilde{T}^1_{\beta}( a,b,c)$ and $\tilde{T}^2_{\alpha}(a,b,c)$ }
\label{fig:G28-orig-volt}}
\end{center}
\end{figure}
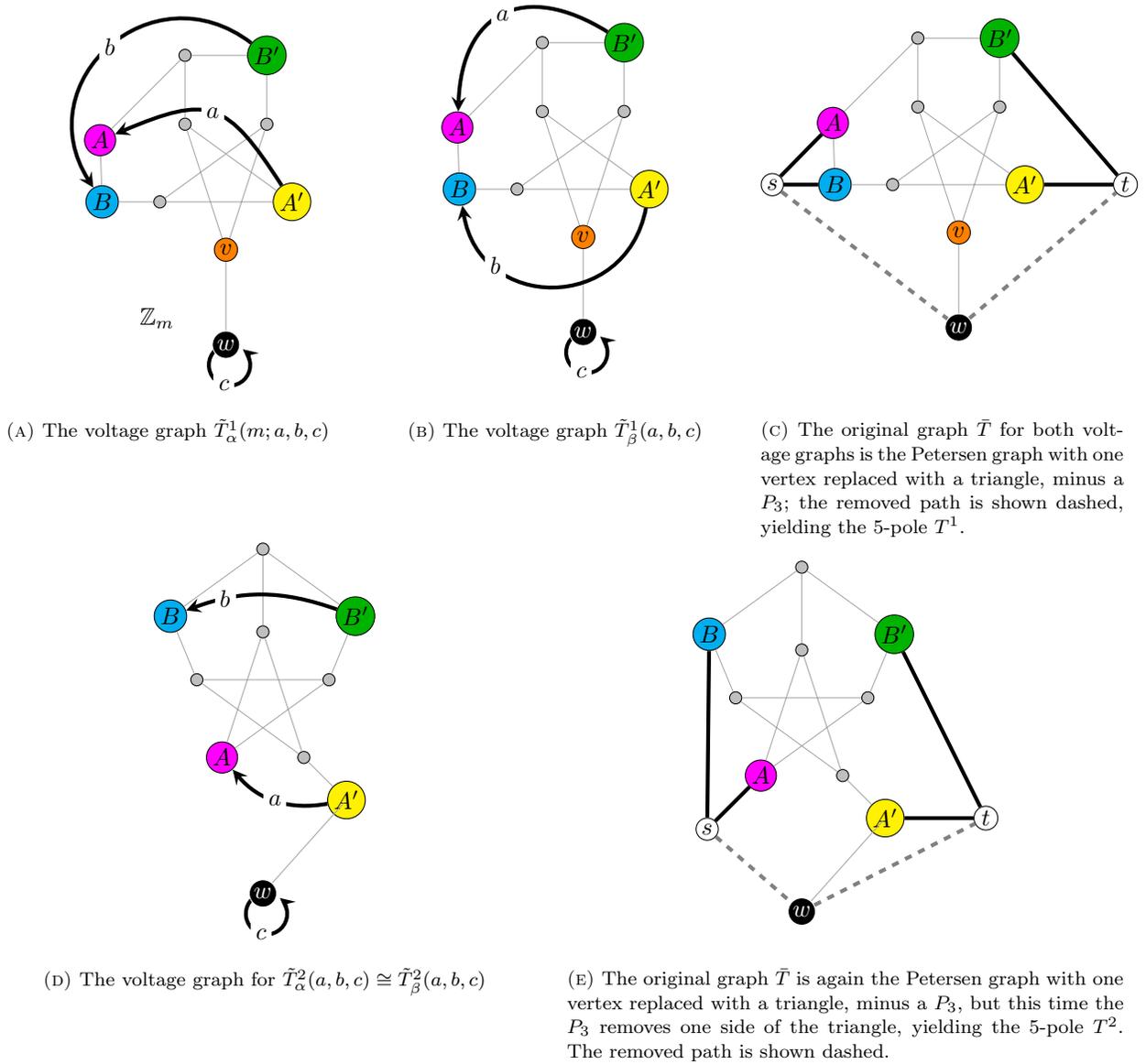

In all three cases, the original graph  that is reconstructed from the clipped voltage graph is the (improper) snark $\bar{T}$ formed by replacing a single vertex of the Petersen graph with a triangle. 
Up to isomorphism, there are two ways to delete a $P_{3}$ from this graph so that the resulting 5-pole has girth 5 (that is, in a triangle-destroying way); either a vertex of the triangle is an endpoint of the $P_{3}$, shown in Figure \ref{fig:Ptri1}, or an edge of the triangle is one of the edges of the $P_{3}$, shown in Figure \ref{fig:Ptri2}. We  name the resulting 5-poles $T^1$ and $T^2$ respectively;  the corresponding voltage graphs  $\tilde{T}^{1}_{\alpha}(a,b,c)$ and $\tilde{T}^{1}_{\beta}(a,b,c)$ are shown in Figures \ref{fig:G28AlphaVolt} and \ref{fig:G28BetaVolt}, and $\tilde{T}^{2}_{\alpha}(a,b,c)$ is shown in Figure \ref{fig:G28newVolt}.

Straightforward analysis shows that the graphs \[T^1_{\alpha}(3; 1,1,1) \not\cong T^1_{\beta}(3; 1,1,1)\not\cong T^2_{\alpha}(3; 1,1,1)\]
but that $T^2_{\alpha}(3; 1,1,1) \cong T^2_{\beta}(3; 1,1,1)$ (since it is easy to show that the voltage graphs $\tilde{T}^2_{\alpha}(a,b,c) \cong  \tilde{T}^2_{\beta}(a,b,c)$).

\begin{prop} The families $T^{1}_{\alpha}(m; a,a,c)$, $T^{1}_{\beta}(m; a,a,c)$, $T^{2}_{\alpha}(m; a,a,c)$, where $m/gcd(m, a)$ is odd and $m\geq 2$, form three infinite families of snarks of order $10m$, which all can be drawn with $m$-fold rotational symmetry.\end{prop}
\begin{proof} This follows immediately from Proposition \ref{thm:LoupFam} and the voltage graphs.\end{proof}

These three snarks have interesting properties.

%
%

\begin{prop}
The snarks $T^{2}_{\alpha}(m; a,a,c)$ are cyclically 4-connected for $m >3$.
\end{prop}

\begin{proof}
It is easy to see a cut set of size 4 that disconnects a cycle in $T^{2}_{\alpha}(m; a,a,c)$. \end{proof}

The snarks $T^{1}_{\alpha}(3; 1,1,1)$ and $T^{1}_{\beta}(3; 1,1,1)$, after the central triangle is contracted to a point, are cyclically 5-connected. Checking specific examples shows that the snarks $T^{1}_{\alpha}(m; 1,1,1)$ and $T^{1}_{\beta}(m; 1,1,1)$ for $m = 5, a = 1, 2, c = 1, 2$ and $m = 7, a = 1, 2, 3, c = 1,2,3$ are cyclically 5-connected as well. We conjecture:

\begin{conjecture}
If $m/\gcd(m, c) > 4$, then the infinite families of snarks $T^{1}_{\alpha}(m; a,a,c)$ and $T^{1}_{\beta}(m; a,a,c)$ are cyclically 5-connected.
\end{conjecture}

A complete proof of this result is beyond the scope of this paper.

The cyclically 4-connected snark on 28 vertices shown in Figure \ref{G28NEW}, corresponding to $T^{2}_{\alpha}(3; 1,1,1)$ with the central triangle contracted to a point, appeared (in a much less symmetric drawing) as Figure 7 of \cite{MacRas2006}, where M\'a{\v{c}}ajov\'a and Raspaud showed that it is the smallest example of a graph (other than the Petersen graph) with circular flow number of 5, serving as a counterexample to  Bohjan Mohar's Strong Circular 5-flow Conjecture.

The \emph{oddness} of a cubic graph is defined to be the smallest number of odd cycles in any 2-factor of the graph. If a cubic graph is a snark, then the oddness must be at least one, since if there existed a 2-factor consisting entirely of even cycles, then each of those cycles could be colored alternately with red and blue, and the remaining edges (forming a perfect matching) with green, producing a proper 3-edge coloring of the graph. A straightforward counting argument shows that in fact, every 2-factor of any 3-valent graph must contain an even number of odd cycles, so the minimum oddness for a snark is 2. (See \cite{LukMacMaz2015} for more details, and for examples of snarks with high oddness.)

For each of the snarks on 28 vertices shown in Figure \ref{fig:28snark}, it is straightforward to find a cycle decomposition with exactly two odd cycles, so the oddness of each of those snarks is 2.

When $\gcd(m,c) = 1$ and $\gcd(m,a) = 1$ and $m$ is odd, it is also straightforward to show that $T^{1}_{\alpha}(m; a,a,c)$ and  $T^{1}_{\beta}(m; a,a,c)$ each have oddness 2, by identifying a cycle decomposition in the voltage graph that lifts to a cycle decomposition of the lift graph; see Figure \ref{fig:G28FamilyOddness}. Interestingly, when $\gcd(c,m) = \gcd(a,m) = 1$, $T^{1}_{\alpha}(m; a,a,c)$ has a decomposition in which the central cycle forms an odd cycle (when $m$ is odd) and the rest of the vertices all lie on a single odd cycle. However, $T^{1}_{\beta}(m; a,a,c)$ has no such decomposition; the best we can do is one long even cycle, of length $6m$, a shorter odd cycle of length $3m$, and the central cycle, of odd length $m$. These decompositions are shown in Figure \ref{fig:G28FamilyOddness}. 

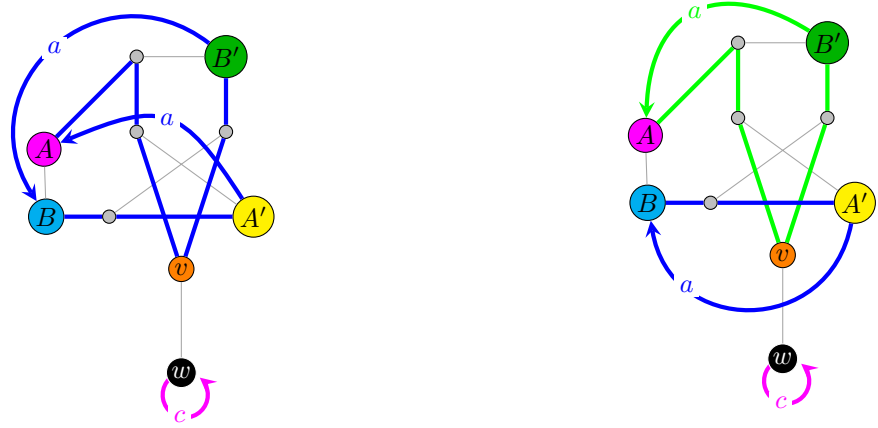
\begin{figure}[htbp]

\begin{center}
\ffigbox{
\begin{subfloatrow}[2]

\ffigbox{\caption{In the voltage graph $\tilde{T}^1_{\alpha}(a,a,c)$, the blue cycle lifts to a single cycle of length $9m$ and the magenta loop to a single cycle of length $m$.}\label{fig:G28AlphaVolt-odd}}{
\begin{tikzpicture}
\drawOneGTwentyEight	
	\draw[<-, ultra thick, blue] (A1) to[bend left = 40, looseness=1.5] node[arclabel, midway]{$a$} (A'1);
	\draw[<-, ultra thick, blue] (B1) to[bend left = 80, looseness=1.5] node[arclabel, midway]{$a$} (B'1);
	\draw[blue, ultra thick] (B'1) -- (x31) -- (C1) -- (x11) -- (x21) -- (A1);
	\draw[blue, ultra thick] (A'1) -- (x41) -- (B1);
	
	\begin{scope}[on background layer]
	\draw[->, ultra thick, mymagenta] (z1) arc(100:360+50:.3)node[midway, arclabel]{$c$} (z1);
	\end{scope}
	
\end{tikzpicture}}

\ffigbox{\caption{In the voltage graph $\tilde{T}^1_{\beta}(a,a,c)$, the blue cycle lifts to a single cycle of length $3m$, the magenta loop lifts to a single cycle of length $m$, and the green cycle lifts to a single cycle of length $6m$.}\label{fig:G28BetaVolt-odd}}{

\begin{tikzpicture}
\drawOneGTwentyEight
		
	\draw[<-, ultra thick, green] (A1) to[bend left = 60, looseness=1.5] node[arclabel, midway]{$a$} (B'1);
	\draw[<-, ultra thick, blue] (B1) to[bend right = 80, looseness=1.5] node[arclabel, near start]{$a$} (A'1);
	\draw[ultra thick, green] (A1) -- (x21) -- (x11) -- (C1) -- (x31) -- (B'1);
	\draw[ultra thick, blue] (B1) -- (x41) -- (A'1);
	\begin{scope}[on background layer]
	\draw[->, ultra thick, mymagenta] (z1) arc(100:360+50:.3)node[midway, arclabel]{$c$} (z1);
	\end{scope}
	
\end{tikzpicture}

}
\end{subfloatrow}}{
\caption{When $\gcd(a,m) = \gcd(c,m) = 1$ and $m$ is odd, the indicated cycles in the voltage graphs lift to long even or odd cycles in the lift graphs $T^{1}_{\alpha}(m; a,a,c)$ and $T^{1}_{\beta}(m; a,a,c)$. In the lift graphs, the blue and magenta cycles have odd length, and the green cycle has even length. The remaining gray edges lift to a perfect matching.}
\label{fig:G28FamilyOddness}
}
\end{center}
\end{figure}

Even more unexpectedly, the oddness of $T^{2}_{\alpha}(m; 1,1,1)$ grows with $m$. In fact:

\begin{theorem}
The snark $T^{2}_{\alpha}(2k+1; 1,1,1)$ has oddness $k+2$ when $k \equiv 0 \bmod 2$ and oddness $k+1$ when $k \equiv 1 \bmod 2$.
\end{theorem}

\begin{proof}
We analyzed all possible matchings on the 5-pole formed by clipping the two arrows and the loop in Figure \ref{fig:G28newVolt} (allowing the dangling edges to participate in the matchings); there were 28 such matchings. A few of these matchings are shown in Figure \ref{fig:oddnessMatchings}, including some (Figures \ref{fig:3a}, \ref{fig:6a}, \ref{fig:14a}) that induce a 5-cycle in the complement. We then considered all possible ordered pairs to determine which of the pairs of matchings were compatible, where two matchings are compatible if matching participation agrees on the joined dangling edges; Figure \ref{fig:oddnessCompatible} shows a few compatible pairs. 
Each of the 100 pairs of compatible matchings had at least one odd cycle in the complement of the matching; of the 100 pairs, 74 contained a 5-cycle in the  complement of the matching (two examples shown in Figures \ref{fig:5b3a} and \ref{fig:4b6a}), 8 contained a 9-cycle (one example shown in Figure \ref{fig:9-16b},  12 contained a 13-cycle, and 4 contained a 17-cycle (one example shown in Figure \ref{fig:12-4b}). Therefore, since any 2-factor of the graph is the complement of a matching on the entire graph, each matching on the graph produces matchings on the individual clusters which are compatible on the dangling edges, and every pair of adjacent clusters must contain an odd cycle participating in the 2-factor induced by the matching, it follows that there must be at least $\left\lceil \frac{2k+1}{2} \right\rceil=k+1$ odd cycles in any 2-factor. When $k$ is even, $k+1$ is odd, so one more odd cycle is required in any two-factor decomposition, since the number of odd cycles must be even in any 2-factor of a 3-valent graph; that is, when $k$ is even, the oddness is at least $k+2$. 

\begin{figure}[htbp]
\begin{center}
\ffigbox{
\begin{subfloatrow}[5]
\ffigbox{\caption{$3a$ (contains a 5-cycle)}\label{fig:3a}}
{
\begin{tikzpicture}
\ThreeA{8}{.5}{.9}{2}{green}{blue}{}
\ExtraAB{8}{.5}{.9}{2}{green}{green}
\draw[red, ultra thick] (A'2) -- (a'2);
\draw[red, ultra thick] (B'2) -- (b'2);
\draw[green] (z2) -- (wb2);
\end{tikzpicture}
}
\ffigbox{\caption{$4b$}\label{fig:4b}}{
\begin{tikzpicture}
\FourB{8}{.5}{.9}{2}{green}{blue}{}
\ExtraAB{8}{.5}{.9}{2}{green}{green}
\draw[blue] (A'2) -- (a'2);
\draw[red, ultra thick] (B'2) -- (b'2);
\draw[red, ultra thick] (z2) -- (wb2);
\end{tikzpicture}
}
\ffigbox{\caption{$5b$}\label{fig:5b}}{
\begin{tikzpicture}
\FiveB{8}{.5}{.9}{2}{green}{blue}{}
\ExtraAB{8}{.5}{.9}{2}{green}{green}
\draw[red, ultra thick] (A'2) -- (a'2);
\draw[green] (B'2) -- (b'2);
\draw[red, ultra thick] (z2) -- (wb2);
\end{tikzpicture}
}
\ffigbox{\caption{$6a$ (contains a 5-cycle)}\label{fig:6a}}{
\begin{tikzpicture}
\SixA{8}{.5}{.9}{2}{green}{blue}{}
\ExtraAB{8}{.5}{.9}{2}{green}{green}
\draw[blue] (A'2) -- (a'2);
\draw[green] (B'2) -- (b'2);
\draw[blue] (z2) -- (wb2);
\end{tikzpicture}
}
\ffigbox{\caption{$9$}\label{fig:9}}{
\begin{tikzpicture}
\Nine{8}{.5}{.9}{2}{green}{blue}{mymagenta}{}
\ExtraAB{8}{.5}{.9}{2}{green}{green}
\draw[blue] (A'2) -- (a'2);
\draw[blue] (B'2) -- (b'2);
\draw[mymagenta] (z2) -- (wb2);
\end{tikzpicture}
}
\end{subfloatrow}
\begin{subfloatrow}[4]
\ffigbox{\caption{$12$}\label{fig:12}}{
\begin{tikzpicture}
\Twelve{8}{.5}{.9}{2}{green}{blue}{}
\ExtraAB{8}{.5}{.9}{2}{green}{green}
\draw[green] (A'2) -- (a'2);
\draw[green] (B'2) -- (b'2);
\draw[red, ultra thick] (z2) -- (wb2);
\end{tikzpicture}
}
\ffigbox{\caption{$13b$ (contains an 8-cycle)}\label{fig:13b}}{
\begin{tikzpicture}
\ThirteenB{8}{.5}{.9}{2}{orange}{blue}{}
\ExtraAB{8}{.5}{.9}{2}{green}{green}
\draw[blue] (A'2) -- (a'2);
\draw[red, ultra thick] (B'2) -- (b'2);
\draw[red, ultra thick] (z2) -- (wb2);
\end{tikzpicture}
}
\ffigbox{\caption{$14a$ (contains a 5-cycle)}\label{fig:14a}}{
\begin{tikzpicture}
\FourteenA{8}{.5}{.9}{2}{green}{blue}{}
\ExtraAB{8}{.5}{.9}{2}{green}{green}
\draw[red, ultra thick] (A'2) -- (a'2);
\draw[red, ultra thick] (B'2) -- (b'2);
\draw[green] (z2) -- (wb2);
\end{tikzpicture}
}
\ffigbox{\caption{$16b$}\label{fig:16b}}{
\begin{tikzpicture}
\SixteenB{8}{.5}{.9}{2}{green}{blue}{}
\ExtraAB{8}{.5}{.9}{2}{green}{green}
\draw[red, ultra thick] (A'2) -- (a'2);
\draw[blue] (B'2) -- (b'2);
\draw[red, ultra thick] (z2) -- (wb2);
\end{tikzpicture}
}

\end{subfloatrow}
}{
\caption{There are 28 possible matchings on $\tilde{T^2}$; a few useful ones are shown here. The thick red lines indicate the matching; the various other colors show the disjoint path parts (but any path part can be any color). To be compatible, the thick red matching edges must match up between two copies, but the colors of the other edges do not have to be the same.}
\label{fig:oddnessMatchings}
}
\end{center}
\end{figure}
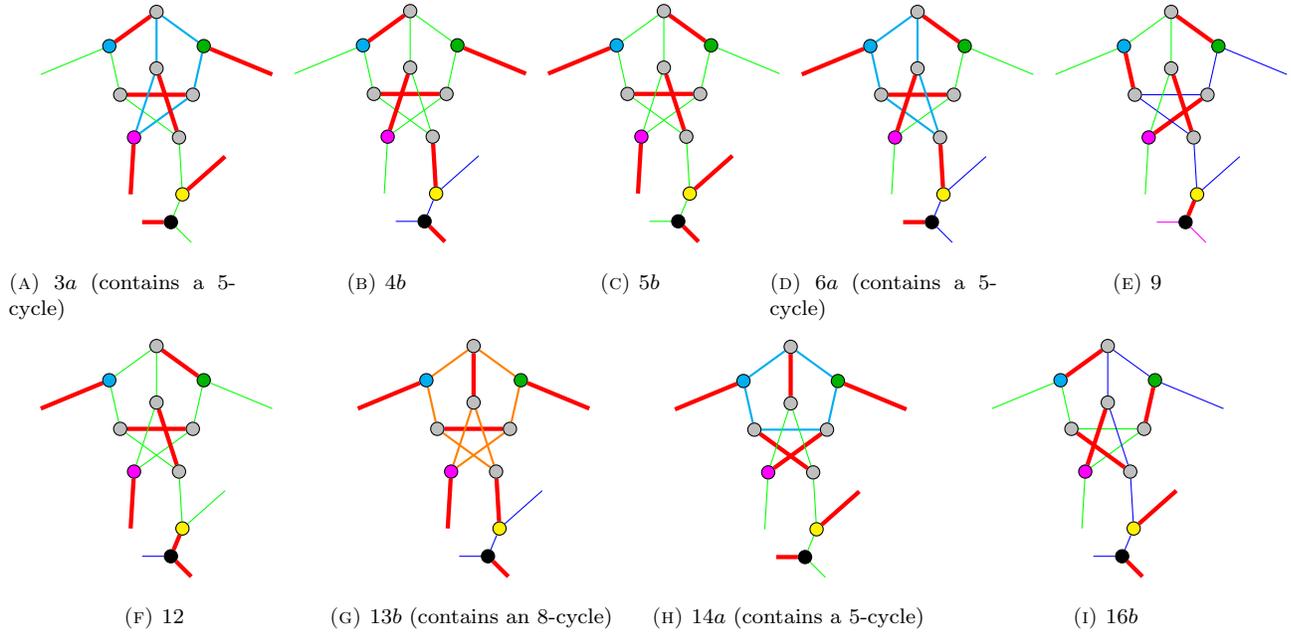

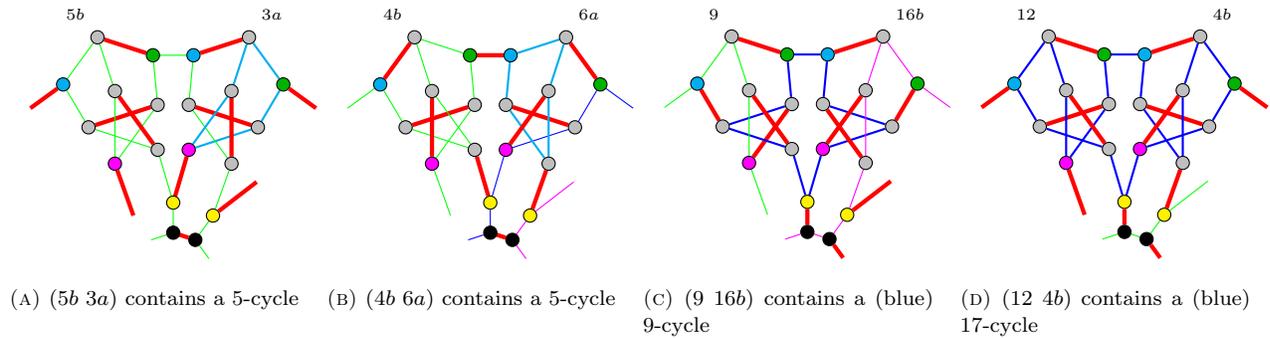
\begin{figure}[htbp]
\begin{center}
\ffigbox{
\begin{subfloatrow}[4]
\ffigbox{\caption{$(5b \ 3a)$ contains a 5-cycle}\label{fig:5b3a}}{
\begin{tikzpicture}
\ThreeA{10}{.5}{.9}{2}{green}{blue}{$3a$}
\FiveB{10}{.5}{.9}{3}{green}{green}{$5b$}
\ExtraAB{10}{.5}{.9}{2}{green}{green}
\draw[red, ultra thick] (A'2) -- (a'2);
\draw[red, ultra thick] (B'2) -- (b'2);
\draw[green] (z2) -- (wb2);
\end{tikzpicture}
}

\ffigbox{\caption{$(4b \ 6a)$ contains a 5-cycle}\label{fig:4b6a}}{
\begin{tikzpicture}
\SixA{10}{.5}{.9}{2}{blue}{mymagenta}{$6a$}
\FourB{10}{.5}{.9}{3}{green}{blue}{$4b$}
\ExtraAB{10}{.5}{.9}{2}{green}{green}
\draw[mymagenta] (A'2) -- (a'2);
\draw[blue] (B'2) -- (b'2);
\draw[mymagenta] (z2) -- (wb2);
\end{tikzpicture}
}

\ffigbox{\caption{$(9 \ 16b)$ contains a (blue) 9-cycle}\label{fig:9-16b}}{
\begin{tikzpicture}
\SixteenB{10}{.5}{.9}{2}{blue, thick}{mymagenta}{$16b$}
\Nine{10}{.5}{.9}{3}{green}{blue, thick}{mymagenta}{$9$}
\ExtraAB{10}{.5}{.9}{2}{green}{green}
\draw[red, ultra thick] (A'2) -- (a'2);
\draw[mymagenta] (B'2) -- (b'2);
\draw[red, ultra thick] (z2) -- (wb2);
\end{tikzpicture}
}

\ffigbox{\caption{$(12 \ 4b)$ contains a (blue) 17-cycle}\label{fig:12-4b}}{
\begin{tikzpicture}
\FourB{10}{.5}{.9}{2}{blue, thick}{green}{$4b$}
\Twelve{10}{.5}{.9}{3}{blue, thick}{green}{$12$}
\ExtraAB{10}{.5}{.9}{2}{green}{green}
\draw[green] (A'2) -- (a'2);
\draw[red, ultra thick] (B'2) -- (b'2);
\draw[red, ultra thick] (z2) -- (wb2);
\end{tikzpicture}
}

\end{subfloatrow}
}{
\caption{All pairs of compatible matchings on $\tilde{T^2}$ induce odd cycles in the corresponding 2-factors; a few examples are shown here.}
\label{fig:oddnessCompatible}}
\end{center}
\end{figure}

In fact, the oddness of $T^{2}_{\alpha}(2k+1; 1,1,1)$ is exactly $k+1$ when $k$ is odd and $k+2$ when $k$ is even.
To see this, it suffices to demonstrate a 2-factor decomposition that uses the required number of odd cycles. 

When $k$ is odd, we write $m = 4j+3$ for $j = 0, 1, 2, \ldots$. The sequence of matchings
\[ 4b \underbrace{ (6a \ 4b)  \cdots (6a\ 4b)}_{j \text{ pairs }} 14a \underbrace{ (5b \ 3a) \cdots (5b \ 3a)}_{j \text{ pairs }} 12\]
produces a 2-factor decomposition consisting of $2j+1$ 5-cycles, one in each cluster labelled $3a$ and $6a$ and one in the cluster labelled $14a$, and one 
 17-cycle, in the adjacent pair $(12 \ 4b)$, for a total of $2j+1+1 = 2j+2$ odd cycles. Additionally, the decomposition contains one long even cycle passing through all the clusters, of length $6 + 2(14j) + 2(j-1) + 4 = 30j+8$. Since $\left \lceil \frac{4j+3}{2} \right \rceil  = 2j + 2$, this decomposition achieves the minimum possible number of odd cycles in the decomposition. Figure \ref{fig:oddnessT2m=7} shows an example of this decomposition, when $m = 7$ and $j =1$.  

When $k$ is even, we write $m = 4j+5$ for $j = 0, 1, 2, \ldots$. The sequence of matchings
\[ 16b \underbrace{ (3a \ 5b)  \cdots (3a \ 5b)}_{j \text{ pairs }} 3a \ 13b \ 6a \underbrace{ (4b \ 6a)  \cdots (4b \ 6a)}_{j \text{ pairs }} 9\]
produces a 2-factor decomposition consisting of $2j + 2$ 5-cycles, in each of the clusters labelled $3a$ and $6a$, one 
 9-cycle, in the adjacent pair $(9 \ 16b)$, and one long odd cycle passing through all the clusters, of length $9 + 2(14j) + 2(j-1) + 2 + 14 = 30j + 23$, for a total of $2j+ 2   +2 = 2j+4$ odd cycles. Additionally, there is one 8-cycle in the cluster labelled $13b$. Since $\left \lceil \frac{4j+5}{2} \right \rceil + 1 = 2j + 4$, this decomposition achieves the minimum possible number of odd cycles in the decomposition. Figure \ref{fig:oddnessT2m=9} shows an example of this decomposition, when $m = 9$ and $j =1$. 
  \end{proof}

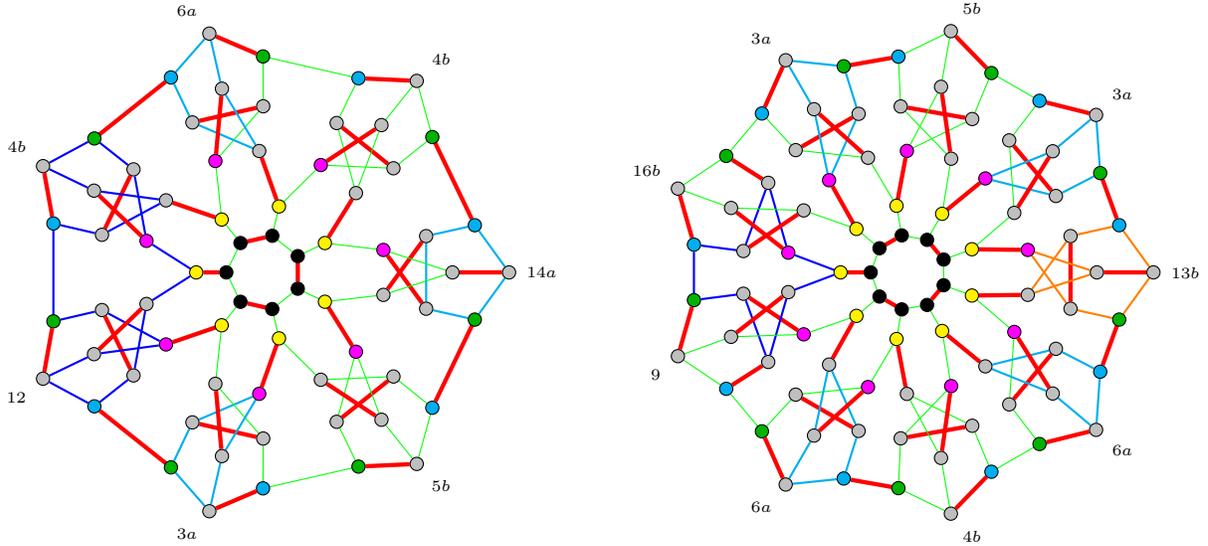
\begin{figure}[htbp]
\begin{center}
\ffigbox{
\begin{subfloatrow}[2]

\ffigbox{\caption{The oddness of $T^{2}_{\alpha}(7; 1,1,1)$ is 4, using three (cyan) 5-cycles and one (blue) 17-cycle. The long green cycle is a 38-cycle. In this case, $7 = 4(1) + 3$, so $j = 1$.}\label{fig:oddnessT2m=7}
}{
\begin{tikzpicture}
\Twelve{7}{.5}{.9}{4}{blue, thick}{green}{$12$}
\ThreeA{7}{.5}{.9}{5}{green}{green}{$3a$}
\FiveB{7}{.5}{.9}{6}{green}{green}{$5b$}
\FourteenA{7}{.5}{.9}{0}{green}{green}{$14a$}
\FourB{7}{.5}{.9}{1}{green}{green}{$4b$}
\SixA{7}{.5}{.9}{2}{green}{green}{$6a$}	
\FourB{7}{.5}{.9}{3}{blue, thick}{green}{$4b$}
\end{tikzpicture}

}

\ffigbox{\caption{The oddness of $T^{2}_{\alpha}(9; 1,1,1)$ is 6, using four (cyan) 5-cycles, one (blue) 9-cycle, and one long (green) 53-cycle; there is also an (orange) 8-cycle. In this case, $9 = 4(1) + 5$, so $j = 1$.}\label{fig:oddnessT2m=9}}{

\begin{tikzpicture}
\Nine{9}{.5}{.9}{5}{green}{blue,  thick}{green}{$9$}
\SixA{9}{.5}{.9}{6}{green}{green}{$6a$}
\FourB{9}{.5}{.9}{7}{green}{green}{$4b$}
\SixA{9}{.5}{.9}{8}{green}{green}{$6a$}
\ThirteenB{9}{.5}{.9}{0}{orange}{green}{$13b$}
\ThreeA{9}{.5}{.9}{1}{green}{blue}{$3a$}
\FiveB{9}{.5}{.9}{2}{green}{blue}{$5b$}
\ThreeA{9}{.5}{.9}{3}{green}{green}{$3a$}
\SixteenB{9}{.5}{.9}{4}{blue, thick}{green}{$16b$}
\end{tikzpicture}

}

\end{subfloatrow}
}{
\caption{Examples of cycle decompositions showing that the oddness of $T^{2}_{\alpha}(2k+1; 1,1,1)$ is $k+1$ when $k$ is odd and $k+2$ when $k$ is even.}
\label{fig:oddnessT2examples}
}
\end{center}
\end{figure}


\section{A new infinite family of pseudo-Loupekine snarks with symmetry}\label{sec:newSnarks}

We also analyzed the snarks on 34 vertices of girth at least 5 whose automorphism group is divisible by 3: there are 19 such snarks, listed in in graph6 format in  \ref{appendix:34}. We found drawings with 3-fold rotational symmetry for all of them. Surprisingly, of the 19 snarks, only one is a symmetric Loupekine snark, shown in Figure \ref{fig:G34no4} and listed in  \ref{appendix:34} as G34no4; its automorphism group has order 48, which was the largest automorphism group order among all snarks of girth more than 4 with 34 vertices. The original graph is the Petersen graph with a claw attached to the midpoints of three edges that are incident with a single vertex (see Figure \ref{fig:G34no4-original}), which is easy to show has chromatic index 4 and thus is a snark. By Proposition \ref{thm:LoupFam}, the corresponding voltage graph, shown in Figure \ref{fig:G34no4volt}, generates an infinite family of symmetric Loupekine snarks; due to the large number of automorphisms, other members of this family are likely to have interesting properties. 

In fact, the snark shown in Figure \ref{fig:G34no4} appears (in a somewhat different drawing) as Figure 7 of \cite{EspMaz2014}, where it is an example of a snark whose edge-set cannot be covered by  four perfect matchings. Such snarks are rare; of the over 64 million non-trivial snarks of order at most 36, there are only two snarks with this property: this one, and the Petersen graph. In their paper, Esperet and Mazzuoccolo generalize this snark to produce a family of snarks which each have 3-fold rotational symmetry, by iteratively applying their ``windmill construction''; their generalization produces a different family of snarks than our Loupekine snark construction, although they agree in the smallest example.

\begin{figure}[htbp]
\begin{center}
\ffigbox{
\begin{subfloatrow}[3]
\ffigbox{\caption{The unique snark with 34 vertices and automorphism group of order 48, a symmetric Loupekine snark}\label{fig:G34no4}}{
\begin{tikzpicture}[scale=1]
\drawGThirtyFourNoFourBlobs{3}{0}{.5}
\foreach \i in {0,1,2}{
\drawAlphaA[0]{3}{\i}{1}{black, thick}{0}
\drawAlphaB[0]{3}{\i}{1}{black, thick}{0}
\drawSpoke{\i}{black}
\drawLoop{3}{\i}{black}{1}
}
\end{tikzpicture}
}

\ffigbox{\caption{The corresponding voltage graph}\label{fig:G34no4volt}}{
\begin{tikzpicture}[scale=1]
\drawOneGThirtyFourNoFour
\draw[ultra thick, ->] (B1) to[bend left = 80, looseness=1.5] node[arclabel]{$b$}  (B'1);
\draw[ultra thick, ->] (A1) to[bend right = 50, looseness = 1.4] node[arclabel, near start]{$a$}  (A'1);
\begin{scope}[on background layer]
	\draw[->, ultra thick] (z1) arc(100:360+50:.3)node[midway, arclabel]{$c$} (z1);
	\end{scope}

\end{tikzpicture}
}

\ffigbox{\caption{The original graph. The introduced claw is shown in orange, and the dashed edges form the removed path.}\label{fig:G34no4-original}}{
\begin{tikzpicture}[scale=1]
\drawOneGThirtyFourNoFour
\node[vtx, fill = white, below left = of A1] (s) {$s$};
\node[vtx, fill = white, below right = of B'1] (t) {$s$};
\draw[ultra thick] (B1) -- (s)--(A1);
\draw[ultra thick] (B'1) -- (t)--(A'1);
\draw[ultra thick, dashed, gray] (s) -- (z1) -- (t);
\draw[thick, orange] (C1) -- (B1) (C1) -- (A1) (C1) -- (z1);

\end{tikzpicture}
}

\end{subfloatrow}
}{
\caption{A Loupekine snark on 34 vertices, along with its voltage graph and original graph; the 5-pole was formed by deleting a $P_{3}$ from a snark. In this case, the snark is constructed by subdividing three edges of the Petersen graph that are all incident at a vertex, and then attaching a claw to the vertices introduced in the subdivision; it is straightforward to show that this new graph is a snark. The vertices introduced in the subdivision are $B$, $A$ and $w$, and the claw (shown in orange) has center $v$.}
\label{}}
\end{center}
\end{figure}
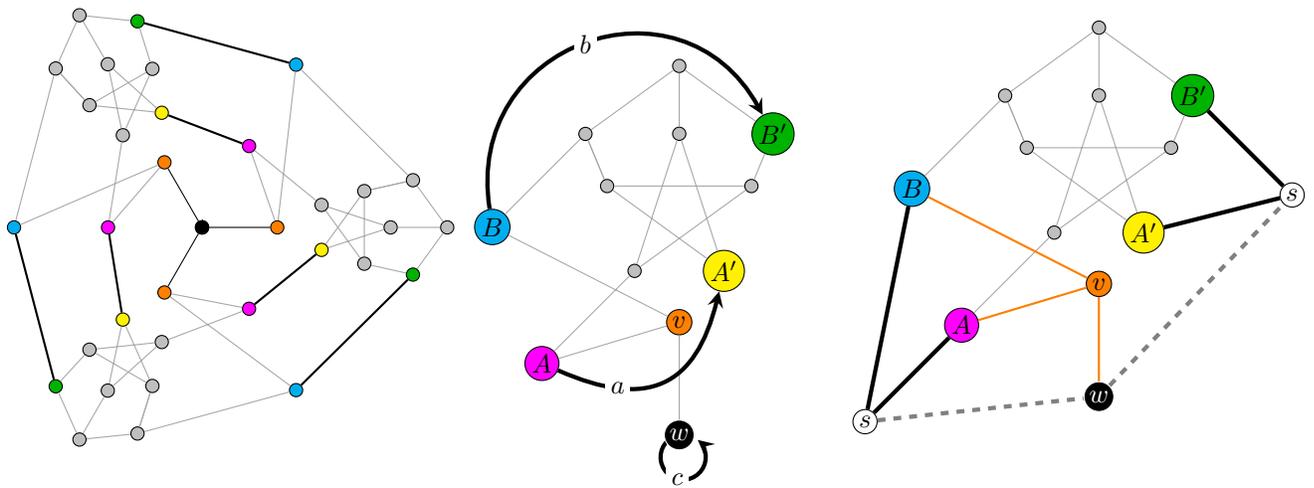

Of the remaining 18 snarks, 7 have voltage graphs with more than two arrows. Their analysis is beyond the scope of this paper. 
This leaves 10 snarks whose corresponding 5-poles and voltage graphs can be derived by deleting a path $P_{3}$ from various original graphs on 14 vertices; it is straightforward to show that these original graphs are all 3-edge-colorable.

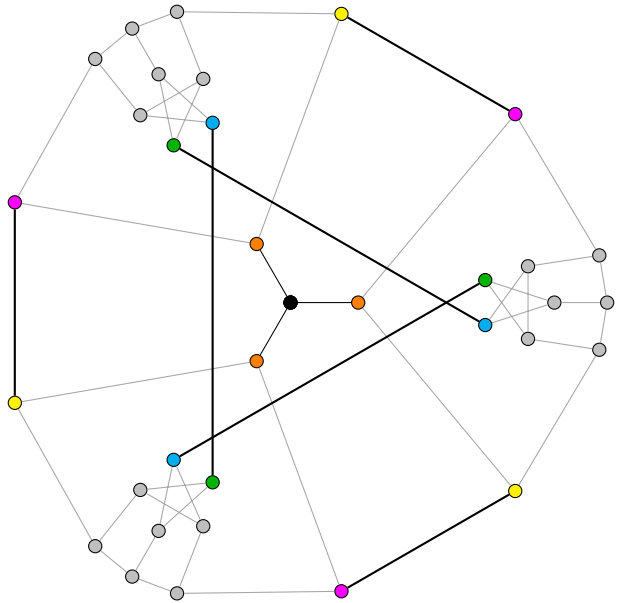
\begin{figure}[htbp]
\begin{center}
\begin{tikzpicture}[scale = 1.5]
\drawGThirtyFourNoNineBlobs{3}{0}{.5}{.6}
\foreach \i in {0,1,2}{
\drawAlphaA[0]{3}{\i}{1}{black, thick}{0}
\drawAlphaB[0]{3}{\i}{1}{black, thick}{0}
\drawSpoke{\i}{black}
\drawLoop{3}{\i}{black}{1}
}
\end{tikzpicture}
\caption{A snark on 34 vertices with 3-fold rotational symmetry (in fact, 3-fold dihedral symmetry), which generates a new infinite family of snarks. If the center point is expanded to a triangle, we will represent the graph as $G_{\alpha}(m; 1,1,1)$. Connecting edges are shown thick.}
\label{fig:G34No9}
\end{center}
\end{figure}

In this section, we focus on one of these snarks (graph G34no9 in  Appendix \ref{appendix:34}), shown in Figure \ref{fig:G34No9}.    For the remainder of this section, $G$ denotes the 5-pole shown in Figure \ref{fig:G34no9-multipole}. The voltage graph $\tilde{G}_{\alpha}(m; a,b,c)$ derived from the snark is shown in Figure \ref{fig:G34no9-volt}, and the original graph $\bar{G}$ is shown in Figure \ref{fig:G34no9-original}. The proper 3-edge-coloring in that figure demonstrates that---unlike the previous examples of symmetric Loupekine snarks---$\bar{G}$ is not itself a snark. As a lift, the snark G34no9 can be denoted by 
$G_{\alpha} (3; 1, 1,1)$ after the fixed point of the snark is expanded to a triangle.

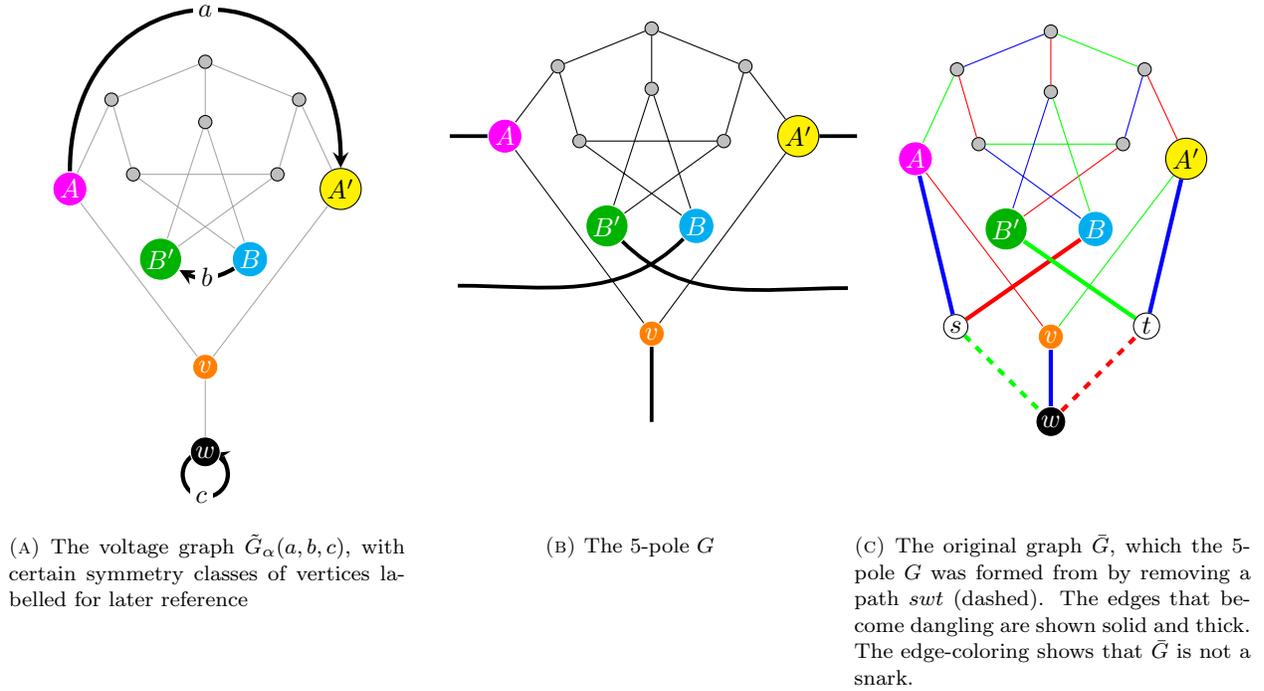
\begin{figure}[htbp]
\begin{center}
\ffigbox{
\begin{subfloatrow}[3]
\ffigbox{ \caption{The voltage graph $\tilde{G}_{\alpha}(a,b,c)$, with certain symmetry classes of vertices labelled for later reference}\label{fig:G34no9-volt}}{
\begin{tikzpicture}[scale=1.5]
	\node[circle, gray!70!white, minimum width = 2 cm, inner sep = 1pt] (G) at ($(90:2)$){};
	\node[vtx, fill = gray!50!white] (x1) at (G.90) {};
	\node[vtx, fill = gray!50!white] (x4) at (G.90-72) {};
	\node[vtx,white, fill = cyan] (B) at (G.90-2*72) {$B$};
	\node[vtx,fill = gray!50!white] (x2) at (G.90+72) {};
	\node[vtx, white,fill = green!70!black] (B') at (G.90+2*72) {$B'$};
	\node[vtx, fill = gray!50!white] (x6) at ($(0,0)!1.3!(x2)$) {};
	\node[vtx, fill = gray!50!white] (x5) at($(0,0)!1.2!(x1)$) {};
	\node[vtx, fill = gray!50!white] (x3) at ($(0,0)!1.3!(x4)$)  {};
	
	\node[vtx, fill = yellow] (A') at ($(90-30:2*1.2)$){$A'$};
	\node[vtx, white,fill = mymagenta] (A) at ($(30+90:2*1.2)$){$A$};
	\node[vtx,white, fill = orange] (C) at (90:.5){$v$};
	\node[vtx, below = .75 of C, white, fill = black] (z) {$w$};

	\draw[gray!70!white] (B') -- (x4) -- (x2) -- (B) -- (x1)--(B');
	\draw[gray!70!white] (x2) -- (x6) -- (x5) -- (x3) -- (x4);
	\draw[gray!70!white] (x6) -- (A) -- (C) -- (A')--(x3);
	\draw[gray!70!white]  (x1)--(x5);
	\draw[gray!70!white] (C) -- (z);
	
	\begin{scope}[on background layer]
	\draw[->, ultra thick, looseness = 2] (A) to[out=90, in=90]  node[midway, arclabel] {$a$} (A');
	\draw[->, ultra thick] (B) to[bend left=30] node[near start, arclabel] {$b$} (B');
	\draw[->, ultra thick] (z) arc(90:360+70:0.2) node[midway, arclabel] {$c$}  (z);
	\end{scope}
\end{tikzpicture}
}

\ffigbox{\caption{The 5-pole $G$}\label{fig:G34no9-multipole}}{

\begin{tikzpicture}[scale =1.5]
	\node[circle, gray!70!white, minimum width = 2 cm, inner sep = 1pt] (G) at ($(90:2)$){};
	\node[vtx, fill = gray!50!white] (x1) at (G.90) {};
	\node[vtx, fill = gray!50!white] (x4) at (G.90-72) {};
	\node[vtx,white, fill = cyan] (B) at (G.90-2*72) {$B$};
	\node[vtx,fill = gray!50!white] (x2) at (G.90+72) {};
	\node[vtx, white,fill = green!70!black] (B') at (G.90+2*72) {$B'$};
	\node[vtx, fill = gray!50!white] (x6) at ($(0,0)!1.3!(x2)$) {};
	\node[vtx, fill = gray!50!white] (x5) at($(0,0)!1.2!(x1)$) {};
	\node[vtx, fill = gray!50!white] (x3) at ($(0,0)!1.3!(x4)$)  {};
	
	\node[vtx, fill = yellow] (A') at ($(90-30:2*1.3)$){$A'$};
	\node[vtx, white,fill = mymagenta] (A) at ($(30+90:2*1.3)$){$A$};
	\node[vtx,white, fill = orange] (C) at (90:.5){$v$};
	\node[below = 1 of C] (z) {};

	\node[left  = .5 of A] (a){}; 
	\node[right  =.5 of A'] (a'){}; 
	\node[below left = .5 and 3 of B] (b){}; 
	\node[below right = .5 and 3 of B'] (b'){};

	\draw[line width = .5 mm, ] (a)--(A);
	\draw[line width = .5 mm,,] (a')--(A');
	\draw[line width = .5 mm, ] (B) to[out=180+45, in = 0] (b);
	\draw[line width = .5 mm, ] (B') to[out=-45, in = 180] (b');
	\draw[line width = .5 mm, ] (C) -- (z);

	\draw[] (A) -- (C) (x2) -- (x6) (x1) -- (x5) (x3) -- (A') (B') -- (x4);
	\draw[] (C) -- (A') (A) -- (x6) (x2)--(x4) (x1)--(B) (x5)--(x3);
	\draw[]  (x6)--(x5) (x2)--(B) (x4)--(x3) (B')--(x1);

	\end{tikzpicture}
}

\ffigbox{ \caption{The original graph $\bar{G}$, which the 5-pole $G$ was formed from by removing a path $swt$ (dashed). The edges that become dangling are shown solid and thick. 
The edge-coloring shows that $\bar{G}$ is not a snark.}\label{fig:G34no9-original}}{
\begin{tikzpicture}[scale=1.5]
	\node[circle, gray!70!white, minimum width = 2 cm, inner sep = 1pt] (G) at ($(90:2)$){};
	\node[vtx, fill = gray!50!white] (x1) at (G.90) {};
	\node[vtx, fill = gray!50!white] (x4) at (G.90-72) {};
	\node[vtx,white, fill = cyan] (B) at (G.90-2*72) {$B$};
	\node[vtx,fill = gray!50!white] (x2) at (G.90+72) {};
	\node[vtx, white,fill = green!70!black] (B') at (G.90+2*72) {$B'$};
	\node[vtx, fill = gray!50!white] (x6) at ($(0,0)!1.3!(x2)$) {};
	\node[vtx, fill = gray!50!white] (x5) at($(0,0)!1.2!(x1)$) {};
	\node[vtx, fill = gray!50!white] (x3) at ($(0,0)!1.3!(x4)$)  {};
	
	\node[vtx, fill = yellow] (A') at ($(90-30:2*1.2)$){$A'$};
	\node[vtx, white,fill = mymagenta] (A) at ($(30+90:2*1.2)$){$A$};
	\node[vtx,white, fill = orange] (C) at (90:.5){$v$};
	\node[vtx, below = .75 of C, white, fill = black] (z) {$w$};
	
	\node[above left= of z, vtx, fill=white] (s) {$s$};
	\node[above right = of z, vtx, fill=white] (t) {$t$};
	
	\draw[red] (A) -- (C) (x2) -- (x6) (x1) -- (x5) (x3) -- (A') (B') -- (x4);
	\draw[green] (C) -- (A') (A) -- (x6) (x2)--(x4) (x1)--(B) (x5)--(x3);
	\draw[blue]  (x6)--(x5) (x2)--(B) (x4)--(x3) (B')--(x1);
	
	\draw[ultra thick, dashed, green] (s) -- (z);
	\draw[ultra thick, dashed,red] (z)  -- (t);
	\draw[ultra thick, blue] (A) -- (s);
	\draw[ultra thick, red] (B) -- (s);
	
	\draw[ultra thick, blue] (A') -- (t) (C) -- (z);
	\draw[ultra thick, green] (B') -- (t) ;

\end{tikzpicture}

}
\end{subfloatrow}
}{
\caption{The 5-pole $G$,  voltage graph $\tilde{G}_{\alpha}(a,b,c)$ and original graph $\bar{G}$ for the new class of snarks, whose smallest member is shown in Figure \ref{fig:G34No9}.}
\label{fig:G34No9-volt-original}}
\end{center}
\end{figure}

The remainder of this section is devoted to proving the following result:

\begin{theorem}\label{thm:G34no9snark}
If $m$ is odd, then $G_{\alpha}(m; 1,1,1)$ is a snark.
\end{theorem}

The proof uses a number of technical lemmas, which we present separately before the proof.

Surprisingly---and in contrast to the case of symmetric Loupekine snarks, where for a fixed Loupekine 5-pole, both the $\alpha$ and $\beta$-connections produced snarks---the graph $G_{\beta}(m;1,1,1)$ is always 3-edge-colorable, which we will prove later in Theorem \ref{lem:G34no9BETA}.

To determine the chromatic index of $G_{\alpha}(m; 1,1,1)$, we begin by considering all possible ways to assign colors to the dangling edges of the 5-pole $G$. The 5 dangling edges in the 5-pole form a cut set in the lift, and by the Parity Lemma, 
three of them must be colored 1 (say red), the \emph{main} color of the pattern, one must be colored 2 (say green), and one must be colored 3 (say blue). 

We define a \emph{color pattern} of the 5-pole $G$ to be an assignment of the colors $\{1,2, 3\}$ to the dangling edges of $G$, up to a permutation of the assigned colors. A color pattern is \emph{admissible} if that assignment of colors to the dangling edges can be completed to a 3-edge-coloring of $G$. We systematically tested all of the possible $\binom{5}{3}$ color assignments, by assigning red to each of the possible 3-subsets of the 5 dangling edges and determining which of those assignments can be completed to a proper 3-edge-coloring of the 5-pole. We conclude:

\begin{lemma}
Up to choice of color, there are only four admissible color patterns for the dangling edges of the 5-pole $G$, shown in Figure \ref{fig:G34No9-patterns}, and they all have the property that the main color of the pattern is the color of the spoke edge.
\end{lemma}

\def\myscale{1}
\def\mythick{1.3 }
\begin{figure}[htbp]
\begin{center}
\ffigbox{
\begin{subfloatrow}[2]
\ffigbox{\caption{Pattern {\sf TwoLeft}, represented as $[(x,x) ,(y,z)]$ or $[(x,x) ,(z,y)]$}\label{twoLeft}}{
\begin{tikzpicture}[scale =\myscale]
	\node[circle, gray!70!white, minimum width = 2 cm, inner sep = 1pt] (G) at ($(90:2)$){};
	\node[vtx, fill = gray!50!white] (x1) at (G.90) {};
	\node[vtx, fill = gray!50!white] (x4) at (G.90-72) {};
	\node[vtx,white, fill = cyan] (B) at (G.90-2*72) {$B$};
	\node[vtx,fill = gray!50!white] (x2) at (G.90+72) {};
	\node[vtx, white,fill = green!70!black] (B') at (G.90+2*72) {$B'$};
	\node[vtx, fill = gray!50!white] (x6) at ($(0,0)!1.3!(x2)$) {};
	\node[vtx, fill = gray!50!white] (x5) at($(0,0)!1.2!(x1)$) {};
	\node[vtx, fill = gray!50!white] (x3) at ($(0,0)!1.3!(x4)$)  {};
	
	\node[vtx, fill = yellow, below right = .8 and .25 of x3] (A') {$A'$};
	\node[vtx, white,fill = mymagenta, below left = .8 and .25  of x6] (A) {$A$}; 
	\node[vtx,white, fill = orange] (C) at (90:.1){$v$};
	\node[below = 1 of C] (z) {};

	\node[left  = 1 of A] (a){}; 
	\node[right  =1 of A'] (a'){}; 
	\node[below left = 1 and 3.5 of B] (b){}; 
	\node[below right = 1 and 3.5 of B'] (b'){};

	\draw[line width = \mythick mm, red] (a)--(A);
	\draw[line width = \mythick mm,blue] (a')--(A');
	\draw[line width = \mythick mm, red] (B) to[out=180+45, in = 0] (b);
	\draw[line width = \mythick mm, green] (B') to[out=-45, in = 180] (b');
	\draw[line width = \mythick mm, red] (C) -- (z);

	\draw[blue, ultra thick] (A) -- (C);
	 \draw[blue, ultra thick] (x2) -- (x6);
	 \draw[green, ultra thick] (x1) -- (x5) ;
	 \draw[red, ultra thick](x3) -- (A') ;
	 \draw[blue, ultra thick] (B') -- (x4);
	\draw[red, ultra thick] (C) -- (A');
	\draw[green, ultra thick] (A) -- (x6);
	\draw[red, ultra thick] (x2)--(x4) ;
	\draw[blue, ultra thick](x1)--(B) ;
	\draw[blue, ultra thick](x5)--(x3);
	\draw[red, ultra thick]  (x6)--(x5);
	\draw[green, ultra thick] (x2)--(B);
	\draw[green, ultra thick] (x4)--(x3);
	\draw[red, ultra thick] (B')--(x1);

	\end{tikzpicture}

}
\ffigbox{\caption{Pattern {\sf TwoRight},  represented  as $[(y,z),(x,x)]$ or $[(z,y),(x,x)]$}\label{twoRight}}{
\begin{tikzpicture}[scale =\myscale]
	\node[circle, gray!70!white, minimum width = 2 cm, inner sep = 1pt] (G) at ($(90:2)$){};
	\node[vtx, fill = gray!50!white] (x1) at (G.90) {};
	\node[vtx, fill = gray!50!white] (x4) at (G.90-72) {};
	\node[vtx,white, fill = cyan] (B) at (G.90-2*72) {$B$};
	\node[vtx,fill = gray!50!white] (x2) at (G.90+72) {};
	\node[vtx, white,fill = green!70!black] (B') at (G.90+2*72) {$B'$};
	\node[vtx, fill = gray!50!white] (x6) at ($(0,0)!1.3!(x2)$) {};
	\node[vtx, fill = gray!50!white] (x5) at($(0,0)!1.2!(x1)$) {};
	\node[vtx, fill = gray!50!white] (x3) at ($(0,0)!1.3!(x4)$)  {};
	
	\node[vtx, fill = yellow, below right = .8 and .25 of x3] (A') {$A'$};
	\node[vtx, white,fill = mymagenta, below left = .8 and .25  of x6] (A) {$A$}; 
	\node[vtx,white, fill = orange] (C) at (90:.1){$v$};
	\node[below = 1 of C] (z) {};

	\node[left  = 1 of A] (a){}; 
	\node[right  =1 of A'] (a'){};  
	\node[below left = 1 and 3.5 of B] (b){}; 
	\node[below right = 1 and 3.5 of B'] (b'){};

	\draw[line width = \mythick mm, green] (a)--(A);
	\draw[line width = \mythick mm,red] (a')--(A');
	\draw[line width = \mythick mm, blue] (B) to[out=180+45, in = 0] (b);
	\draw[line width = \mythick mm, red] (B') to[out=-45, in = 180] (b');
	\draw[line width = \mythick mm, red] (C) -- (z);

	\draw[blue, ultra thick] (A) -- (C);
	 \draw[green, ultra thick] (x2) -- (x6);
	 \draw[red, ultra thick] (x1) -- (x5) ;
	 \draw[blue, ultra thick](x3) -- (A') ;
	 \draw[green, ultra thick] (B') -- (x4);
	\draw[green, ultra thick] (C) -- (A');
	\draw[red, ultra thick] (A) -- (x6);
	\draw[blue, ultra thick] (x2)--(x4) ;
	\draw[green, ultra thick](x1)--(B) ;
	\draw[green, ultra thick](x5)--(x3);
	\draw[blue, ultra thick]  (x6)--(x5);
	\draw[red, ultra thick] (x2)--(B);
	\draw[red, ultra thick] (x4)--(x3);
	\draw[blue, ultra thick] (B')--(x1);

	\end{tikzpicture}
}
\end{subfloatrow}
\begin{subfloatrow}[2]
\ffigbox{\caption{Pattern {\sf AltTop},  represented as $[(x,z) ,(x,y)]$ or $[(x,y), (x,z)]$}\label{altTop}}{

\begin{tikzpicture}[scale =\myscale]
	\node[circle, gray!70!white, minimum width = 2 cm, inner sep = 1pt] (G) at ($(90:2)$){};
	\node[vtx, fill = gray!50!white] (x1) at (G.90) {};
	\node[vtx, fill = gray!50!white] (x4) at (G.90-72) {};
	\node[vtx,white, fill = cyan] (B) at (G.90-2*72) {$B$};
	\node[vtx,fill = gray!50!white] (x2) at (G.90+72) {};
	\node[vtx, white,fill = green!70!black] (B') at (G.90+2*72) {$B'$};
	\node[vtx, fill = gray!50!white] (x6) at ($(0,0)!1.3!(x2)$) {};
	\node[vtx, fill = gray!50!white] (x5) at($(0,0)!1.2!(x1)$) {};
	\node[vtx, fill = gray!50!white] (x3) at ($(0,0)!1.3!(x4)$)  {};
	
	\node[vtx, fill = yellow, below right = .8 and .25 of x3] (A') {$A'$};
	\node[vtx, white,fill = mymagenta, below left = .8 and .25  of x6] (A) {$A$}; 
	\node[vtx,white, fill = orange] (C) at (90:.1){$v$};
	\node[below = 1 of C] (z) {};

	\node[left  = 1 of A] (a){}; 
	\node[right  =1 of A'] (a'){}; 
	\node[below left = 1 and 3.5 of B] (b){}; 
	\node[below right = 1 and 3.5 of B'] (b'){};

	\draw[line width = \mythick mm, red] (a)--(A);
	\draw[line width = \mythick mm,red] (a')--(A');
	\draw[line width = \mythick mm, blue] (B) to[out=180+45, in = 0] (b);
	\draw[line width = \mythick mm, green] (B') to[out=-45, in = 180] (b');
	\draw[line width = \mythick mm, red] (C) -- (z);

	\draw[blue, ultra thick] (A) -- (C);
	 \draw[blue, ultra thick] (x2) -- (x6);
	 \draw[blue, ultra thick] (x1) -- (x5) ;
	 \draw[blue, ultra thick](x3) -- (A') ;
	 \draw[blue, ultra thick] (B') -- (x4);
	\draw[green, ultra thick] (C) -- (A');
	\draw[green, ultra thick] (A) -- (x6);
	\draw[green, ultra thick] (x2)--(x4) ;
	\draw[green, ultra thick](x1)--(B) ;
	\draw[green, ultra thick](x5)--(x3);
	\draw[red, ultra thick]  (x6)--(x5);
	\draw[red, ultra thick] (x2)--(B);
	\draw[red, ultra thick] (x4)--(x3);
	\draw[red, ultra thick] (B')--(x1);

	\end{tikzpicture}
}
\ffigbox{\caption{Pattern {\sf AltBot}, represented  as $[(z,x),(y,x)]$ or $[(y,z),(z,x)]$}\label{altBot}}{

\begin{tikzpicture}[scale =\myscale]
	\node[circle, gray!70!white, minimum width = 2 cm, inner sep = 1pt] (G) at ($(90:2)$){};
	\node[vtx, fill = gray!50!white] (x1) at (G.90) {};
	\node[vtx, fill = gray!50!white] (x4) at (G.90-72) {};
	\node[vtx,white, fill = cyan] (B) at (G.90-2*72) {$B$};
	\node[vtx,fill = gray!50!white] (x2) at (G.90+72) {};
	\node[vtx, white,fill = green!70!black] (B') at (G.90+2*72) {$B'$};
	\node[vtx, fill = gray!50!white] (x6) at ($(0,0)!1.3!(x2)$) {};
	\node[vtx, fill = gray!50!white] (x5) at($(0,0)!1.2!(x1)$) {};
	\node[vtx, fill = gray!50!white] (x3) at ($(0,0)!1.3!(x4)$)  {};
	
	\node[vtx, fill = yellow, below right = .8 and .25 of x3] (A') {$A'$};
	\node[vtx, white,fill = mymagenta, below left = .8 and .25  of x6] (A) {$A$}; 
	\node[vtx,white, fill = orange] (C) at (90:.1){$v$};
	\node[below = 1 of C] (z) {};
	
	\node[left  = 1 of A] (a){}; 
	\node[right  =1 of A'] (a'){}; 
	\node[below left = 1 and 3.5 of B] (b){}; 
	\node[below right = 1 and 3.5 of B'] (b'){};

	\draw[line width = \mythick mm, green] (a)--(A);
	\draw[line width = \mythick mm,,blue] (a')--(A');
	\draw[line width = \mythick mm,, red] (B) to[out=180+45, in = 0] (b);
	\draw[line width = \mythick mm, red] (B') to[out=-45, in = 180] (b');
	\draw[line width = \mythick mm, red] (C) -- (z);

	\draw[blue, ultra thick] (A) -- (C);
	 \draw[blue,ultra thick] (x2) -- (x6);
	 \draw[red, ultra thick] (x1) -- (x5) ;
	 \draw[red, ultra thick](x3) -- (A') ;
	 \draw[blue, ultra thick] (B') -- (x4);
	\draw[green, ultra thick] (C) -- (A');
	\draw[red, ultra thick] (A) -- (x6);
	\draw[red, ultra thick] (x2)--(x4) ;
	\draw[blue,ultra thick](x1)--(B) ;
	\draw[blue, ultra thick](x5)--(x3);
	\draw[green,ultra thick]  (x6)--(x5);
	\draw[green, ultra thick] (x2)--(B);
	\draw[green, ultra thick] (x4)--(x3);
	\draw[green, ultra thick] (B')--(x1);

	\end{tikzpicture}

}

\end{subfloatrow}
}{
\caption{Admissible color patterns for the 5-pole $G$. The left-hand pair of edges is the \emph{input pair}, and the right-hand pair of edges is the \emph{output pair}. When reading the text representation, the description $[(x_{1},x_{2}),(y_{1},y_{2})]$ indicates that in the input pair, the top left edge is assigned color $x_{1}$ and the bottom left edge is assigned color $x_{2}$, while in the output pair, the top right edge is assigned color $y_{1}$ and the bottom right edge is assigned color $y_{2}$. }
\label{fig:G34No9-patterns}}
\end{center}
\end{figure}
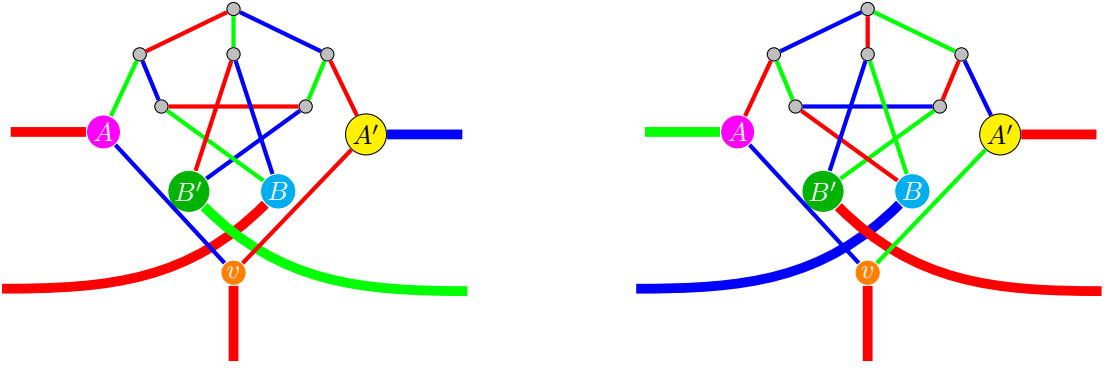
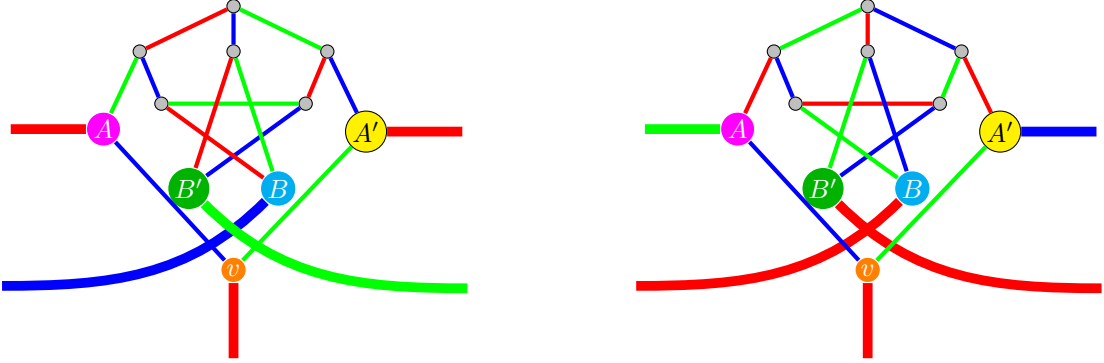

Notice that all admissible color patterns for $G$ have the spoke edge colored the main color of the pattern (red). Therefore, any assignment of colors to the left-hand dangling edges with endpoints $A$ and $B$, the \emph{input pair},  and the right-hand dangling edges with endpoints $A'$ and $B'$, the \emph{output pair}, determines the color of the spoke edge with endpoint $v$, by determining which color is repeated twice in the assignment of colors to the input and output pairs. Thus, we can completely specify the color pattern by assigning colors to the input and output pairs. 
 For notational convenience, we use ordered pairs of ordered pairs of colors, denoted by $[(x_{1},x_{2}),(y_{1},y_{2})]$, to represent each admissible  color pattern, where  
 the first pair corresponds to the colors assigned to the input pair (from top to bottom), and the second pair to the colors assigned to the output pair (from top to bottom). Specifically, $[(x_{1},x_{2}),(y_{1},y_{2})]$ assigns colors $x_{1}, x_{2}, y_{1}, y_{2}$ to the dangling edges with endpoints $A, B, A', B'$ respectively. We call $[(x_{1},x_{2}),(y_{1},y_{2})]$  the \emph{color pattern assignment for the cluster}.
For example, Figure \ref{twoLeft} shows color pattern {\sf TwoLeft}, which assigns red to the  input pair (the two left-hand dangling edges) and the spoke edge, and blue and green to output pair (the upper and lower right-hand dangling edges); the corresponding color pattern assignment is $[(x,x),(y,z)]$.

We  extend the notion of color pattern assignments to 
a consecutive sequence of $r$ connected clusters
, called a \emph{cluster sequence}, 
which we  denote by $G'_{r}$ (or $G'^{j}_{r}$ if we care that it starts at $G_{j}$, which we usually don't). Specifically, if $r <m$, then $G'^{j}_{r}$ is the induced subgraph of $G_{\alpha}(m; 1,1,1)$ formed by pulling out the $r$ clusters $G_{j}, G_{j+1}, \ldots, G_{j+r-1}$ (with index arithmetic mod $m$ and $1\leq r \leq m-1$) from the graph, including the connecting loop edges, 
 along with the dangling edges $A'_{j-1}A_{j}$, $B'_{j-1}B_{j}$, $w_{j-1}w_{j}$,  $A'_{j+(r-1)}A_{j+(r-1)+1}$, $B'_{j+(r-1)}B_{j+(r-1)+1}$ and $w_{j+(r-1)}w_{j+(r-1)+1}$. Alternately, we think of connecting up $r$ copies of the 5-pole $G$ using $\alpha$-connections (with $a = b=1$) between each successive copy. If $r = m$, then the corresponding cluster sequence is formed by cutting the edges  between clusters $G_{0}$ and $G_{m-1}$. (That is, a cluster sequence always has on each end 3 dangling edges, two connecting edges and one loop edge. Gluing together the corresponding dangling edges in $G'_{m}$ results in $G_{\alpha}(m; 1,1,1)$.)
 
The \emph{input pair} for the cluster sequence ${G'_{r}}^{j}$ is the pair of dangling edges $B'_{j-1}B_{j}$, $A'_{j-1}A_{j}$ (colored $x_{1}, x_{2}$ respectively), and the \emph{output pair} is the pair of dangling edges  $B'_{j+(r-1)}B_{j+(r-1)+1}$, $A'_{j+(r-1)}A_{j+(r-1)+1}$ (colored $y_{1}, y_{2}$ respectively).
Furthermore,  the \emph{input loop edge} is the edge $w_{j-1}w_{j}$ (colored $x_{3}$) and the \emph{output loop edge} is the edge $w_{j+(r-1)}w_{j+(r-1)+1}$  (colored $y_{3}$), which together with the input and output pairs give us the \emph{color pattern assignment for the cluster sequence} $[(x_1,x_2,x_3),(y_1,y_2,y_3)]_{{G'_{r}}^{j}}$
A schematic of a sequence of 3 clusters is shown in Figure \ref{fig:seqOfClusters}, in which non-connecting edges in the clusters are suppressed and represented by a cloud, and the vertices $A_{i}, B_{i}$ are clearly on the left of each cluster and $A'_{i}, B'_{i}$ clearly on the right of each cluster. 

In general, we assume the cluster cycle starts at 0; everything is rotationally symmetric, so we can rotate around if necessary.

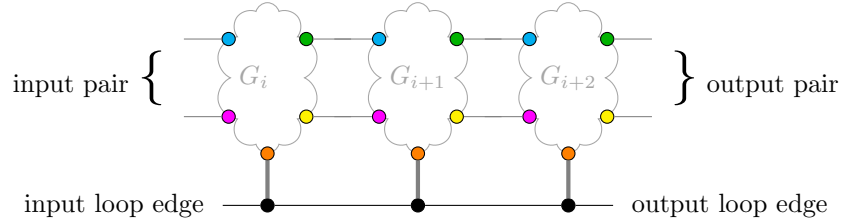
\begin{figure}[htbp]
\begin{center}
\begin{tikzpicture}[node distance=0.5cm]
\drawOneBlob[$G_{i\phantom{+1}}$]{0}{gray, thin}{gray,thin}{gray}{gray,thin}{gray,thin}
\drawOneBlob[$G_{i+1}$]{2}{gray,thin}{gray,thin}{gray}{gray,thin}{gray,thin}\drawOneBlob[$G_{i+2}$]{4}{gray,thin}{gray,thin}{gray}{gray,thin}{gray,thin}
\node[left=of z0] (start) {};
\node[right=of z4] (end) {};
\draw[] (z0) -- (z2);
\draw[] (z2) -- (z4);
\draw (start) -- (z0);
\draw (z4) -- (end);
\draw (start) node[left] {input loop edge};
\draw (end) node[right] {output loop edge};
\draw ($(a0)!.5!(b0)$) node[left] {input pair {\Huge $\{$}} ;
\draw ($(a'4)!.5!(b'4)$) node[right] {{\Huge $\}$} output pair } ;

\end{tikzpicture}
\caption{A cluster sequence $G'^{i}_{3}$, with labelled elements}
\label{fig:seqOfClusters}
\end{center}
\end{figure}

We extend the notion of a color pattern assignment to a cluster sequence by focusing on the colors of the spokes in the clusters in the sequence, since if $r >1$, knowing the colors assigned to the input and output pairs for the cluster sequence $G'_{r}$ does not specify the spoke colors (unlike the case of a single cluster). We define a \emph{color string} to be a sequence of colors
\[\mc{C} = c_{0}c_{1}\cdots c_{r-1}\] where $c_{i} \in \{1,2,3\}$.
A \emph{substring} of a color string $\mc{C}$ is any subset of consecutive entries in $\mc{C}$.

We will use color strings to assign colors to spoke edges, both in graphs $G_{\alpha}(m;1,1,1)$  and in cluster sequences $G'_{r}$, by assigning color $c_{i}$ to spoke $v_{i}w_{i}$. We will then analyze the effect of that color string on coloring the rest of the object. For the graphs $G_{\alpha}(m;1,1,1)$ we will define the \emph{string index of $\mc{C}$}, while for cluster sequences, we will define the \emph{flow} of $\mc{C}$.

\newcommand{\ind}{\ensuremath{\chi_{\mc{C}}}}

If $\mc{C}$ is a color string of length $m$, the \emph{string index of $\mc{C}$}, denoted \ind, is the number of colors needed to color the graph $G_{\alpha}(m; 1,1,1)$ if the colors from the color string $\mc{C}$ are assigned to the spokes. Of course, if there exists a length $m$ string $\mc{C}$ where $\ind = 3$, then $G_{\alpha}(m; 1,1,1)$ has chromatic index 3; however, just because \ind = 4, it does not follow that $G_{\alpha}(m; 1,1,1)$ is a snark. That is, it is clear that:

\begin{lemma} The chromatic index of $G_{\alpha}(m; 1,1,1)$ is the minimum of the string indices \ind taken over all possible color strings $\mc C$ of length $m$.\end{lemma}

For the remainder of this section, $\mc{C}$ is a color string of length $m$, and we will be determining $\ind$ by analyzing properties of certain substrings of $\mc{C}$. 

In what follows, $x, y, z$ correspond to any particular choice of color assignments from the color set $\{1, 2, 3\}$, with the assumption that $x,y,z$ are all distinct colors. 

\begin{lemma}\label{lemOddx}
If $m$ is odd and $\ind = 3$, then there must be an odd number of entries in $\mc{C}$ of each of the colors 1, 2, 3.
\end{lemma}

\begin{proof}
This follows immediately from the Parity Lemma, 
since cutting all the spokes disconnects the loop edges from the rest of the graph.
\end{proof}

\begin{lemma}\label{lem:xyx}
If a color string $\mc{C}$ contains either of the substrings $x\underbrace{y \cdots y}_{\text{odd}, \geq 1}x$ or $x\underbrace{y\cdots y}_{\text{even, } \geq 2}z$, then $\ind{m} = 4$.
\end{lemma}

\begin{proof}Case 1: substrings of the form $x\underbrace{y \cdots y}_{\text{odd}}x$. 

Suppose $\ind = 3$. Recall we index the clusters by $G_{0}, \ldots, G_{m-1}$ and we are coloring the spoke edge for cluster $G_{i}$ as $c_{i}$. Without loss of generality (we can cyclically reindex the string) suppose that the substring we are analyzing starts at cluster 0 and that the substring is length $2 + 2k$ for some $k \geq 0$, that is, the substring is $c_{0}c_{1}\cdots c_{2k+2}$, where $c_{0}= c_{2k+2} = x$ and $c_{i} = y$ for $y = 1, \ldots, 2k+1$.

Consider the coloring of the loop edges $w_{i}w_{i+1}$ in $G_{\alpha}(m; 1,1,1)$ that is induced by the spoke coloring assigned by the substring. Since $c_{0}c_{1} = xy$ and $c_{2k+1}c_{2k+2} = yx$, edges $w_{0}w_{1}$ and  $w_{2k+2}w_{2k+1}$ must both be colored $z$, forcing edges $w_{1}w_{2}$ and $w_{2k+1}w_{2k}$ to both be colored $x$. Then (since $c_{2} = c_{2k} = y$) edges $w_{2}w_{3}$ and $w_{2k}w_{2k-1}$ must both be colored $z$, and continuing in this fashion, edges $w_{i}w_{i+1}$ and $w_{2k+2-i}w_{2k+2-(i+1)}$ must both be colored  $x$ when $i$ is odd and $z$ when $i$ is even. This means that the two loop edges $w_{k}w_{k+1}$ and $w_{2k+2-k}w_{2k+2-(k+1)} = w_{k+2}w_{k+1}$  in the middle of the substring, which are adjacent,  are colored the same color, a contradiction.   

 Figure \ref{fig:G34no9-badStrings-xyx} shows the smallest forbidden substring of this type, $xyx$.

Case 2: substrings of the form $x\underbrace{y \cdots y}_{\text{even, }\geq 2}x$.  Again, suppose that $\ind = 3$ and suppose that the substring we are analyzing starts at cluster 0 and that the substring is length $2 + 2k$ for some $k \geq 1$, that is, the substring is $c_{0}c_{1}\cdots c_{2k+1}$, where $c_{0}=x$,  $c_{2k+1} = z$ and $c_{i} = y$ for $y = 1, 2, \ldots, 2k$.

Since $c_{0}c_{1} = xy$ and $c_{2k}c_{2k+1} = yz$, edges $w_{0}w_{1}$ is colored $z$ and edge  $w_{2k}w_{2k+1}$ is colored $x$, so edges $w_{1}w_{2}$ and $w_{2k+1}w_{2k}$ must  both be colored $x$ and $z$ respectively. Continuing to alternate colors for the loop edges as required, from the front and back of the string, we conclude that edge $w_{i}w_{i+1}$ is colored $z$ when $i$ is even and $x$ when $i$ is odd, while edge $w_{2k+1-i}w_{2k-i}$ is colored $x$ when $i$ is even and $z$ when $i$ is odd. But this means that the middle edge $w_{k}w_{k+1} = w_{2k+1-k}w_{2k-k}$ must be colored both $x$ and $z$, a contradiction.
 
Figure \ref{fig:G34no9-badStrings-xyyz} shows the smallest forbidden substring of this type, $xyyz$.
\end{proof}

Since we can cyclically reindex the color string (or alternately, start applying the color string at cluster $j$):

\begin{corollary}\label{lem:badGlueLoops}
If $\mc{C}$ is a length $m$ string of the form  $x\cdots xy$, $xy \cdots y$, or $yyz \cdots x$, $yz \cdots xy$, $z \cdots xyy$ then $\ind = 4$ as well.\end{corollary}

\begin{figure}[htbp]
\begin{center}
\ffigbox{
\begin{subfloatrow}[2]
\ffigbox{\caption{$xyx$  requires 4 colors}\label{fig:G34no9-badStrings-xyx}}{
\begin{tikzpicture}[node distance=0.5cm]
\drawOneBlob{0}{gray, thin}{gray,thin}{red}{gray,thin}{gray,thin}\drawOneBlob{2}{gray,thin}{gray,thin}{blue, dotted}{gray,thin}{gray,thin}\drawOneBlob{4}{gray,thin}{gray,thin}{red}{gray,thin}{gray,thin}
\node[left=of z0] (start) {};
\node[right=of z4] (end) {};
\draw[ultra thick, green, dashed] (z0) -- (z2);
\draw[orange, decorate, decoration={
zigzag, segment length=3}] (z2) -- (z4);
\end{tikzpicture}
}

\ffigbox{\caption{$xyyz$  requires 4 colors}\label{fig:G34no9-badStrings-xyyz}}{
\begin{tikzpicture}[scale=.6, node distance=0.5cm]
\drawOneBlob{0}{gray, thin}{gray,thin}{red}{gray,thin}{gray,thin}
\drawOneBlob{3}{gray,thin}{gray,thin}{blue, dotted}{gray,thin}{gray,thin}
\drawOneBlob{6}{gray,thin}{gray,thin}{blue, dotted}{gray,thin}{gray,thin}
\drawOneBlob{9}{gray,thin}{gray,thin}{green, dashed}{gray,thin}{gray,thin}

\foreach \i in {0,3,6,9}{
\draw  (z\i) node[circle, inner sep = 1, fill = black] {} ;}
\node[left=of z0] (start) {};
\node[right=of z9] (end) {};
\draw[ultra thick, green, dashed] (z0) -- (z3);
\draw[ultra thick, red] (z6) -- (z9);
\draw[orange, decorate, decoration={
zigzag, segment length=3}] (z3) -- (z6);
\end{tikzpicture}
}

\end{subfloatrow}
}{
\caption{The substrings $xyx$ and $xyyz$  are forbidden in a color string $\mc{C}$ with $\ind = 3$,\ 
 since those  spoke colors require four colors in the loop edge colorings (shown with zig-zags).}
\label{fig:G34no9-badStrings}}
\end{center}
\end{figure}

Following the same notation as before, given a color string $\mc{C}$ of length $r$, a \emph{color pattern assignment for $\mc{C}$} is an ordered pair of ordered triples $[(x_{1}, x_{2}, x_{3}), (y_{1}, y_{2}, y_{3})]_{\mc{C}}$ that is the color pattern assignment for the corresponding cluster sequence $G'_{r}$. 
which has colors $(x_{1}, x_{2})$  assigned (top to bottom) to the input pair of the cluster sequence $G'_{r}$, color $x_{3}$ assigned to the input loop edge,  colors  $c_{0}, \ldots, c_{r-1}$ from $\mc{C}$  assigned to the spokes of $G'_{r}$,  colors $(y_{1}, y_{2})$  assigned to the output pair of $G'_{r}$ (top to bottom), and color $y_{3}$ assigned to the output loop edge. 
Figure \ref{fig:colorPatternAssignmentExample} shows that $[(1,3,2), (3,2,1)]$ is a valid color pattern assignment for the string $123322$ (where 1=red, 2=blue, 3=green).
 
 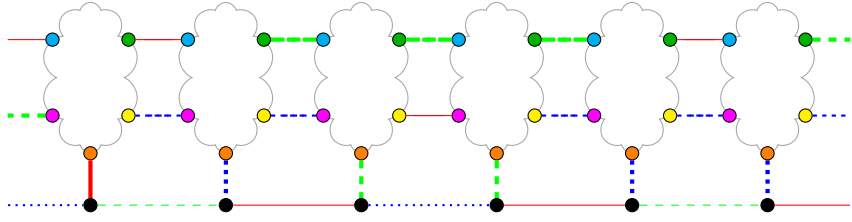
\begin{figure}[htbp]
\begin{center}
\begin{tikzpicture}[scale=.6, node distance=0.5cm]
\drawOneBlob{0}{red, thin}{green,dashed}{red}{blue,dotted, thick}{red,thin}
\drawOneBlob{3}{red,thin}{blue,dotted, thick}{blue, dotted}{blue,dotted, thick}{green,dashed}
\drawOneBlob{6}{green,dashed}{blue,dotted, thick}{green, dashed}{red,thin}{green,dashed}
\drawOneBlob{9}{green,dashed}{red,thin}{green, dashed}{blue,dotted, thick}{green,dashed}
\drawOneBlob{12}{green,dashed}{blue,dotted, thick}{blue, dotted}{blue,dotted, thick}{red,thin}
\drawOneBlob{15}{red,thin}{blue,dotted, thick}{blue, dotted}{blue,dotted, thick}{green,dashed}

\foreach \i in {0,3,6,9, 12, 15}{
\draw  (z\i) node[circle, inner sep = 1, fill = black] {} ;}
\node[left=1 cm of z0] (start) {};
\node[right= 1 cm of z15] (end) {};
\draw[blue, dotted, thick](z6)--(z9)  (start) -- (z0) ;
\draw[red] (z3) -- (z6) (z9) -- (z12) (z15) -- (end);
\draw[green, dashed]  (z0) -- (z3)  (z12)--(z15);
\end{tikzpicture}

\caption{$[(1,3,2), (3,2,1)]$ is a valid color pattern assignment for the string $123322$.}
\label{fig:colorPatternAssignmentExample}
\end{center}
\end{figure}

 We define the \emph{flow} of a color string $\mc{C}$, denoted $\Flow{\mc C}$, to be the set of all possible color pattern assignments for $\mc{C}$.
If $\Flow{\mc C}=\varnothing$, then there are no possible color patterns associated with $\mc C$.
 
 \begin{observation} If the color string $\mc{C}$ does not contain one of the forbidden substrings from Lemma \ref{lem:xyx} and contains two consecutive distinct entries, then the entries $x_{3}$ and $y_{3}$, corresponding to possible colors for the input and output loop edges, are completely determined. If the color string $\mc{C}$ has all entries the same (e.g., $\mc{C} = 1  1 \cdots 1)$, then $x_{3}$ and $y_{3}$ each can be assigned either one of the other two colors. \end{observation}
 
 \begin{proof} If a color string contains a substring $c_{i}c_{i+1} = xz$, then the loop edge between those two spokes, $w_{i}w_{i+1}$, must be colored $y$, and the loop edges $w_{i-1}w_{i}$ and $w_{i+1}w_{i+2}$ must be colored $z$ and $x$, respectively. We can continue to propagate the colors on the loop edges out from spoke $i$.
\end{proof}
 
 Typically, we determine the flow of a given (short) color string $\mc{C}$ of length $r$ by assigning the colors from $\mc{C}$ to the spokes of $G'_{r}$, analyzing the assignment of the loop edge colors induced by the color string, and then systematically testing all of the 9 possible pairs of colors $(x_{1}, x_{2})$, $x_{i} \in \{1,2,3\}$ that could be assigned as input pairs to see which of them can be completed to a proper 3-coloring of the entire cluster sequence, recording the valid output colors. 
 
 For example, the color string $xyzzyy$ forces the input loop color to be $y$ and the output loop color to be $z$, and (see Table \ref{t:xyzzyy} for details)
 \[ \Flow{xyzzyy} = \{ [(x,x,y), (x,y,z)], [(x,x,y),(y,x,z)],  [(x,z,y), (z,y,z)], [(z,x,y), (y,z,z)] \}. \]

\begin{lemma}\label{lem:xyxFlow}
If a color string $\mc{C}$ contains a substring of the form $x\underbrace{y \cdots y}_{2k+1} x$ or $x\underbrace{y \cdots y}_{2k}z$, then $\Flow{\mc{C}} = \varnothing$.
\end{lemma}

\begin{proof} This follows similarly to Lemma \ref{lem:xyx}; assigning those colors to the spokes forces one of the internal loop edges to require a fourth color. \end{proof}

\begin{lemma}\label{condition3}
Suppose $\mc{C}$ is a color string of length $m$. The string index $\ind = 3$ if and only if $\Flow{ \mc{C} }$ contains a pair of the form $[(x_{1},x_{2}, x_{3}),(x_{1},x_{2}, x_{3})]$.
\end{lemma}

\begin{proof} 
If $\Flow{\mc{C}}$ contains a pair of the form $[(x_{1},x_{2}, x_{3}),(x_{1},x_{2}, x_{3})]$ then we can glue together the correspondingly-colored input and output edges to produce a coloring of $G_{\alpha}(m; 1,1,1)$ using only three colors.

Conversely, if $\ind = 3$, then clipping the 3-edge-colored version of $G_{\alpha}(m; 1,1,1)$ (where the spokes are colored with $\mc{C}$) between clusters $G_{0}$ and $G_{m-1}$ produces the required element of $\Flow{\mc C}$, by recording the colors of the clipped connecting edges. Lemmas \ref{lem:xyx} and \ref{lem:badGlueLoops} assert that $\mc{C}$ does not contain the forbidden substrings.
\end{proof}

\begin{definition}
We say that two color strings ${\mc C}_{r}$ and ${\mc C}_{s}$ are \emph{functionally equivalent}, denoted ${\mc C}_{r}\sim {\mc C}_{s}$, if 
$\Flow{{\mc C}_r}=\Flow{{\mc C}_s}$.
\end{definition}

We denote the \emph{concatenation} of two (not necessarily functionally equivalent) color strings 
by juxtaposition; that is, the concatenation of ${\mc C}_{r}$ and ${\mc C}_{s}$ is represented by $\mc{C}_{r}\mc{C}_{s}$.

An easy observation is
\begin{observation}\label{obs:flows} If 
$\Flow{\mc{C}_{r}\mc{C}_{s}}\neq \varnothing$, then 
\begin{multline}\text{Flow}({\mc C}_{r}{\mc C}_{s}) =  \{[(x_{1},x_{2}, x_{3}),(z_{1},z_{2}, z_{3})]\mid\\ \exists (y_{1},y_{2}, y_{3}) \text{ s.t. }[(x_{1},x_{2}, x_{3}),(y_{1},y_{2}, y_{3})]\in \text{Flow}({\mc C}_{r}) \text{ and }
[(y_{1},y_{2}, y_{3}),(z_{1},z_{2}, z_{3})]\in \text{Flow}({\mc C}_{s}) \}.\end{multline}
\end{observation}

\begin{proof} Certainly, if $[(x_{1},x_{2}, x_{3}),(z_{1},z_{2}, z_{3})] \in 
\Flow{\mc{C}_{r}\mc{C}_{s}}$, then the colors of the  edges connecting clusters $G_{r}$ and $G_{r+1}$ gives the required element $(y_{1}, y_{2}, y_{3})$. Conversely, if  $[(x_{1},x_{2}, x_{3}),(y_{1},y_{2}, y_{3})] \in \Flow{\mc{C}_{r}}$ and $[(y_{1},y_{2}, y_{3}),(z_{1},z_{2}, z_{3})]\in \text{Flow}({\mc C}_{s})$, then gluing together the two colored cluster sequences gives a 3-coloring of $\mc{C}_{r}\mc{C}_{s}$.\end{proof}

The basic idea we will use to show that a graph $G_{\alpha}(m; 1,1,1)$ is a snark is to look at a possible color string $\mc{C}$ and argue that we can pull out a cluster sequence from the graph and replace it with a shorter cluster string that has the same string index. Specifically:

\begin{lemma}\label{l:replaceSubstrings} 

If $\mc{C} = \mc{C}_{r}\mc{C}_{t}$, 
 $\mc{C}' = \mc{C}_{s}\mc{C}_{t}$, 
and $\mc{C}_{r} \sim \mc{C}_{s}$
, then $\chi_{\mc{C}} = \chi_{\mc{C}'}$.
\end{lemma}

\begin{proof}

Suppose 
that $\mc{C} = \mc{C}_{r}\mc{C}_{t}$, 
 $\mc{C}' = \mc{C}_{s}\mc{C}_{t}$, 
and  $\chi_{\mc{C}} \neq \chi_{\mc{C'}}$. Without loss of generality, suppose $\chi_{\mc{C}} = 3$ and $\chi_{\mc{C'}} = 4$.  
Since $\chi_{\mc{C}'} = 4$, 
$\Flow{\mc{C}_{s}\mc{C}_{t}}$ does not contain an element of the form $[(x_{1}, x_{2}, x_{3}), (x_{1}, x_{2}, x_{3})]$ (by Lemma \ref{condition3})
but since $\chi_{\mc{C}} = 3$, 
$\Flow{\mc{C}_{r}\mc{C}_{t}}$ 
does contain such an element. Hence by Observation \ref{obs:flows}, there exists a triple $(y_{1}, y_{2}, y_{3})$ so that $[(x_{1}, x_{2}, x_{3}), (y_{1}, y_{2}, y_{3})] \in \Flow{\mc{C}_{r}}$ and $[(y_{1}, y_{2}, y_{3}), (x_{1}, x_{2}, x_{3})] \in \Flow{\mc{C}_{t}}$. But since 
$\Flow{\mc{C}_{s}\mc{C}_{t}}$ doesn't contain $[(x_{1}, x_{2}, x_{3}), (x_{1}, x_{2}, x_{3})]$, it follows that $[(x_{1}, x_{2}, x_{3}), (y_{1}, y_{2}, y_{3})] \notin \Flow{\mc{C}_{s}}$. Therefore $\mc{C}_{r}$ and $\mc{C}_{s}$ are not functionally equivalent.
\end{proof}

\begin{observation} If $\mc{C}_{r} \sim \mc{C}_{s}$, then the two color strings formed by reversing the entries are also functionally equivalent.\end{observation}

\begin{lemma}[Reductions Lemma]\label{collapses} The following color strings are functionally equivalent.
\be
\item $xxx\sim x$\label{i:xxx}
\item $xyzzyy\sim xyzz$\label{i:xyzzyy}
\item $xyzxyy\sim xyzx$ \label{i:xyzxyy}
\item $xyzzyxx\sim xyzzy$  \label{i:xyzzyxx}
\item $xyzzyxz\sim xyzxy$ \label{i:xyzzyxz}
\item $(xyz)^{4} \sim (xyz)^{2}$ \label{xyz4}
\ee
\end{lemma}
\begin{proof}
We first determine what the possible colors for the loop input and output edges are, based on the spoke assignments, and observe that in all cases, the input and output colors for the loop edges for the pairs of strings can be chosen to be the same; see Table \ref{tab:loopColors}. Therefore, we can suppress recording the colors of the loop edges in later analysis.

\begin{table}[htbp]
\caption{Possible loop colors for the input and output loop edges for strings in \autoref{collapses}.}
\label{tab:loopColors}
\begin{tabular}{c |c c || c| c c}
string & input loop color & output loop color& string & input loop color & output loop color\\ \hline
$xxx$ &$y$  &$z$& $x$ &$y$&$z$ \\
$xxx$ & $z$ & $y$& $x$ &$z$&$y$ \\
$xyzzyy$&$y$&$x$& $xyzz$ &$y$&$x$ \\
$xyzxyy$ &$y$&$z$ & $xyzx$ &$y$&$z$\\
$xyzzyxx$ &$y$& $z$&$xyzzy$ &$y$&$z$\\
$xyzzyxz$ & $y$& $x$ & $xyzxy$ & $y$ & $x$ \\
$(xyz)^{4}$ &$z$&$y$ &  $(xyz)^{2}$ &$z$ & $y$
\end{tabular}
\end{table}

Each case involves seeing which of the ordered pairs of $\{(x,x),\allowbreak (y,y),\allowbreak (z,z),\allowbreak (x,y),\allowbreak (y,x),\allowbreak (x,z),\allowbreak (z,x),\allowbreak (y,z),\allowbreak (z,y)\}$ occur as possible colors for the input pair, and what the induced colors on the output pairs must be, when possible. 
For a cluster with spoke edge colored $x$, input $(x,y)$  is only seen in the {\sf AltTop} pattern, so the only possible coloring of the output edges for that cluster is $(x,z)$, while $(x,x)$ corresponds to the {\sf TwoLeft} pattern, which has two possible outputs, $(y,z)$ and $(z,y)$. For compactness in our presentation, we will preserve the order of alternatives in subsequent columns, unless they collapse; e.g., if $(y,z),(z,y)$ appears to the left of a column with spoke color $y$, the first pair gives output color $(y,x)$ and the second gives output color $(x,y)$ and  we write ``$(y,x),(z,y)$'' to indicate this. On the other hand if $(y,z),(z,y)$ appears to the left of a column with spoke color $x$, then both input colors yield $(x,x)$ as an output color, and so we simply write $(x,x)$ in the succeeding column. If a pair of input colors is impossible given a certain spoke color (e.g., $(x,x)$ is not an admissible labeling of the input edges of a cluster with spoke edge colored $z$)  we write ``---'' for its output position. 

We present analyses of two specific reductions here, in Tables \ref{t:xxx} and \ref{t:xyzzyy}, first showing that $xxx \sim x$ and then that $xyzzyy \sim xyyz$; the proof of the remaining reductions is (tediously) similar. We simplify the presentation of the analysis by storing the analysis in an array as follows: the first column of the array contains all possible assignments of colors for the input pairs. Subsequent columns are labelled with the spoke color of the corresponding cluster,  and entries in each column correspond to the possible output pairs that can be obtained from the input listed at the left after traversing that cluster. Loop colors are suppressed entirely.
\begin{table}[htp]
\caption{The color string $xxx$ is functionally equivalent to $x$. 
}
\begin{center}
\begin{tabular}{c|a|c|a}
\backslashbox{inputs}{spoke colors}& $x$&$x$&$x$\\ \hline
$(x,x)$&$(y,z)$, $(z,y)$&$(x,x)$&$(y,z)$, $(z,y)$\\ 
$(y,y)$&---&&\\
$(z,z)$&---&&\\
$(x,y)$&$(x,z)$&$(x,y)$&$(x,z)$\\ 
$(y,x)$&$(z,x)$&$(y,x)$&$(z,x)$\\ 
$(x,z)$&$(x,y)$&$(x,z)$&$(x,y)$\\ 
$(z,x)$&$(y,x)$&$(z,x)$&$(y,x)$\\ 
$(y,z)$&$(x,x)$&$(y,z)$, $(z,y)$&$(x,x)$\\ 
$(z,y)$&$(x,x)$&$(y,z)$, $(z,y)$&$(x,x)$
\end{tabular}
\end{center}
\label{t:xxx}
\end{table}%
Observe that in Table \ref{t:xxx} the first and last columns (highlighted) match, and since the strings have the same possible colors for the loop edges, shown in the first two rows of Table \ref{tab:loopColors},  $\Flow{xxx}=\Flow{x}$. 

In Table \ref{t:xyzzyy}, observe that the second to last and last columns of the array agree, and the loop colors from the 3rd row of Table \ref{tab:loopColors} also agree, so $\Flow{xyzzyy}=\Flow{xyzz}$, and thus $xyzzyy$ and $xyzz$ are functionally equivalent.

\begin{table}[htp]
\caption{The substring $xyzzyy$ is functionally equivalent to $xyzz$.
}
\begin{center}
\resizebox{\columnwidth}{!}{%
\begin{tabular}{c|c|c|c|a|c|a}
\backslashbox{inputs}{spoke colors}& $x$&$y$&$z$&$z$&$y$&$y$\\\hline
$(x,x)$&$(y,z)$, $(z,y)$&$(y,x)$, $(x,y)$&$(z,z)$&$(x,y)$ ,  $(y,x)$&$(z,y)$, $(y,z)$&$(x,y)$, $(y,x)$\\
$(y,y)$&---&&&&&\\
$(z,z)$&---&&&&&\\
$(x,y)$&$(x,z)$&$(y,y)$&---&&&\\
$(y,x)$&$(z,x)$&$(y,y)$&---&&&\\
$(x,z)$&$(x,y)$&$(z,y)$&$(z,x)$&$(z,y)$&$(x,y)$&$(z,y)$\\
$(z,x)$&$(y,x)$&$(y,z)$&$(x,z)$&$(y,z)$&$(y,x)$&$(y,z)$\\
$(y,z)$&$(x,x)$&---&&&&\\
$(z,y)$&$(x,x)$&---&&&&
\end{tabular}}
\end{center}
\label{t:xyzzyy}
\end{table}%
\end{proof}

\begin{corollary}In the color string $\mc{C}$ any subsequence $xx\cdots x$ of odd length is functionally equivalent to a single $x$.
\label{lemxxx}
\begin{proof} This follows immediately by repeated application of reduction \ref{i:xxx} of  \autoref{collapses}.
\end{proof}
\end{corollary}
\begin{lemma}\label{lemxyz}If color string $\mc{C}$ has odd length and has no consecutive odd sequences of the same color of length more than 1, then $\mc C$ must contain a color substring of the form $xyz$, $xy \cdots z$ or $x\cdots yz$, up to permutation of the $x,y,z$. 
\end{lemma}

\begin{proof}
If $\mc{C}$ has no same-color subsequences of odd length larger than 1, then it must contain consecutive subsequences of only length 1 or 2. If there is no sequence $xyz$ (up to permutation of the $x, y, z$), taken cyclically, then all runs must be of length 2, so there are an even number of $x$, $y$, and $z$ entries in $\mc{C}$, which contradicts Lemma \ref{lemOddx}.
\end{proof}


Finally, we can prove Theorem \ref{thm:G34no9snark}.

\begin{proof}[Proof of Theorem \ref{thm:G34no9snark}]
We proceed by induction on odd $m$. Observe that to prove Theorem \ref{thm:G34no9snark}, it suffices to show that $\ind = 4$ for all color strings $\mc{C}$ of length $m$.

{\bf Base Cases: $m = 3, 5, 7,9$}

For $m = 3, 5, 7, 9$, we (via computer) listed all possible color strings, beginning with $123$, of odd length up to 9. By Lemma \ref{lemxyz} we then removed all strings that contained the forbidden substrings of the form $xyx$, taken cyclically (i.e., we eliminated $xy\cdots y$ and $x\cdots xy$ as well).

Next, we iteratively applied the substring reductions $xxx$ to $x$, $xyzzyy$ to $xyzz$, and $xyzxyy$ to $xyzx$ until the resulting strings were fixed, eliminating strings that contained $xyx$ and $xyyz$ after each reduction, until we had all valid strings of length at most 9 that could not be reduced.
There are two irreducible color strings ${\mc C}$, with $m$ odd and at most 9 up to assignment of colors:  $123$ and $123123123$. It is straightforward to show that $\chi_{123} = 4$ and $\chi_{123123123}= 4$. To see this requires inspection of the flows corresponding to these color strings. For string 123, the input loop color must be 2, and the output loop color must be 2, and likewise for $123123123$. Therefore, we must analyze possible colors for the input and output pairs. First, consider the constraints imposed on the edge colorings shown in Table \ref{t:xyz}.
\begin{table}[htp]
\caption{Analysis of the admissible color assignments for the input and output pairs for $123$. }
\begin{center}
\begin{tabular}{c|c|c|c}
\backslashbox{inputs}{dangling edges}&1&2&3\\
\hline
\rowcolor{lightgray}(1,1)&(2,3), (3,2)&(2,1), (1,2)& (3,3)\\
(2,2)&---&&\\
(3,3)&---&&\\
(1,2)&(1,3)&(2,2)&---\\
(2,1)&(3,1)&(2,2)&---\\
\rowcolor{lightgray}(1,3)&(1,2)&(3,2)&(3,1)\\
\rowcolor{lightgray}(3,1)&(2,1)&(2,3)&(1,3)\\
(2,3)&(1,1)&---&\\
(3,2)&(1,1)&---&
\end{tabular}
\end{center}
\label{t:xyz}
\end{table}%
We immediately 
conclude that $\chi_{123} = 4$, since \Flow{123} does not contain any input and output pairs with equal entries. In Table \ref{t:xyzxyzxyz} we consider the effect of repeated application of $123$, and see that there are only two admissible assignments to the input pair, but again, they do not agree with the output pair. Therefore, $\chi_{123123123} = 4$. Since all other possible color strings for $m = 3, 5, 7, 9$ can be reduced via functionally equivalent reductions, which preserve string index, to the string $123$, it follows that $\chi= 4$ for all odd $m \leq 9$. 

\begin{table}[htp]
\caption{Analysis of the admissible color assignments to the input and output pairs of edges for the color string $123123123$. 
}
\begin{center}
\begin{tabular}{c|c|c|c}
\backslashbox{inputs}{blocks}&123&123&123\\
\hline
(1,1)& (3,3)&---&\\
\rowcolor{lightgray}(1,3)&(3,1)&(1,3)&(3,1)\\
\rowcolor{lightgray}(3,1)&(1,3)&(3,1)&(1,3)
\end{tabular}
\end{center}
\label{t:xyzxyzxyz}
\end{table}%

{\bf Induction Hypothesis}: Assume that every color string $\mc{C}$ of odd length at most $m$ has
$\chi_{\mc{C}} = 4$.  

Now let $\mc{C}$ be any color string of odd length $m+2$, for $m \geq 9$.
\begin{enumerate}[{\bf {Case} 1:}]
\item  Suppose that $\mc{C}$ contains a substring consisting of an odd number of consecutive repeated entries (say $x\cdots x$); then we can replace that substring with a single entry (say $x$) forming a smaller color string $\mc{C'}$ with odd length $m' \leq m$. By Corollary \ref{lemxxx} and Lemma \ref{l:replaceSubstrings}, we conclude that $\chi_{\mc{C}} = \chi_{\mc{C'}}$. By the induction hypothesis, $\chi_{\mc{C'}}= 4$, so it follows that $\chi_{\mc{C}} = 4$. 

\item Suppose next that $\mc{C}$ does not contain any consecutive same-color substrings of odd length 3 or more.  By Lemma \ref{lemxyz} we may assume (possibly after reindexing) that $\mc{C}$
contains the string $(xyz)^k$  for some $k\geq 1$. 

If $k\geq 4$, by reduction (\ref{xyz4}) of Lemma \ref{collapses} we can replace $(xyz)^k$ by the substring $(xyz)^{k-2}$ to yield a new string $\mc{C'}$ of length $m+2-6 \leq m$. By the induction hypothesis, $\chi_{\mc{C'}} = 4$ and since $\mc{C} \sim \mc{C'}$, it follows that $\chi_{\mc{C}} = 4$.  

\item Finally, suppose that $(xyz)^k$ ocurrs only for $1\leq k\leq 3$.  
Cyclically shifting if necessary, we can assume that 
\[ \mc{C} = (xyz)^{k-1}\underbrace{x\ y\ z \ c_{3k+1}\ c_{3k+2} \ldots c_{m+2}}_{\text{length} \geq 5}.\]

Recall $m\geq 9$. 
We analyze what $x\ y\ z \ c_{3k+1}\ c_{3k+2} \ldots c_{m+2}$ can look like. Recall that subsequences of the form $xyx$ (or in this case, $yzy$) force the string index to equal 4. If we analyze all ways to choose $c_{3k+1}$ and $c_{3k+2}$, observe that $c_{3k+1} = x$ or $z$.

There are thus six possible subsequences of length $5$ containing $xyz$, three of which can possibly be extended further:
\bi
\item $xyzxx$. If $m= 11$, then $\mc{C} = (xyz)^2xyzxx$ which contains the subsequence $x\cdots xx$ forcing $\chi_{\mc{C}} = 4$. If $m >11$ then we can extend $xyzxx$ further.
\item $xyzxy$.  If $m= 11$, then $\mc{C} = (xyz)^2xyzxy$ which contains the subsequence $x \cdots xy$ forcing $\chi_{\mc{C}} = 4$; if $m >11$ then we can extend $xyzxy$ further.
\item $xyzxz$ forces $\chi_{\mc{C}} = 4$ because of $zxz$.
\item $xyzzx$ forces $\chi_{\mc{C}} = 4$ because of $yzzx$. 
\item $xyzzy$. If $m= 11$, then $\mc{C} = (xyz)^2xyzzy$ which contains the subsequence $xy \cdots xy$, forcing $\chi_{\mc{C}} = 4$; if $m >11$ then we can extend $xyzzy$ further.
\item $xyzzz$ can't occur, because we assumed we had removed all runs of an odd length at least 3 already, before we reached this case.
\ei

Extending the string by one more step, we list all possible strings of length $6$ or larger starting with $xyz$. The only resulting strings that do not contain a substring that force the string index to be 4 (i.e., $xyx$ or $xyyz$, up to permutation) and do not contain a run of at least three consecutive same-valued entries, which were supposed to have been eliminated in an earlier step, are $xyzxxz$,  $xyzxyy$ and $xyzzyx$, and $xyzzyy$; note that $xyzxyz$ cannot appear, since we already assumed that we were beginning with the last $xyz$ in the subsequence $(xyz)^{k}$. Observe that if $\mc{C}$ contains $xyzxyy$, then we can apply reduction \eqref{i:xyzxyy} to get a smaller string $C'$ which, by the induction hypothesis, has $\chi_{C'} = 4$. Similarly, $xyzzyy$ reduces to a smaller string $C'$ by reduction \eqref{i:xyzzyy}, which must also have $\chi_{C'} = 4$ by the induction hypothesis. 

The remaining two strings $xyzxxz$ and $xyzzyx$ again can be extended, again eliminating extensions that contain 3 or more consecutive entries or substrings which force string index 4: 
$xyzxxz$ can be extended by $z$ or $y$, and $xyzzyx$ can be extended by $x$ or $z$. However, in each case, $\mc{C}$ may be reduced by replacing a substring by a shorter functionally equivalent substring, yielding a smaller substring:
\bi
\item $xyzxxzz = x (yzxxzz)$ and $yzxxzz \sim yzxx$ by reduction \eqref{i:xyzzyy} and permutation $(y,x,z)$;
\item $xyzxxzy \sim  yxzyx$
by the reverse of reduction \eqref{i:xyzzyxz}  and  permutation $(x,z,y)$  ;
\item $xyzzyxx \sim xyzzy$ by reduction \eqref{i:xyzzyxx} 
\item $xyzzyxz \sim xyzxy$ by reduction \eqref{i:xyzzyxz}
\ei

By Lemma \ref{l:replaceSubstrings} and the induction hypothesis, it follows that $\chi_{\mc{C}} = 4$. 
\end{enumerate}
\end{proof}

\begin{corollary}\label{thm:G34no9snark-multiply}
If $\gcd(m, a) = 1$ and $m$ is odd, then $G_{\alpha}(m; a,a,a)$ is a snark.\end{corollary}

\begin{proof} This follows as an immediate consequence of Theorem \ref{thm:G34no9snark} and Proposition \ref{thm:voltageFacts}.
\end{proof}

Observe in Figure \ref{fig:Gno9cycCon} that the highlighted edges form a cycle-separating set for the voltage graph;  clipping the arrow from $B'$ to $B$ corresponds to clipping two arrows in the lift. That is:

\begin{theorem} The graph $G_{\alpha}(m; 1,1,1)$ is cyclically 4-connected for $m > 3$.\end{theorem}

(When $m=3$, clipping the spokes separates the central triangle in $G_{\alpha}(3; 1,1,1)$  from the rest of the graph, so that graph is only cyclically 3-connected.)

In addition, we are interested in the  oddness of the snark. 

\begin{theorem}
The oddness of $G_{\alpha}(2k+1; 1,1,1)$  is 2.
\end{theorem}

\begin{proof}
By Theorem \ref{thm:G34no9snark}, $G_{\alpha}(2k+1;1,1,1)$ is a snark, so the oddness is at least 2. For $2k+1 \geq 3$, we exhibit a 2-factorization of $G_{\alpha}(2k+1;1,1,1)$ with precisely two odd cycles and no even cycles, shown as a cycle in the voltage graph in Figure \ref{fig:Gno9oddness}. 
\end{proof}

\begin{figure}[htbp]
\begin{center}
\ffigbox{
\begin{subfloatrow}
\ffigbox{ \caption{A cycle-separating cut set in the voltage graph $\tilde{G}_{\alpha}(a,b,c)$, shown in red, shows that the cyclic edge-connectivity of $G_{\alpha}(m; a,b,c)$ is at most 4. Note the arrow from $B$ to $B'$ corresponds to two edges in the lift graph (that is, $B'_{i}B_{i+b}$ and $B_{i}B'_{i-b}$ are both part of the cut set).}\label{fig:Gno9cycCon}}{
\begin{tikzpicture}[scale=1.5]
	\node[circle, gray!70!white, minimum width = 2 cm, inner sep = 1pt] (G) at ($(90:2)$){};
	\node[vtx, fill = gray!50!white] (x1) at (G.90) {};
	\node[vtx, fill = gray!50!white] (x4) at (G.90-72) {};
	\node[vtx,white, fill = cyan] (B) at (G.90-2*72) {$B$};
	\node[vtx,fill = gray!50!white] (x2) at (G.90+72) {};
	\node[vtx, white,fill = green!70!black] (B') at (G.90+2*72) {$B'$};
	\node[vtx, fill = gray!50!white] (x6) at ($(0,0)!1.3!(x2)$) {};
	\node[vtx, fill = gray!50!white] (x5) at($(0,0)!1.2!(x1)$) {};
	\node[vtx, fill = gray!50!white] (x3) at ($(0,0)!1.3!(x4)$)  {};
	
	\node[vtx, fill = yellow] (A') at ($(90-30:2*1.2)$){$A'$};
	\node[vtx, white,fill = mymagenta] (A) at ($(30+90:2*1.2)$){$A$};
	\node[vtx,white, fill = orange] (C) at (90:.5){$v$};
	\node[vtx, below = .75 of C, white, fill = black] (z) {$w$};

	\draw[gray!70!white] (B') -- (x4) -- (x2) -- (B) -- (x1)--(B');
	\draw[gray!70!white] (x2) -- (x6) -- (x5) -- (x3) -- (x4);
	\draw[gray!70!white] (x6) -- (A) -- (C) -- (A')--(x3);
	\draw[gray!70!white]  (x1)--(x5);
	\draw[gray!70!white] (C) -- (z);
	
	\draw[red, ultra thick] (A) -- (x6) (A') -- (x3);
	
	\begin{scope}[on background layer]
	\draw[<-,  looseness = 2] (A) to[out=90, in=90]  node[midway, arclabel] {$a$} (A');
	\draw[<-, ultra thick, red] (B) to[bend left=30] node[near start, arclabel] {$b$} (B');
	\draw[->] (z) arc(90:360+70:0.2) node[midway, arclabel] {$c$}  (z);
	\end{scope}
\end{tikzpicture}
}

\ffigbox{ \caption{A 2-factor in $\tilde{G}_{\alpha}(1,1,1)$ which lifts to a pair of odd cycles in $G_{\alpha}(m; 1,1,1)$; the blue cycle lifts to a single cycle of length $11m$, while the magenta loop lifts to a single odd cycle of length $m$. Hence the oddness is 2.}\label{fig:Gno9oddness}}{
\begin{tikzpicture}[scale=1.5]
	\node[circle, gray!70!white, minimum width = 2 cm, inner sep = 1pt] (G) at ($(90:2)$){};
	\node[vtx, fill = gray!50!white] (x1) at (G.90) {1};
	\node[vtx, fill = gray!50!white] (x4) at (G.90-72) {4};
	\node[vtx,white, fill = cyan] (B) at (G.90-2*72) {$B$};
	\node[vtx,fill = gray!50!white] (x2) at (G.90+72) {2};
	\node[vtx, white,fill = green!70!black] (B') at (G.90+2*72) {$B'$};
	\node[vtx, fill = gray!50!white] (x6) at ($(0,0)!1.3!(x2)$) {6};
	\node[vtx, fill = gray!50!white] (x5) at($(0,0)!1.2!(x1)$) {5};
	\node[vtx, fill = gray!50!white] (x3) at ($(0,0)!1.3!(x4)$)  {3};
	
	\node[vtx, fill = yellow] (A') at ($(90-30:2*1.2)$){$A'$};
	\node[vtx, white,fill = mymagenta] (A) at ($(30+90:2*1.2)$){$A$};
	\node[vtx,white, fill = orange] (C) at (90:.5){$v$};
	\node[vtx, below = .75 of C, white, fill = black] (z) {$w$};

	\draw[gray!70!white] (B') -- (x4) -- (x2) -- (B) -- (x1)--(B');
	\draw[gray!70!white] (x2) -- (x6) -- (x5) -- (x3) -- (x4);
	\draw[gray!70!white] (x6) -- (A) -- (C) -- (A')--(x3);
	\draw[gray!70!white]  (x1)--(x5);
	\draw[gray!70!white] (C) -- (z);
	
	\draw[ultra thick, blue] (B') -- (x4) -- (x2) -- (x6) -- (A) -- (C) -- (A') -- (x3) -- (x5) -- (x1) -- (B);
	
	\begin{scope}[on background layer]
	\draw[<-, looseness = 2] (A) to[out=90, in=90]  node[midway, arclabel] {$1$} (A');
	\draw[<-, ultra thick, blue] (B) to[bend left=30] node[near start, arclabel] {$1$} (B');
	\draw[->, ultra thick, mymagenta] (z) arc(90:360+70:0.2) node[midway, arclabel] {$1$}  (z);
	\end{scope}
\end{tikzpicture}
}

\end{subfloatrow}

}{
\caption{Properties of the graph $G_{\alpha}(2k+1;1,1,1)$}
\label{fig:properties}
}
\end{center}
\end{figure}

The snark on 34 vertices shown in Figure \ref{fig:G34No9}, $G_{\alpha}(3;1,1,1)$ with the central triangle contracted to a point, also has oddness 2; it is reasonably straightforward to find a decomposition using two odd cycles (and one even cycle) which includes the central vertex.

\begin{remark} Note that nothing in Theorem \ref{thm:G34no9snark} depended on the specific interconnections inside the graph $G$, but rather on the possible admissible color patterns associated with the cluster. Therefore, any other 5-pole with the same restrictions on the possible admissible color patterns (that is, the admissible color patterns are precisely {\sffamily AltTop}, {\sffamily AltBot}, {\sffamily TwoRight} and {\sffamily TwoLeft}) will also be a snark, using the same argument, in the same way that Loupekine's argument extends to any 5-pole that has the match/mismatch property for its collection of admissible color patterns.\end{remark}

\begin{remark}
In the language of admissible color patterns, of course, Loupekine's construction doesn't depend on starting with a snark at all; it only depends on a 5-pole that possesses the ``match/mismatch'' property. That is, the only admissible color patterns  are $[(x,x), (y,z)]$ or $[(x,x), (z,y)]$ (spoke colored $x$) or $[(x,x),(y,x)]$ or $[(x,x),(x,y)]$ (spoke colored $z$), or their reverses. Any 5-pole $G$ that has this property will produce a snark, by the same arguments as in Observation \ref{LoupekineSnarks} and Proposition \ref{thm:LoupFam}, whether or not the starting graph was a snark. Snarks simply provide a convenient class of graphs with this property. 
\end{remark}


\section{Chromatic index of other graphs of type $G_{\alpha}(m;a,b,c)$ and $G_{\beta}(m;a,b,c)$, and future directions}

We do not have a complete characterization of the chromatic index of graphs $G_{\alpha}(m;a,b,c)$  and $G_{\beta}(m;a,b,c)$ for all possible voltage assignments $a, b, c$. However, we do have some partial results towards such a characterization. In general, the graphs are (unsurprisingly) not snarks.
\subsection{Graphs of Type $G_{\beta}(m;a,b,c)$}

\begin{theorem}\label{lem:G34no9BETA}The graph $G_{\beta} (m; 1,1,1)$ is always 3-edge-colorable.
\end{theorem}

\newcommand{\flowB}[1]{\text{Flow}_{\beta}(#1)}

\begin{proof}
It suffices to provide a valid color string for any $m$. Figure \ref{fig:G34no9-Beta} provides examples of color strings and corresponding 3-edge-colorings for $m = 3,4,5$. Define $\flowB{\mc{C}}$ to be the assignment of colors to the input and output edges where successive clusters are connected using the $\beta$-connecting edges $A'_{i}B_{i+1}$ and $B'_{i}A_{i+1}$. Observe that $\chi_{\mc{C}}(G_{\beta}(m;1,1,1)) = 3$ if and only if $\flowB{C}$ contains $[(x_{1}, x_{2}, x_{3}), (x_{2}, x_{1}, x_{3})]$. 

It is straightforward to show that $\flowB{123}$ contains $[(1,2,3), (2,1,3)]$ by using the {\sf AltTop} color pattern three times (on each of the colors 1, 2, 3 assigned to consecutive spokes).

For any even $k$, the color string $\underbrace{11\cdots1}_{k}$ contains $[(1,2,3), (2,1,3)]$ in $\flowB{11\cdots1}$ by applying, alternately, the {\sf AltTop} and {\sf AltBot} patterns.

Finally, using the trivial fact that if $m$ is odd, $m-3$ is even, to color $G_{\beta} (m; 1,1,1)$, if $m$ is odd, use the color string $\underbrace{11\cdots1}_{m-3}132$ while if $m$ is even,  use the color string $11\cdots1$ (and use the color pattern sequences described  above).
\end{proof}

\begin{figure}[htbp]
\begin{center}
\ffigbox{
\begin{subfloatrow}[3]

\ffigbox{\caption{$G_{\beta}(3; 1,1,1)$; $\mc{C} = 132$. }\label{}}{
\begin{tikzpicture}[scale=.8]
\drawGThirtyFourNoNineBlobs{3}{.3}{.5}{.8}
\drawSpoke{0}{red, ultra thick}
\drawSpoke{1}{green, ultra thick}
\drawSpoke{2}{blue, ultra thick}
\betaPair{3}{0}{red}{green}
\betaPair{3}{1}{green}{blue}
\betaPair{3}{2}{blue}{red}

\colorNineAlt{0}{red}{green}{blue}
\colorNineAlt{1}{green}{blue}{red}
\colorNineAlt{2}{blue}{red}{green}

\drawLoop{3}{0}{blue, thick}{1}
\drawLoop{3}{1}{red, thick}{1}
\drawLoop{3}{2}{green, thick}{1}
\end{tikzpicture}

}

\ffigbox{\caption{$G_{\beta}(4;1,1,1)$; $\mc{C} = 1111$.}\label{}}{
\begin{tikzpicture}[scale=.8]
\drawGThirtyFourNoNineBlobs{4}{.3}{.5}{.8}
\drawSpoke{0}{red, ultra thick}
\drawSpoke{1}{red, ultra thick}
\drawSpoke{2}{red, ultra thick}
\drawSpoke{3}{red, ultra thick}

\betaPair[30]{4}{0}{red}{green}
\betaPair[30]{4}{1}{blue}{red}
\betaPair[30]{4}{2}{red}{green}
\betaPair[30]{4}{3}{blue}{red}

\foreach \i in {0,2}{\drawLoop{4}{\i}{blue, thick}{1}}
\foreach \i in {1,3}{\drawLoop{4}{\i}{green, thick}{1}}

\colorNineAlt{0}{red}{green}{blue}
\colorNineAltBot{1}{red}{blue}{green}
\colorNineAlt{2}{red}{green}{blue}
\colorNineAltBot{3}{red}{blue}{green}

\end{tikzpicture}
}

\ffigbox{\caption{$G_{\beta}(5; 1,1,1)$; $\mc{C} = 11132$.}\label{}}{

\begin{tikzpicture}[scale=.8]
\drawGThirtyFourNoNineBlobs[.8]{5}{.4}{.5}{.9}
\drawSpoke{0}{red, ultra thick}
\drawSpoke{1}{green, ultra thick}
\drawSpoke{2}{blue, ultra thick}
\betaPair[30]{5}{0}{red}{green}
\betaPair[30]{5}{1}{green}{blue}
\betaPair[30]{5}{2}{blue}{red}

\colorNineAlt{0}{red}{green}{blue}
\colorNineAlt{1}{green}{blue}{red}
\colorNineAlt{2}{blue}{red}{green}

\drawLoop{5}{0}{blue, thick}{1}
\drawLoop{5}{1}{red, thick}{1}
\drawLoop{5}{2}{green, thick}{1}
\drawLoop{5}{3}{blue, thick}{1}
\drawLoop{5}{4}{green, thick}{1}

\foreach \i in {3,4}{\drawSpoke{\i}{red, ultra thick}}
\betaPair[30]{5}{3}{red}{green}
\betaPair[30]{5}{4}{blue}{red}
\colorNineAlt{3}{red}{green}{blue}
\colorNineAltBot{4}{red}{green}{blue}

\end{tikzpicture}
}

\end{subfloatrow}

}{
\caption{$G_{\beta}(m; 1,1,1)$ is always 3-edge-colorable}
\label{fig:G34no9-Beta}}
\end{center}
\end{figure}
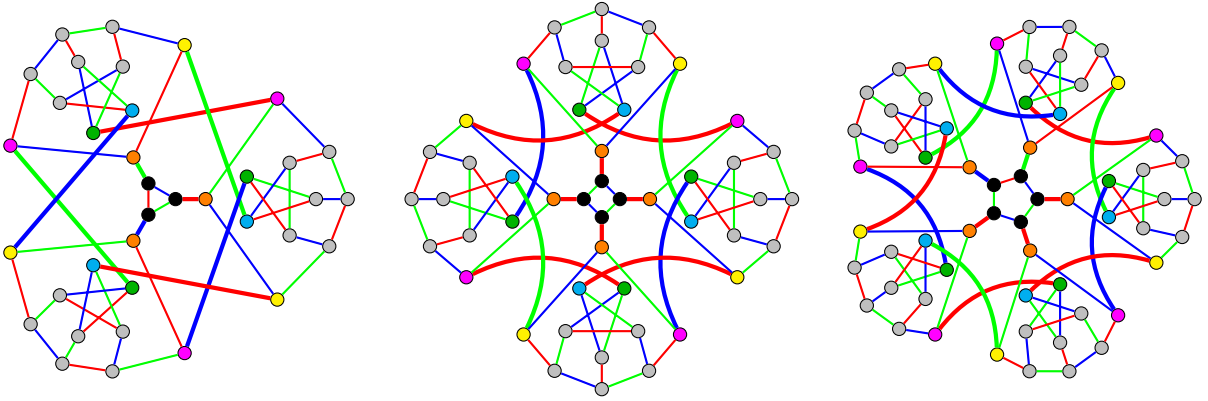

\begin{conjecture}
The graph $G_{\beta}(m; a,b,c)$ is never a snark.
\end{conjecture}


\subsection{Graphs of Type $G_{\alpha}(m;a,b,c)$ that are mostly not snarks}

\begin{theorem} If $m$ is even, and $m/\gcd(m,c)$ is even, then the graph $G_{\alpha}(m; 1,1,c)$ is 3-edge-colorable.\end{theorem}

\begin{proof} The collection of loop edges forms $\gcd(c,m)$ disjoint cycles of length $m/\gcd(m,c)$; by hypothesis, these are all even cycles.  Assign the colors 2 and 3 alternately to each of the the loop cycles, which forces all spoke edges to be colored 1. Then color the pair of connecting edges $A'_{i}A_{i+1}$ and $B'_{i}B_{i+1}$ both 1 if  $i$ is even and color them $2$ and $3$ respectively when $i$ is odd.
\end{proof}

\begin{corollary} If $m$ is even,  $m/\gcd(m,c)$ is even, and  $m/\gcd(m,a)$ is even, then the graph $G_{\alpha}(m; a,a,c)$ is 3-edge-colorable.
\end{corollary}

\begin{proof} In this case, the loop edges are partitioned into $\gcd(m,c)$ disjoint cycles of even length $m/\gcd(m,c)$, and the repeated parameter $a$ induces $\gcd(a,m)$ disjoint cluster-cycles, each of even length $m/\gcd(a,m)$. Again, color all the spokes 1, color the loop edges of each loop cycle alternately 2 and 3, and color the pairs of connecting edges in each cluster-cycle alternately with $(1,1)$ and $(2,3)$. (See Figure \ref{fig:3-color-ex-Even}.)
\end{proof}

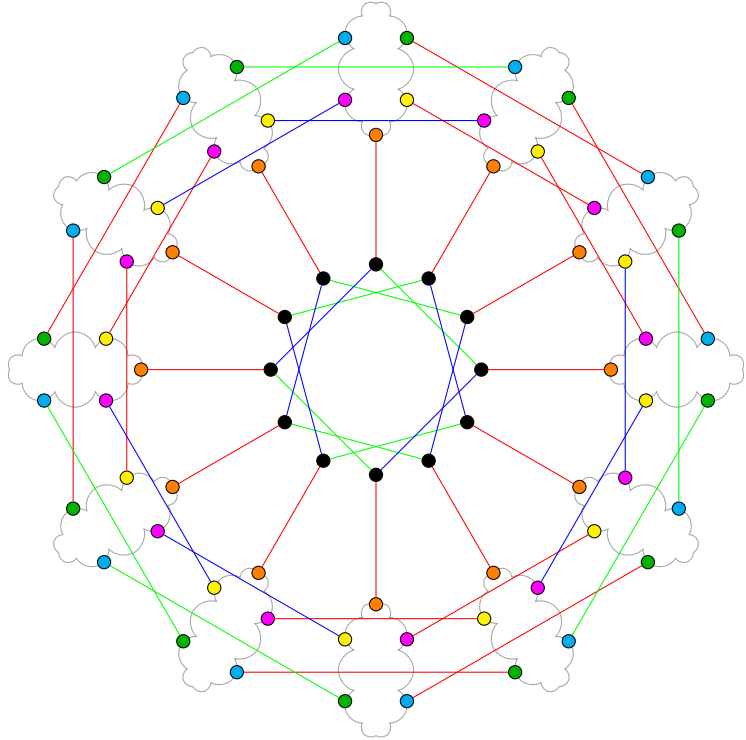
\begin{figure}[htbp]
\begin{center}
\begin{tikzpicture}[scale=2]
\drawBlobs[1.8]{12}{.7}{1}
\foreach \i in {0,...,11}{
\drawSpoke{\i}{red}
}

\foreach \i in {0,4, 8, 1, 5, 9}{\drawAlphaB[.5]{12}{\i}{2}{red}{90}}
\foreach \i in {0, 4, 8, 1, 5, 9}{\drawAlphaA[.5]{12}{\i}{2}{red}{90}}

\foreach \i in {2,6,10}{\drawAlphaA[.5]{12}{\i}{2}{blue}{0}}
\foreach \i in {2,6,10}{\drawAlphaB[.5]{12}{\i}{2}{green}{0}}

\foreach \i in {3,7,11}{\drawAlphaB[.5]{12}{\i}{2}{green}{0}}
\foreach \i in {3,7,11}{\drawAlphaA[.5]{12}{\i}{2}{blue}{0}}

\foreach \i in {0,6, 1,7, 2,8}{\drawLoop{12}{\i}{green}{3}}
\foreach \i in {3,9,4,10, 5,11}{\drawLoop{12}{\i}{blue}{3}}
\end{tikzpicture}
\caption{$G_{\alpha}(12;2,2,3)$ is 3-edge-colorable, since $12/\gcd(12,3)$ and $12/\gcd(12,2)$ are both even.}
\label{fig:3-color-ex-Even}
\end{center}
\end{figure}

\begin{theorem}
 If $m$ is odd, $\gcd(c,m) = 1$, and $c >1$ then the graph $G_{\alpha}(m; 1,1,c)$ is 3-edge-colorable.
\end{theorem}

\begin{proof} Without loss of generality, we may assume that  $1<c<\lfloor\frac{m}{2}\rfloor$ (otherwise, replace $c$ with $m-c$). It suffices to provide a color string $\mc{C}$ and a 3-edge-coloring associated with that color string. Since $\gcd(m, c) = 1$, the loop edges in $G_{\alpha}(m; 1,1,c)$ form a single $m$-cycle, and since $c \neq 1$, the loop edges form a $m/c$ star polygon. We will use a color string $\mc{C}$ with $m-2$ edges colored 1, one edge colored 2 and one edge colored 3. Lemma \ref{lemxyz} in the situation where the loop edges are of the form $w_{i}w_{i+c}$ (as opposed to $w_{i}w_{i+1}$) can be generalized to show that in the color string, the pattern $c_{i-c}\cdots c_{i}\cdots c_{i+c} = x \cdots y \cdots x$ is forbidden (with index arithmetic mod $m$, of course). That is, for each entry $c_{i}$ in $\mc{C}$, the entries $c$ places forward and backwards (taken cyclically) cannot be equal.

We claim that (1) the color string 
\[\mc{C} =  2 \underbrace{1 \cdots 1}_{c-1} 3 \underbrace{1 \cdots 1}_{m-c-1} \]
is a valid color string and (2) using this color string, it is possible to assign 3 colors to the connecting edges without contradiction, using the coloring patterns shown in Figure \ref{fig:G34No9-patterns}.

To see that the string is valid, observe that for each $C_{i}$ in $\mc{C}$, the colors $C_{i-c}$ and $C_{i+c}$ are not equal to each other but different from $C_{i}$, by construction of the string, so we have no collision among the loop edge colors. To color the corresponding graph, observe that if $c$ is odd, then $c-1$ is even, while if $c$ is even, $m-2-(c-1) = m-1-c$ is even. Cyclically shift the color string (and swap 2 and 3 if necessary) so that it is of the form
\[ \mc{C} = 2 \underbrace{1 \cdots 1}_{\text{odd}}3 \underbrace{1 \cdots 1}_{\text{even}} \]

Color the cluster $G_{0}$, which is to have spoke color 3, by the {\sf AltBot} pattern,  choosing  input pair (3,2) and output pair (1,2). Then color the next odd sequence of clusters, which all have spoke color 1, using  {\sf AltTop}. This produces an output pair colored (1,3) at the end of the sequence. Color the next cluster, with spoke color 2, using the {\sf AltBot} pattern, so that the the output pair is colored (2,3) and the input pair is colored (1,3), which matches up with the output pair. Then color the clusters corresponding to the even substring of 1s alternately with {\sf TwoLeft} and {\sf TwoRight}, beginning  with {\sf TwoLeft} and ending with {\sf TwoRight}, and choose (3,2) as the output colors from the last {\sf TwoRight} to match up with the input colors from the first cluster.

An example of a schematic 3-edge-coloring of $G_{\alpha}(9; 1,1,4)$ using this method is shown in Figure \ref{fig:3-color-ex}.
\end{proof}

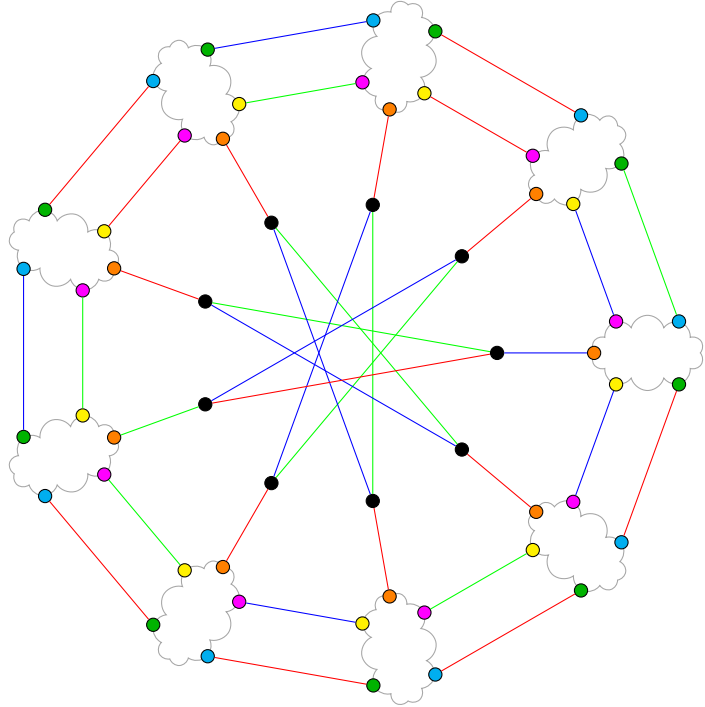
\begin{figure}[htbp]
\begin{center}
\begin{tikzpicture}[scale=2]
\drawBlobs[1.5]{9}{1}{1}
\foreach \i in {1,2,3,4,6,7,8}{
\drawSpoke{\i}{red}
}
\drawSpoke{0}{blue}
\drawSpoke{5}{green}

\foreach \i in {1,3,5,6,7,8}{\drawAlphaB[.5]{9}{\i}{1}{red}{0}}
\foreach \i in {1,3}{\drawAlphaA[.5]{9}{\i}{1}{red}{0}}

\foreach \i in {0,6,8}{\drawAlphaA[.5]{9}{\i}{1}{blue}{0}}
\foreach \i in {2,4}{\drawAlphaB[.5]{9}{\i}{1}{blue}{0}}

\foreach \i in {0}{\drawAlphaB[.5]{9}{\i}{1}{green}{0}}
\foreach \i in {2,4,5,7}{\drawAlphaA[.5]{9}{\i}{1}{green}{0}}

\foreach \i in {0, 8,7,6}{\drawLoop{9}{\i}{green}{4}}
\foreach \i in {4,3,2,1}{\drawLoop{9}{\i}{blue}{4}}
\foreach \i in {5}{\drawLoop{9}{\i}{red}{4}}
\end{tikzpicture}
\caption{The graph $G_{\alpha}(9; 1,1,4)$ is 3-edge-colorable, using color string $\mc{C}=211131111$. (In this example, red=1, blue=2, green=3.)}
\label{fig:3-color-ex}
\end{center}
\end{figure}

If  $G_{\alpha}(m; 1,1,c)$ has $\gcd(c, m) \neq 1$, things are more complicated and less well-understood, since the loop edges form $\gcd(c,m)$ disjoint cycles. Each collection of spoke edges associated with one of the loop cycles itself forms a cut set subject to the Parity Lemma, 
 so there must be at least $\gcd(c,m)$ spokes colored 2 and $\gcd(c,m)$ spokes colored 3, and hence allowable color strings $\mc{C}$ are more complicated.

We conjecture:
\begin{conjecture}
The graph $G_{\alpha}(m; 1,1,c)$ with $\gcd(c,m)>1$, is 3-edge-colorable.
\end{conjecture}

\begin{prop}The graph $G_{\alpha}(3k; k, k, 1)$ is a cyclically 3-connected snark. \end{prop}

\begin{proof} The $a = k$ parameter induces disjoint 3-cluster-cycles. Each cluster-cycle is connected by 3 spokes to the inner $3k$-cycle, so the graph is only cyclically 3-connected. Each of those spoke triples must either be colored $xxx$ or $xyz$, and neither of those color strings (on just the three spokes connecting the triangle cluster-cycle) leads to a viable coloring of the connecting edges of the cluster triangle.
\end{proof}

\begin{conjecture}
If $a \neq b$, then $G_{\alpha}(m; a,b, c)$ is 3-edge-colorable.
\end{conjecture}

\begin{question}
Suppose $m/\gcd(a,m)$ is odd. When is $G_{\alpha}(m; a,a, 1)$ a snark?\end{question}

We have computationally verified that  $G_{\alpha}(10; 2,2,1)$ is not a snark. The next example in which the cluster-cycles are not triangles is $G_{\alpha}(15; 3,3,1)$. Unfortunately, this graph on 180 vertices is too big for Sage to analyze the chromatic index. Personal communication with Martin \v{S}koviera suggests that this graph is also not a snark.

\subsection{Other directions}
We have some preliminary analysis of some of the other snarks on 34 vertices, and this analysis suggests that there may be interesting infinite classes induced by these examples. For instance, consider  graph G34no1, which can be interpreted (using a particular voltage graph) as $\text{G34no1}_{\alpha}(3;1,1,1)$. Unlike the case analyzed in the previous section, in this case, $\text{G34no1}_{\alpha}(5;1,1,1)$ is 3-edge-colorable, but $\text{G34no1}_{\alpha}(5;1,2,2)$ and $\text{G34no1}_{\alpha}(5;1,4,2)$ are snarks.

The House of Graphs \cite{HoG} separates its snarks by cyclic edge-connectivity. Among the 19 snarks on 34 vertices with 3-fold rotational symmetry, those that appear in the list of snarks with cyclic edge-connectivity at least 5 are G34no5, which has 3 arrows, two forming a pair of parallel edges, and graph G34no15, which has four arrows, forming two pairs of parallel edges. We conjecture that these snarks can be generalized into broader families of snarks with rotational symmetry that have cyclic edge-connectivity at least 5.

\subsection*{Acknowledgements}The first and third author appreciate the hospitality of the Institute of Mathematics at UNAM--Juriquilla, and CINNMA.A.C. Juriquilla, Quer\'etaro, Mexico, during their sabbatical visit. The three authors were partially supported by grant PAPIIT 104915. Leah Berman's research was supported by the Simons Foundation (Grant \# 209161 to L. Berman) and D\'eborah Oliveros's research was supported by (Grant CONACyT 166306) as well. The authors appreciate several very helpful conversations with Martin \v{S}koviera, beginning at the ACCOTA 2016 conference.

\bibliographystyle{spmpsci} 
\bibliography{PseudoLoupekine}

\begin{appendix}
\section{Graph6 representations of snarks with 3-fold rotation, $n = 28$}\label{appendix:28}
\begin{enumerate}[{G28no}1:]
\item\verb=[??G?EO?G?GB_AO_g_?CP_??C??O????[O?CA?AH??C?G?@GO??C???cA??????J=
\item\verb=[??A@E?OKCCAP??B`??_@_??O?G`???@OC?C???wO???G?@I?????C@?_?@????F=
\item\verb=[?GI?AO?KCCBa?OC_??A@O??@?GE???@OO?C???w??_?_?@G??A@?G?A???_???F=

\end{enumerate}
\section{Graph6 representations of snarks with 3-fold rotation, $n=34$}\label{appendix:34}
\begin{minipage}{\linewidth}
\begin{enumerate}[{G34no}1:]

\begin{footnotesize}
\item\verb=a?`Q?cOG???@_??@g???T?a??@g@??G?C???I??K??@???HA??@???g_G???@??@??????S??H????O????c?O??@????IG=
\item\verb=a?GQ?e??GC?B_?O@g_??T????OH???`????????{??????`??AA???__?@C????J?????K???H????O??@?CG???@?O??AG=
\item\verb=a?hY@eOGG??B_??@_???O?a???o?????P??A???W??G???X??G_???E_?????@?@????C??O?G?A??????@K?A??_????Ag=
\item\verb=a?gW@eOGG?GA_??_g_?????C?A?C???O???I??@W??W???XO?O??AC?_?_????_??A?????_?????@O????k????o????BG=
\item\verb=a??Y@EOOGCCAP???????D???O?GO@??G?G????O[??O??A?_?????O@_??_??Q?@G??G???B?GA??@??OO?C??_A????B?G=
\item\verb=a?hY@eOGGC?A_??@g???@?AC?@??C???C?????@W?@????T??CG???Q_?_??????????G??C?W?O??????@c??G?O????Ag=
\item\verb=a?G?@eO?KCCA????_???X_??C??OC???[G?C??aIC????G@O?O??C??g???????J@??????_?GG@??????@GO???G?????w=
\item\verb=a?HI?_?OGCCAC?????@A@CG??DCO????CCA???@??Q????D_?G??_??_C?????A`_???G??H?GA??A??P??C?B??????D?G=
\item\verb=a??I@E??GC?B_?O@g???T_??O???_?_?C?_@??BHO?????@???_?O?@c???O_??H?????C?_?GO???O??G?CC???@??A?AG=
\item\verb=a?hW@eOGG?GB???_g?G??????AG????@c?GA??EG?O_???p_????_??_???O??O?????G????W???AG???@_?A??G?????w=
\item\verb=a??Y??O???CAa?O?g_??G?G??OK??O@_??@C???x????W?@GO??A???aG??????R_A?????O?G?O_???@??CAG???????@W=
\item\verb=a?gY@eOGGC?A_???g_??D??O?GI?????C??@??AG?O@???P?@?O???D_?_?????@???????O?W??G?????@S?@??G????Cg=
\item\verb=a?@G@c??G?GA_A?_g??CP?B??@k@?????O??@??W_?A?C?@???o???Hp???@??????????O??X????O???@CG???G????CW=
\item\verb=a?@?@c??G?GA_A?_g??CP?B??@k@?????O??@??W_?A?C?@???o???Hp???@?_??B????????H????????@KG???A????DG=
\item\verb=a??Q@eOOGCC?@???@???G?A??A??@_??KG????A[?@??@?@G??G?CC?G??_?W???_?A???@?CG?I??????EC?B???????KG=
\item\verb=a?gY@eOGGC?B???@_???O?????c??Q?A@?????cG??_???_@?G????B_?C@???@H??B????A???C???????[??G?A???O@G=
\item\verb=a?_W@e?GG??A_?O_?_??@????IGC?O??c?K???CI?G???G???G????A_???O?C?HO???G???DHO????`???C?G@????AC?G=
\item\verb=a??I@AOGGC?B_??@g_??Da????G???G?CO?C???g??C???@G???C??@_W??????PO?O???GO?GO??_???I?C@??C?????cG=
\item\verb=a??O?E?G??GB_AO_g_?CPA??@???????[O???O?g??C???d???A?C??C???O@??BI???????BK??O??W????O_????????w=

\end{footnotesize}

\end{enumerate}

\end{minipage}

\section{Graph6 representations of snarks with 3-fold rotation, $n = 36$}\label{appendix:36}
\begin{enumerate}[{G36no}1:]
\begin{tiny}
\item\verb=c?`O?cAG??GOG?C?_?A?@?G??Gc?@?OC?K?????k@?????J??O???@?_@??O??O@?B??????@X??????__????G???????wA???C?A???D=
\item\verb=c?GA?a??G???@??A`??C?_A???{?@??Oc?A?A??A??@?_?BD??????K_?????I?@G???G?S??H?G???A??@D?O???CG????E?????????F=
\item\verb=c?@Y?a?GGC?B_AO?_W??JI??O??O??@?G??G???_??C???G???K???Bo??@????h???????C?G????O??@?C@???G????@P????C?????F=
\item\verb=c?@W@e??G?GPI?CA????@??G?BC?OA?_?_????@?O?????F??CG??AC?O?????E@???A????BG@??@??@??Q??G???CC???????E?????F=
\item\verb=c?gI?C?OGC?A???Bo@?A?C?_??@?G?a??O_???GAG?????@?G????@C_?O????GH??D????A?K???C???@_C_?O?????O@GG??@???@?_@=
\item\verb=c?hG?aO?G?CAD??B_?????????y?????T??G?_?_?_A???F??G??A??_O??G?P??_???_????w??_??A???CA???@?OC??@????C?????F=
\item\verb=c?HY@?OGG??B_??@g???T?a??@gO??A?A?_???????C???J_??????c_??E????T?????A????_???O????K@????????BG??A?O?????L=
\item\verb=c?@A?eOOGCCA@???_@????????o_?OO?C@?O?@OCO???@??_?????O@__?????W@I???????_H??_?????_K?AG?????OGG?@?C????@_@=
\item\verb=c??W?EOG??GB_A?_g_?C@a??@???????[O???O?G??C???D???A?C??????O@??BI???????AI?A????O_?C_???_????PG??GG????Q?@=
\item\verb=c??I@aO????@_A??a_??aO??_CG?A?C?D????A?WE?????TA?????A?c???OG??@a???????@W?_????@??C?_??A??C??G?g????????J=

\end{tiny}

\end{enumerate}

\end{appendix}

\end{document}